\documentclass[12pt]{amsart} 
\usepackage{verbatim, latexsym, amssymb, amsmath,color}
\usepackage[french,english]{babel}
\usepackage{enumitem}
\usepackage{epsfig}

\usepackage{geometry}
\usepackage{hyperref}

\DeclareMathOperator{\Diff}{Diff}

\DeclareMathOperator{\interior}{Int}
\DeclareMathOperator{\spt}{spt}

\DeclareMathOperator{\Vol}{Vol}
\DeclareMathOperator{\Area}{Area}
\DeclareMathOperator{\divergence}{div}

\newtheorem{theo}{Theorem}[section]
\newtheorem{prop}[theo]{Proposition}
\newtheorem{lemme}[theo]{Lemma}
\newtheorem{definition}[theo]{Definition}
\newtheorem{coro}[theo]{Corollary}
\newtheorem{claim}[theo]{Claim}

\theoremstyle{definition}
\newtheorem{remarque}[theo]{Remark}
\newtheorem{exemple}[theo]{Example}

\begin{document}
\title[Minimal hypersurfaces in manifolds thick at infinity]
{A dichotomy for minimal hypersurfaces in manifolds thick at infinity}

\author[Antoine Song]{Antoine Song}
\address{Department of Mathematics, University of California, Berkeley, Berkeley, CA 94720, USA}
\email{aysong@berkeley.edu}
\thanks{The author was partially supported by NSF-DMS-1509027.}

\maketitle

\selectlanguage{english} 
\begin{abstract}
 Let $(M,g)$ be a complete $(n+1)$-dimensional Riemannian manifold with $2\leq n\leq 6$. Our main theorem generalizes the solution of S.-T. Yau's conjecture on the abundance of minimal surfaces and builds on a result of M. Gromov. Suppose that $(M,g)$ has bounded geometry, or more generally is thick at infinity.
%i.e. any connected finite volume complete minimal hypersurface is closed.
%it contains no non-compact finite volume connected complete minimal hypersurface. 
Then the following dichotomy holds for the space of closed embedded hypersurfaces in $(M,g)$: 
either there are infinitely many saddle points of the $n$-volume functional, or there is none.

Additionally, we give a new short proof of the existence of a finite volume minimal hypersurface in finite volume manifolds, we check Yau's conjecture for finite volume hyperbolic 3-manifolds and we extend the density result due to Irie-Marques-Neves when $(M,g)$ is shrinking to zero at infinity.
\end{abstract}

\section{Introduction}

The search for minimal hypersurfaces in compact manifolds has enjoyed significant progress recently, thanks to the development of various min-max methods, such as the systematic extension of Almgren-Pitts' min-max theory \cite{P} led by Marques and Neves \cite{MaNeWillmore, MaNeinfinity,MaNeindexbound,MaNemultiplicityone}, the Allen-Cahn approach \cite{Guaraco,GasGua,ChodMant}, or others \cite{SacksUhlenbeck, Jost, C&Mextinction, Zhou2, Riviere, PigatiRiviere}. One central motivation was the following conjecture of S.-T. Yau: \newline

\textbf{Yau's conjecture } \cite{Yauproblemsection}: In any closed three-dimensional manifold, there are infinitely many minimal surfaces.
\newline

Strong results implying the conjecture were obtained for generic metrics by Irie-Marques-Neves \cite{IrieMaNe} (see \cite{MaNeSong} for a quantified version), Chodosh-Mantoulidis \cite{ChodMant}, X. Zhou \cite{Zhoumultiplicityone}, Y. Li \cite{yyLigeneric}. Concurrently to these results, the conjecture for non-generic metrics was treated with a different line of arguments. When the manifold satisfies the ``Frankel property'', it was solved by Marques-Neves in \cite{MaNeinfinity}. We recently settled the general case in \cite{AntoineYau}, where we localized min-max constructions appearing in \cite{MaNeinfinity} to some compact manifold with stable minimal boundary by introducing a non-compact manifold with cylindrical ends.

On the other hand, results about minimal hypersurfaces in complete non-compact manifolds are comparatively few and far between, and most of them are existence results. We give here a non-exhaustive list. In \cite{ColHauMazRos,ColHauMazRoserratum}, Collin-Hauswirth-Mazet-Rosenberg constructed a closed embedded minimal surface in any finite volume hyperbolic $3$-manifold; there is also the work of Z. Huang and B. Wang \cite{HuangWanghyperbolic}, and of Coskunuzer \cite{Coskunuzer}. In \cite{Montezuma1}, Montezuma showed that a strictly mean concave compact domain in a complete manifold intersects a finite volume embedded minimal hypersurface. In \cite{GromovPlateauStein}, Gromov proved the following existence theorem, which we interpret as the analogue of Almgren-Pitts existence result \cite{P} for non-compact manifolds:\newline

\textbf{Gromov's result }\cite{GromovPlateauStein}: In a complete non-compact manifold $M$, either there is an embedded finite volume complete minimal hypersurface, or there is a possibly singular strictly mean convex foliation of any compact domain of $M$.
\newline

In \cite{ChamLiok}, Chambers and Liokumovich showed the existence of a finite volume embedded minimal hypersurface in finite volume complete manifolds; in fact they proved the existence of such a minimal hypersurface if there is a region whose boundary is, say, ten times smaller than its width. In asymptotically flat $3$-manifolds, Chodosh and Ketover constructed minimal planes in \cite{ChodKet}, using a degree argument (see Mazet-Rosenberg \cite{MazetRosenbergplanes} for generalizations).

The goal of this paper is to propose a relevant generalization of the solution of Yau's conjecture to non-compact manifolds, by building on Gromov's result. Motivations came from our solution of the conjecture when the Frankel property is not satisfied \cite{AntoineYau}, where we perform min-max in a non-compact manifold with cylindrical ends. Moreover some classes of manifolds naturally contain non-compact manifolds, for instance finite volume hyperbolic $3$-manifolds. The non-compact situation substantially differs from the compact case: there are many non-compact manifolds without any closed (or finite volume) minimal hypersurfaces. For any integer $m>0$, it is easy to construct a metric on $\mathbb{S}^2\times\mathbb{R}$ for instance, with exactly $m$ closed minimal surfaces. That metric can look like a long tube which gets thinner around $ \mathbb{S}^2\times\{0\}$, and the minimal surfaces are $\mathbb{S}^2\times\{0\}$ and some other slices $\mathbb{S}^2\times\{t\}$ which are degenerate stable. At first sight, it seems hard to come up with essentially different examples of manifolds that would contain only finitely many closed minimal hypersurfaces. In our main result, we confirm this intuition for manifolds called ``thick at infinity'', which we define below.

\subsection*{The class $\mathcal{T}_\infty$ of manifolds thick at infinity} 
Before stating our main theorem, we recall the notion of ``thickness at infinity'' introduced by Gromov \cite{GromovPlateauStein}. 
%In the following the dimension $n+1$ satisfies $2\leq n \leq 6$. 
Minimal hypersurfaces in this paper are all \emph{embedded} and unless mentioned, we consider hypersurfaces without boundary.
\begin{definition}
Let $(X^{n+1},g)$ be a complete $(n+1)$-dimensional Riemannian manifold. $(X,g)$ is said to be \textit{thick at infinity} (in the weak sense) if any connected finite volume complete minimal hypersurface in $(X,g)$ is closed.
We denote by $\mathcal{T}_\infty$ the class of manifolds that are thick at infinity.
\end{definition}
In \cite{GromovPlateauStein}, Gromov actually uses a slightly stronger notion of thickness at infinity, since he asks that any connected finite volume minimal hypersurface with maybe non-empty compact boundary is compact.

The property of ``thickness at infinity'' is checkable. Complete manifold with \textit{bounded geometry} are important examples of manifolds thick at infinity. Some other examples are given in \cite[Section 1.3]{GromovPlateauStein}, and the condition $\star_k$ in \cite{Montezuma1} also implies thickness at infinity. Easy special cases of the previous conditions include coverings of closed manifolds and asymptotically flat manifolds. 

Note that $M$ can be thick at infinity and at the same time ``thin'' in a certain sense. Indeed, using the monotonicity formula, it is easy to construct a warped product metric $g_t \oplus dt^2$ on a cylinder $N^n\times \mathbb{R}$ (where $N$ is any closed $n$-dimensional manifold) such that $(N^n\times \mathbb{R},g_t \oplus dt^2)$ is thick at infinity, has finite volume, and the $n$-volume of the cross section $N\times \{t\}$ decreases to zero as $t\to \pm\infty$.

\subsection*{A zero-infinity dichotomy for manifolds thick at infinity} 

Almost by definition, closed minimal hypersurfaces are critical points of the $n$-volume functional. By the properties of the Jacobi operator, which encodes the second variation of the $n$-volume at a minimal hypersurface, the space of deformations that do not increase the area at second order is finite dimensional. It is natural to define \textit{saddle points of the $n$-volume functional} (or simply \textit{saddle point minimal hypersurfaces}) as follows. Consider a connected closed embedded minimal hypersurface $\Gamma$. If it is $2$-sided then we call it a saddle point minimal hypersurface if there is a smooth family of hypersurfaces $\{\Gamma_t\}_{t\in(-\varepsilon,\varepsilon)}$ ($\varepsilon>0$) which are small graphical perturbations of $\Gamma=\Gamma_0$ so that $\{\Gamma_t\}_{t\in(-\varepsilon,0)}$ and $\{\Gamma_t\}_{t\in(0,\varepsilon)}$ are on different sides of $\Gamma$ and distinct from $\Gamma$, and 
$$\Vol_n(\Gamma) = \max_{t\in(-\varepsilon,\varepsilon)} \Vol_n(\Gamma_t).$$
If $\Gamma$ is $1$-sided, we call it a saddle point minimal hypersurface if its connected double cover is a saddle point minimal hypersurface in a double cover of the ambient manifold.
Note that if the metric is bumpy (i.e. no closed minimal hypersurface has a non-trivial Jacobi field), then saddle point minimal hypersurfaces are exactly unstable 2-sided closed minimal hypersurfaces and 1-sided closed minimal hypersurfaces with unstable double cover. By ``compact domain'', we mean a compact $(n+1)$-dimensional submanifold of $M$ with smooth boundary.

Our main theorem is a dichotomy for the space of closed hypersurfaces embedded in a manifold thick at infinity. It says that either this space has infinite complexity from a Morse theoretic point of view, or its structure is locally simple.

\begin{theo}  \label{mainthm}
Let $(M,g)$ be an $(n+1)$-dimensional complete manifold with $2\leq n\leq 6$, thick at infinity. Then the following dichotomy holds true:
\begin{enumerate}
\item either $(M,g)$ contains infinitely many saddle point minimal hypersurfaces,
\item or there is none; in that case for any compact domain $B$, there is an embedded closed area minimizing hypersurface $\Sigma_B$ such that $B\backslash \Sigma_B$ has a singular weakly mean convex foliation.

\end{enumerate}
\end{theo}

We make the following comments, which will be developed in Sections \ref{saddl}, \ref{sectiondichotomy} and \ref{hyperbolicyau}:

\begin{itemize}
\item In the second case, the minimal hypersurface $\Sigma_B \subset (M,g)$ is globally area minimizing in its $\mathbb{Z}_2$-homology class and may be empty or disconnected. The foliation comes from the mean curvature flow so it has the corresponding regularity \cite{WhiteregMCF}, and it is shrinking towards $\Sigma_B$. 
%Saying that the foliation is weakly mean convex means here that the non strictly mean convex leaves are smooth minimal inside $B$. 
\item This theorem still holds if $M$ has minimal boundary and if each component of $\partial M$ is compact. Closed manifolds are trivially in $\mathcal{T}_\infty$. We will see in Section \ref{saddl} that $1$-parameter min-max produces a saddle point minimal hypersurface, thus Theorem \ref{mainthm} implies the existence of infinitely many saddle point minimal hypersurfaces in closed manifolds. This fact does not follow from the solution of Yau's conjecture \cite{MaNeinfinity}\cite{AntoineYau}. 
\item  Finite volume hyperbolic $3$-manifolds do not belong to $\mathcal{T}_\infty$ in general. Nevertheless we will show in Section \ref{hyperbolicyau} that they satisfy Yau's conjecture since they contain infinitely many saddle point minimal hypersurfaces, extending the existence result of Collin-Hauswirth-Mazet-Rosenberg \cite{ColHauMazRos,ColHauMazRoserratum}. 
\item The situation for geodesics in surfaces is different: some $2$-spheres contain only three simple closed geodesics, and immersed closed geodesics in hyperbolic surfaces are all strictly stable.
\end{itemize}

In the process of proving the Theorem \ref{mainthm}, we will explain the following local version of Gromov's result for complete manifolds that are not necessarily in $\mathcal{T}_\infty$. 

\begin{theo} \label{existence}
Let $(M,g)$ be an $(n+1)$-dimensional complete manifold with $2\leq n\leq 6$, and let $B$ be a compact domain. Then 
\begin{enumerate}
\item either $M$ contains a complete embedded minimal hypersurface intersecting $B$ and with finite $n$-volume,
\item or $B$ has a singular strictly mean-convex foliation.
\end{enumerate}
\end{theo}
This theorem is essentially already contained in \cite{GromovPlateauStein}, but we find it useful for the reader to present a detailed proof with some new arguments, for instance the use of Marques-Neves lower index bound \cite{MaNeindexbound} and the mean curvature flow. We think that modulo some improvements, our proof should also cover higher dimensions (the minimal hypersurface then may have a codimension at least $7$ singularity set). 

Let us list a few corollaries (see Corollary \ref{explanation} for more details). If $M$ has finite volume, then a complete finite volume embedded minimal hypersurface exists; this result  was first proved by Chambers-Liokumovich \cite{ChamLiok}. Moreover, if there is a compact subset $X\subset M$ whose boundary is mean concave in the sense that the mean curvature vector is pointing outside of $X$, then such a hypersurface also exists and intersects $X$. In particular this recovers a result of Montezuma \cite{Montezuma1}.

\subsection*{A density result \`{a} la Irie-Marques-Neves}

A complete $(n+1)$-dimensional manifold $(M,g)$ is said to have a \textit{thin foliation at infinity} if there is a proper Morse function $f:M\to [0,\infty)$ so that the $n$-volume of the level sets $f^{-1}(t)$ converges to zero as $t$ goes to infinity. The relevant topology on the space of complete metrics on $M$ is the strong (Whitney) $C^\infty$-topology. Let $\mathcal{F}_{\mathrm{thin}}$ be the family of complete metrics on $M$  with a thin foliation at infinity; it is an open subset for the strong topology. Similarly, the intersection $\mathcal{F}_{\mathrm{thin}} \cap \interior(\mathcal{T}_\infty)$ is a non-empty open subset for that topology (here $\interior(\mathcal{T}_\infty)$ denotes the interior of $\mathcal{T}_\infty$ in the space of complete metrics). The following theorem generalizes the density theorem of Irie, Marques and Neves to these metrics:

\begin{theo}
Let $M$ be an $(n+1)$-dimensional manifold with $2\leq n\leq 6$. 
\begin{enumerate}
\item For any metric $g$ in a $C^\infty$-dense subset of $\mathcal{F}_{\mathrm{thin}}$, the union of complete finite volume embedded minimal hypersurfaces in $(M,g)$ is dense.
\item For any metric $g'$ in a $C^\infty$-generic subset of $\mathcal{F}_{\mathrm{thin}} \cap \interior(\mathcal{T}_\infty)$, the union of closed embedded minimal hypersurfaces in $(M,g')$ is dense.
\end{enumerate}
\end{theo}

The proof borrows an idea of Irie, Marques and Neves in \cite{IrieMaNe}, where they use an elegant argument based on the Weyl law for the volume spectrum proved by Liokumovich, Marques and Neves \cite{LioMaNe}.  Many non-compact manifolds of finite volume do not obey the Weyl law, even if all the min-max widths are finite (see Remark \ref{weird} for an informal justification).
%: it is not hard to cook up metrics of finite volume but whose widths behave linearly for instance (see Remark \ref{weird}). 
Thus, we have to find a more robust property of the min-max widths which in fact gives an alternative argument even in the compact case, not based on the Weyl law. On the other hand, the Weyl law seems essential in the quantified result we obtained with Marques and Neves \cite{MaNeSong} about the generic equidistribution of a sequence of minimal hypersurfaces. Another remark is that we cannot prove the result for a $C^\infty$-generic subset of  $\mathcal{F}_{\mathrm{thin}}$, but only for a $C^\infty$-dense subset, because non-compact minimal hypersurfaces may appear and structural results like White's bumpy metric theorems \cite{Whitebumpy,Whitebumpy2} become false.

\subsection*{Organisation}
In Section \ref{locex}, we reprove a local version of Gromov's result \cite{GromovPlateauStein} and derive a few corollaries. After explaining how to construct saddle point minimal hypersurfaces with $1$-parameter min-max in Section \ref{saddl}, we show the zero-infinity dichotomy for manifolds thick at infinity in Section \ref{sectiondichotomy}. We also check that finite volume hyperbolic $3$-manifolds satisfy Yau's conjecture in Section \ref{hyperbolicyau}, and extend in the last section the density result of Irie-Marques-Neves.

\subsection*{Acknowledgement} 
I am grateful to my advisor Fernando Cod\'{a} Marques for his crucial guidance. I thank Yevgeny Liokumovich for explaining \cite{ChamLiok} to me and mentioning \cite{GromovPlateauStein}, \cite{PapaSwen}. I am thankful to Misha Gromov for exchanges about \cite{GromovPlateauStein}. I also want to thank Franco Vargas Pallete for discussing with me Yau's conjecture for finite volume hyperbolic $3$-manifolds, a result of which he was also aware. Moreover, a very careful reading by the referees improved the writing of this article.

\section{Local existence of finite volume minimal hypersurfaces} \label{locex}

%Suppose $X$ is a compact $(n+1)$-submanifold of $M$. We now define the width of $(X,\partial_0 X,\partial_1 X)$ in $M$.

%Let $N$ be a compact $(n+1)$-submanifold of $M$ containing $X$. Suppose that $\partial X = \partial_0 X \cup \partial_1 X$ where $\partial_0 X$ and $\partial_1 X$ are disjoint closed (maybe empty) and let $C_0,C_1$ be the associated cycles in $\mathcal{Z}_n(N,\mathbb{Z}_2)$. An important technical point is that in our definition of sweepout of $X\subset M$ from $\partial_0 X$ to $\partial X_1$, the generalized hypersurfaces are allowed to be not contained in $X$ (see Appendix). We introduced such a definition in \cite[Subsection 1.2]{Antoinespheres}. In the Appendix we define the width of an equivalence class of sweepouts $\Pi \in \pi_1^{\sharp}(\mathcal{Z}_n(N,\mathbb{Z}_2),C_0,C_1)$ by 
%$$\mathbf{L}(\Pi) = \inf\{\mathbf{L}(S) ; S \in \Pi\},$$
%where $\mathbf{L}(S)$ is the width of a sweepout $S$.
%Let us define
%$$W(N) =  \inf\{\mathbf{L}(\Pi)  ; \mathcal{A}(\Pi) = [|X|]\}$$
%where $\mathcal{A}$ is the Almgren map (see Appendix).
%We finally define the width $W$ of $(X,\partial_0 X,\partial_1 X)$ in $M$ as follows. Let $N_1\subset ...\subset N_k\subset ...$ be an exhaustion of $M$ by compact regions with smooth boundaries. Then
%$$W=W(X,\partial_0X,\partial_1X) = \inf_k W(N_k).$$

Let $(M^{n+1},g)$ be a complete possibly non-compact $(n+1)$-dimensional Riemannian manifold. We will use in this section a local version of Almgren-Pitts theory; definitions and some relevant results are stated in Appendix B. We will also need the mean curvature flow, in its level set flow formulation, whose basic properties are recalled in Appendix C. Strict and weak mean convexity in the sense of level set flow are also defined there. By ``compact domain'', we mean a compact $(n+1)$-dimensional submanifold of $M$ with smooth boundary. 

We say that $B$ has a \textit{singular strictly mean convex foliation} if there is a compact Riemannian manifold $(B_0,g')$ with $C^{1,1}$ boundary containing isometrically $(B,g)$ so that there is a family of closed subsets of $B_0$, $\{K_t\}_{t\in[0,1]}$, satisfying:
\begin{itemize}
\item $K_0=B_0$, $K_1\cap B=\varnothing$, $K_{t'}\subset \interior(K_{t})$ if $t'>t$,
\item  each $K_t$ is an integral $n+1$-current and $\{\partial K_t\}$ yields a family of cycles in $\mathcal{Z}_n(B_0;\mathbb{Z}_2)$ continuous in the flat topology (see Appendix B-C), 
\item for all $t$ $\partial K_t$ is strictly mean convex for $g'$ in the sense of mean curvature flow (see Appendix C),
\item for any $t\in(0,1)$, if $\partial K_t$ is smooth then $\{\partial K_t\}_{t\in[t-\varepsilon,t+\varepsilon]}$ is a smooth foliation for some small $\varepsilon>0$, and if $\partial K_t$ is not smooth, then $\{K_t\}_{t\in[t-\varepsilon,t+\varepsilon]}$ is a level set flow for some small $\varepsilon$ (see Appendix C).
\end{itemize}
The last bullet will be automatically satisfied in our constructions and will somewhat clarify the proofs.

The main theorem of this section, which will be useful in subsequent sections, is essentially proved in \cite{GromovPlateauStein}. We give here a detailed proof with new arguments, especially for the min-max part. Another more formal difference is the alternative use of level set flow instead of constructing foliations ``by hand'' (already suggested by B. Kleiner, see \cite{GromovPlateauStein}). Here $M$ can be compact or non-compact.

\begin{theo}[Local version of Gromov's theorem]\label{a}
Let $(M,g)$ be an $(n+1)$-dimensional complete manifold with $2\leq n\leq 6$, and let $B$ be a compact domain. Then 
\begin{enumerate}
\item either $M$ contains a complete embedded minimal hypersurface intersecting $B$ and with finite $n$-volume,
\item or $B$ has a singular strictly mean-convex foliation.
\end{enumerate}
\end{theo}

That theorem immediately implies the following corollaries which were proved in \cite{ChamLiok} and \cite{Montezuma1} respectively (it seems that the relation with \cite{GromovPlateauStein} was not in the literature).
\begin{coro}\label{explanation}
Let $M$ be as in the previous theorem. 
\begin{enumerate}
\item If $M$ has finite volume or more generally if there is an exhaustion $X_1\subset ...\subset X_i\subset...$ of $M$ be compact subsets with smooth boundaries such that 
$$\lim_{i\to \infty}\Vol_n(\partial X_i)=0,$$
then there is a complete finite volume embedded minimal hypersurface.
\item If $M$ contains a compact subset $X$ with mean concave smooth boundary (the mean curvature vector is non-zero pointing outwards), then there a complete finite volume embedded minimal hypersurface intersecting $X$.
\end{enumerate}
\end{coro}

%\begin{remarque}
%\end{remarque}

We will say that a hypersurface $S$ embedded in $(N,g)$ is \textit{locally} (resp. \textit{globally}) \textit{area minimizing} if $S$ is a minimal hypersurface and if any hypersurface isotopic to $S$ in a neighborhood of $S$ in $N$ (resp. any hypersurface in the same $\mathbb{Z}_2$-homology class as $S\subset N$) has $n$-volume at least $\Vol_n(S)$. Before giving the proof of Theorem \ref{a} and Corollary \ref{explanation}, we need the following local min-max theorem for non-bumpy metrics. In our setting, the width $W$ of a compact manifold $(N,g)$ is defined in Appendix B, (\ref{wiidth}). Condition [M] and Type I, II, III stable minimal hypersurfaces are introduced in the discussion of Appendix A 

\begin{prop} \label{nonbumpy}
Let $(N^{n+1},g)$ be a compact manifold with minimal boundary, with $2\leq n\leq 6$, such that $\partial N$ is locally area minimizing inside $N$. Then there is a closed embedded minimal hypersurface $\Gamma$ inside the interior $\interior(N)$ whose index is at most one and whose $n$-volume is bounded by the width $W$ of $(N,g)$. 
\end{prop}

\begin{proof}
Since from \cite{MorganRos} the width $W$ of $(N,g)$ is larger than the $n$-volume of any connected component of $\partial N$, by Lemma \ref{degenerate} in Appendix A and the discussion following it, we can suppose that each component of $\partial N$ satisfies Condition [M] and is either strictly stable or degenerate stable of Type II. Then by Lemma \ref{adeformation} (1) in Appendix A, there is a sequence of metrics $\{h^{(q)}\}$ converging to $g$ such that for each $h^{(q)}$, $\partial N$ is strictly stable and has a neighborhood $\mathcal{N}_q$ foliated by hypersurfaces which are strictly mean convex (except $\partial N$ of course). The thickness of this neighborhood $\mathcal{N}_q$ essentially does not depend on $q$. For the next argument, it can be helpful to think of the compact $(n+1)$-manifold $(N,h^{(q)})$ as a domain isometrically embedded into a closed Riemannian $(n+1)$-manifold $\hat{N}$ endowed with a metric still denoted by $h^{(q)}$.  Consider a sequence of bumpy metrics $g_q$ converging to $g$, so that $g_q$ is very close to $h^{(q)}$. Since $\partial N$ is strictly stable with respect to $h^{(q)}$, the implicit function theorem \cite{Whitebumpy} implies that there is a unique minimal hypersurface for $g_q$ smoothly close to $\partial N$ and by pulling back the metric with a diffeomorphism close to the identity we can make sure that this minimal hypersurface for $g_q$ is in fact $\partial N$. Similarly, by a deformation argument involving the Jacobi operator, $\mathcal{N}_q$ is still foliated by hypersurfaces which are strictly mean convex except for $\partial N$. By the maximum principle, $(\mathcal{N}_q,g_q)$ contains no minimal hypersurface except $\partial N$.
Now we apply the local min-max theorem for bumpy metrics to $(N,g_q)$, see Theorem \ref{discreteminmax} in Appendix B. For each $q$ we have a closed embedded minimal hypersurface $\Gamma_q$ inside the interior $\interior(N)$ with index at most one and $n$-volume bounded by the width of $(N,g_q)$, intersecting $N\backslash \mathcal{N}_q$. Sending $q$ to infinity, $\Gamma_q$ converges (in the varifold sense) subsequently by \cite{Sharp} to a minimal hypersurface embedded in $N$, with index at most one and $n$-volume at most $W$. It cannot be contained in the boundary $\partial N$ by the monotonicity formula and the fact that $\Gamma_q\cap N\backslash \mathcal{N}_q\neq \varnothing$ for all $q$.

\end{proof}

\begin{proof} [Proof of Theorem \ref{a}]

Fix $p\in M$ and a large radius $r>0$. Let $B\subset M$ be a compact domain which we can assume to be connected, and let $D$ be a compact domain strictly containing a geodesic open ball $B_{r}(p)$ which itself contains $B$. We modify the metric in an arbitrarily thin neighborhood of $\partial D$ to make it mean convex (the mean curvature vector points inwards). The new metric is called $g_D$ and coincide with $g$ on $B_{r}(p)$. 

It is enough to show the following claim:
\vspace{1em}

\begin{claim}
 Suppose that $(B,g)$ admits no singular strictly mean-convex foliations. Then there is a closed embedded minimal hypersurface $S$ inside $(D,g_D)$ of Morse index at most one, with $n$-volume bounded by a constant depending only on $(B,g)$, which intersects $B$.
 \end{claim}

\vspace{1em}
Indeed, once this claim is proved, then by sending $r$ to infinity, we get a sequence of minimal hypersurfaces (with respect to a sequence of new metrics) intersecting $B$ with bounded index and area so by \cite{Sharp} they subsequentially converge to a complete embedded finite volume minimal hypersurface (with respect to $g$) which intersects $B$ and the theorem is proved.

\vspace{1em}

 Let us show the claim. We will need to set up a constrained minimization problem depending on some data. Before defining it we introduce some notations.
Let $D'\subset D$ be any domain containing $B$ with $C^{1,1}$ weakly mean convex boundary with respect to the new metric $g_D$ (hence $D$ itself would be an example of such a $D'$), let $\Sigma$ be any closed minimal hypersurface (possibly empty, $1$-sided or $2$-sided) embedded inside $\interior(D')\backslash B$ and let $\mathbf{D}$ be the metric completion of $D'\backslash \Sigma$. The domain $B$ is isometrically embedded inside $\mathbf{D}$.

Next, we consider the following constrained minimization problem which depends on the data $\mathbf{D}$ defined in the previous paragraph: minimize the $n$-volume of $\partial B'$ over open sets $B'$ with integer rectifiable boundary, containing $B$, such that there is a family of integral currents $\{b_t\}_{t\in[0,1]} \subset \mathbf{I}_{n+1}(\mathbf{D};\mathbb{Z}_2)$ verifying:
 \begin{itemize}
  \item $\{\partial b_t\}_{t\in[0,1]} \subset\mathcal{Z}_{n}(\mathbf{D};\mathbb{Z}_2)$ is continuous in the $\mathbf{F}$-topology,
 \item $\spt(\partial b_0) = \partial B$, $\spt(\partial b_1) = \partial B'$,
 \item for all $t\in[0,1]$, $B \subset \spt(b_t) \subset \mathbf{D}$
 \item for all $t\in[0,1]$, $\mathbf{M}(\partial b_t) \leq \Vol_n(\partial B)+1$
 \end{itemize}
 (see Appendix B for notations).
%$\partial B'$ and $\partial B$ are part of a $1$-parameter family of integral cycles continuous in the $\mathbf{F}$-topology, whose $n$-volumes are all at most $\Vol_n(\partial B) +1$. 
This is a constrained Plateau problem since there is an $n$-volume constraint (which does not affect regularity) and a geometric constraint given by $B$, $\mathbf{D}$ (which does). A solution of this minimization problem exists by weak mean convexity of $\partial \mathbf{D}$, by compactness of cycles for the flat topology and by interpolation results \cite[Proposition A.2]{MaNeindexbound}: it gives a $C^{1,1}$-hypersurface $\Gamma$ smooth except maybe at points touching $\partial  B$ and with $n$-volume at most $\Vol_n(\partial B)$. Let $B''$ be the closure of the component of $\mathbf{D}\backslash \Gamma$ containing the interior of $B$. 

Given $B\subset B''$ as above, 
%and a connected component of $\partial B''$, either it is a smooth minimal hypersurface, or it is not, it which case it is a $C^{1,1}$ hypersurface that touches $B$. In any case, 
by smoothing out the constrained Plateau problem using $\mu$-bubbles as in \cite[Subsection 1.4]{GromovPlateauStein}, $B''$ can be approximated from outside by domains of $\mathbf{D}$ with smooth boundary having nonnegative mean curvature. By the avoidance principle for the level set flow, $B''$ itself is thus a weakly mean convex set and so we get  the following dichotomy (see Appendix C):
\begin{enumerate}[label=(\roman*)] 
\item either one component (called $A$) of $\partial B''$ is strictly mean convex,
\item or the boundary $\partial B''$ is a smooth minimal hypersurface.
\end{enumerate}

For a choice of data $\mathbf{D}$ and completion $B''$ as above, in case (i), the strictly mean convex component $A$ has to touch $\partial B$. We can run the mean curvature flow starting from $A$ (see Appendix C) and get $\{A_t\}_{t\geq0}$. The level sets $A_t$ cannot sweep out the whole domain $B$ by assumption in the statement of the claim.
So $A_t$ converges as $t\to\infty$ to a non empty stable closed embedded minimal hypersurface $S$ intersecting $B$ and having $n$-volume less than $\Vol_n(\partial B)$. 
%If situation (i) occurs for a sequence of radii $r_i$ going to infinity and some corresponding $\mathbf{D}$, $B''$, then by sending $r_i$ to infinity (i.e. $D$ covers larger and larger regions of $M$), we get subsequently a limiting complete embedded finite volume minimal hypersurface (\cite{Sharp}) which intersects $B$. 
This proves the claim in case (i).

In case (ii) we can assume that $\partial B''$ does not touch $B$ (otherwise the claim is true). It means we can suppose that $\partial B''$ is a smooth minimal hypersurface locally area minimizing inside $B''$. Recall that the radius $r$ and $D$ are fixed. If for any possible choice of $\mathbf{D}$, $B''$, situation (ii) occurs then we can consider such a $B''$ of minimal volume, where the minimum is taken over all data $\mathbf{D}$ and  $B''$ as in the constrained minimization problem defined previously. Such a minimizer exists by compactness because of the $n$-volume bound on $\Gamma$ and the stability of $\partial B''$ (see \cite{Sharp}). This manifold $B''$ is compact.
Given such a minimizer $B''$, we remove a maximal number of disjoint 1-sided minimal hypersurfaces $S_1,...,S_q$ and $2$-sided non-separating minimal hypersurfaces $T_1,...,T_r$ contained in the interior of $B''\backslash B$ so as to have 
$$\Vol_n(\partial B'') + \sum_{i=1}^q 2\Vol_n(S_i)+ \sum_{i=1}^r \Vol_n(T_i) \leq \Vol_n(\partial B) +1$$
and consider the metric completion $\mathbf{B}$ of $B''\backslash \big(\bigcup_{i=1}^q S_i\cup \bigcup_{i=1}^r T_i\big)$. This manifold $\mathbf{B}$ is compact, 
%has a non empty boundary 
and by the previous inequality we have
$$\Vol_n(\partial \mathbf{B}) \leq \Vol_n(\partial B) +1.$$
Besides, the original domain $B$ is isometrically embedded in $\mathbf{B}$. We used similar ideas of considering a ``core'' in \cite{AntoineYau}.

We can apply a local min-max argument to $\mathbf{B}$. Let $W(\mathbf{B})$ be the width of $\mathbf{B}$ as defined in Appendix B, let $\Phi:[0,1] \to \mathcal{Z}_{n}(B;\mathbb{Z}_2)$ be any fixed sweepout of $B$ so that $\Phi(0)=0$, $\Phi(1)$ is $\partial B$ with multiplicity one ($\Phi$ can be for instance constructed with the level sets of a well chosen Morse function). By Proposition \ref{nonbumpy}, we get a closed connected embedded minimal hypersurface $S$ inside the interior $\interior(\mathbf{B})$, which has Morse index at most one. Moreover the width of $\mathbf{B}$ and thus the $n$-volume of $S$ are bounded in terms of $(B,g)$ only. To see this, it suffices 
to understand that one can deform $\partial B$ inside $\mathbf{B}$ continuously in the $\mathbf{F}$-topology to $\partial \mathbf{B}$, such that along the deformation the $n$-volume of the hypersurfaces is say less than $\Vol_n(\partial B) +1$. Gluing this deformation to $\Phi$ would then yield a sweepout of $\mathbf{B}$ and show
$$W(\mathbf{B})\leq \max_{t\in [0,1]} \mathbf{M}(\Phi(t)) +1.$$ 
The existence of such a deformation from $\partial B$ to $\partial\mathbf{B}$ inside $\mathbf{B}$ follows from the volume minimality property of $\mathbf{B}$ and \cite[Proposition A.2]{MaNeindexbound}. 

To finish the proof of the claim, we now wish to show that this min-max hypersurface $S$ intersects $B$. To argue towards a contradiction, assume that $S\subset \interior(\mathbf{B})$ does not intersect $B$. Suppose first that $S$ is 2-sided: then consider the metric completion $\mathbf{B}'$ of $\mathbf{B}\backslash S$. If $S$ is 1-sided or $2$-sided non-separating, then we could run again the constrained minimization procedure described earlier to $\partial B$ inside $\mathbf{B}'$. Either we find a minimizer not entirely contained in the boundary $\partial \mathbf{B}'$ or a minimizer is exactly $\partial \mathbf{B}'$ (then we have $\Vol_n(\partial \mathbf{B}') \leq \Vol_n(\partial B) +1$). In any case, this contradicts either the minimality of the volume of $\mathbf{B}$ or the maximality of the number of boundary components of $\partial \mathbf{B}$, since $\partial \mathbf{B}'$ has more boundary components than $\partial \mathbf{B}$. Similarly, if $S$ is $2$-sided and separates $\mathbf{B}$ then $\mathbf{B}'$ has two components and by another constrained minimizing procedure applied to $\partial B$ in the component of $\mathbf{B}'$ containing it, it is clear that one can find a competitor $\hat{\mathbf{B}}$ to $\mathbf{B}$ contradicting the minimality of its volume.

\end{proof}

\begin{proof} [Proof of Corollary \ref{explanation}]
The first item $(1)$ follows from the following argument. Note that the $\bigcap_i X_i$ contains a small fixed ball $b$. By Theorem \ref{a},  if the conclusion we want is not true then for each $i$ there is a family of closed sets $\{K_t\}_{0\leq t\leq 1}$ such that $X_i \subset K_0$, $X_i \cap K_1 =\varnothing$ and $\{\partial K_t\}$ is a mean convex foliation of $K_0$ for a metric that coincides with $g$ near $X_i$. Moreover this foliation is locally either smooth, or given by a level set flow. Let us check that the $n$-volume of $\partial (K_t\cap X_i)$ decreases in $t$, for all $i$. Fix $i$, let $t'<t\in(0,1)$, first have $K_t\subset K_{t'}$. Consider 
$$S:=\partial (K_{t'}\cap X_i )\backslash \interior(K_t),$$ which is a hypersurface (rigourously speaking a cycle) contained in $K_{t'}\backslash \interior(K_t)$ and with boundary in $\partial K_{t}$. Minimize the $n$-volume of $S$ under the constraints that 
\begin{itemize}
\item the hypersurface remains in the $\mathbb{Z}_2$-homology class of $S$ inside $K_{t'}\backslash \interior(K_t)$,
\item the boundary of the hypersurface is $\partial S$,
 \item the hypersurface is contained in $K_{t'}\backslash \interior(K_t)$.
 \end{itemize}
 By mean convexity of $K_{t'}$ (see \cite{WhiteregMCF}), we get a minimizer $S' \subset K_{t'}\backslash \interior(K_t)$ which is locally area minimizing.  By the maximum principle \cite[Theorem 3.5]{WhiteregMCF} (see also \cite[Theorem 7.1]{Whitetopology}), $S'$ is actually contained in $\partial K_t$. By the constancy theorem, $S'=\partial (K_{t}\cap X_i )\backslash \interior(K_t)$. So we just showed that
 \begin{align*}
 \Vol_n(\partial (K_{t'}\cap X_i )) & = \Vol_n(\partial (K_{t'}\cap X_i )\backslash \interior(K_t)) +\Vol_n(\partial X_i \cap \interior(K_t))\\
 & > \Vol_n(\partial (K_{t}\cap X_i )\backslash \interior(K_t))+\Vol_n(\partial  X_i \cap \interior(K_t)) \\
 & = \Vol_n(\partial (K_{t}\cap X_i )).
 \end{align*}

On the other hand, since for each $i$, $\partial (K_t\cap X_i)$ gives a family of cycles in $\mathcal{Z}_n(M;\mathbb{Z}_2)$ that sweeps out the ball $b$, the relative isoperimetric inequality for cycles implies
$$\max_{t\in[0,1]} \Vol_n(\partial (K_t\cap X_i))>c>0$$  
for a positive constant $c$ independent of $i$. This contradicts the facts that $\lim_{i\to \infty} \Vol_n(\partial X_i) = 0$ and that $\Vol_n(\partial (K_t\cap X_i))$ decreases in $t$.

Similarly, for the second item, if the desired conclusion does not hold then by Theorem \ref{a} (2), there is a family of closed sets  $\{K_t\}_{0\leq t\leq 1}$ such that $X \subset K_0$, $X \cap K_1 =\varnothing$ and $\{\partial K_t\}$ is a mean convex foliation of $K_0$ for a metric that coincides with $g$ near $X$. But this cannot happen if $X$ is strictly mean concave by the maximum principle.

\end{proof}

\section{Local min-max and saddle point minimal hypersurfaces} \label{saddl}

Consider a complete manifold $(M^{n+1},g)$. We define \textit{saddle points of the $n$-volume functional} (or simply \textit{saddle point minimal hypersurfaces}) as follows. Let $\Gamma$ be a connected closed embedded minimal hypersurface. If it is $2$-sided then we call it a saddle point if there is a smooth family of hypersurfaces $\{\Gamma_t\}_{t\in(-\varepsilon,\varepsilon)}$ ($\varepsilon>0$) which are small graphical perturbations of $\Gamma=\Gamma_0$ so that $\{\Gamma_t\}_{t\in(-\varepsilon,0)}$ and $\{\Gamma_t\}_{t\in(0,\varepsilon)}$ are on different sides of $\Gamma$ and distinct from $\Gamma$, and 
$$\Vol_n(\Gamma) = \max_{t\in(-\varepsilon,\varepsilon)} \Vol_n(\Gamma_t).$$
If $\Gamma$ is $1$-sided, we call it a saddle point if its connected double cover is a saddle point in a double cover of the ambient manifold.
Note that if the metric is bumpy (i.e. no closed minimal hypersurface has a non-trivial Jacobi field), then saddle point minimal hypersurfaces are exactly $2$-sided unstable closed embedded minimal hypersurfaces and $1$-sided closed embedded minimal hypersurfaces with unstable double cover.

The goal of this section, which constitutes one of the technical cores of this paper, is to construct saddle points for general metrics from a localized $1$-parameter min-max procedure. When the metric is bumpy, this was achieved by the index bounds of Marques-Neves in \cite{MaNeindexbound}. Here we need to prove similar results for possibly non-generic metrics, but we cannot just rely on an approximation argument by generic metrics. Most notations are defined in Appendix B.

\subsection{Deformation theorems}

In this subsection, we consider a compact manifold $(N,g)$. To simplify the presentation, let us assume that: 

\vspace{1em}
\textit{each component of $\partial N$ is minimal and either strictly stable or degenerate stable of Type II} 

\vspace{1em}
(see Appendix A for the definition of Type I, II, III degenerate stable minimal hypersurfaces).

Let $\{\Phi_i\}_{i\in \mathbb{N}}$ be a sequence of continuous maps from $[0,1]$ into $\mathcal{Z}_n(N;\mathbf{F};\mathbb{Z}_2)$. Set
\begin{equation} \label{wwwidth}
L=\mathbf{L}(\{\Phi_i\}_{i\in \mathbb{N}}) := \limsup_{i\to \infty} \sup_{x\in [0,1]} \mathbf{M}(\Phi_i(x))
\end{equation}
and suppose that $\liminf_{i\to \infty} (L-\max_{j=0,1} \mathbf{M}(\Phi_i(j)))>0$.

We first prove a $1$-parameter version of Deformation Theorem A of \cite{MaNeindexbound} for non-bumpy metrics. Let $\mathcal{V}_n(N)$ be the closure, in the weak topology, of the set of $n$-dimensional rectifiable varifolds in $N$. Let $\mathcal{S}(L)$ be the family of stationary integral varifolds in $\mathcal{V}_n(N)$ of total mass $L$ with support a smooth closed embedded minimal hypersurface in $N$, which are $2$-unstable (see \cite[Definitions 4.1, 4.2]{MaNeindexbound}). The unstable components of $\spt(V)$, where $V\in \mathcal{S}(L)$, are inside the interior of $N$ by assumption on the boundary $\partial N$.

\begin{theo} [Deformation Theorem A, \cite{MaNeindexbound}]
Given $\{\Phi_i\}_{i\in \mathbb{N}}$ and a compact set $K\subset \mathcal{V}_n(N)$ which is at positive $\mathbf{F}$-distance of $\mathcal{S}(L)\cup |\Phi_i|([0,1])$ for all $i$ large, there exists another sequence $\{\Psi_i\}_{i\in \mathbb{N}}$ such that 
\begin{enumerate}[label=(\roman*)]
\item $\Psi_i $ is homotopic to $\Phi_i$ with fixed endpoints in the $\mathbf{F}$-topology for all $i\in \mathbb{N}$,
\item $\mathbf{L}(\{\Psi_i\}_{i\in \mathbb{N}}) \leq L$,
\item for any $\Sigma\in \mathcal{S}(L)$, there exists $\bar{\varepsilon}>0$, $j_0\in \mathbb{N}$, so that for all $i\geq j_0$, $|\Psi_i|([0,1]) \cap (\bar{\mathbf{B}}^{\mathbf{F}}_{\bar{\varepsilon}}(\Sigma) \cup K)= \varnothing$.
\end{enumerate}

\end{theo}

\begin{proof} Let $K_1:= K\cup\{V ;\|V\|(N)\leq \frac{1}{2}\big(L+\liminf_{i\to \infty} \max_{j=0,1} \mathbf{M}(\Phi_i(j))\big)\}$ and 
let $d:=\min\{\mathbf{F}(\mathcal{S}(L),Z ) ; Z\in K_1\}>0$. Note that $\mathcal{S}(L)$ can be written as a countable union of compact subsets $\mathcal{S}_k$ ($k\geq 0$) of $\mathcal{V}_n(N)$:
$$\mathcal{S}(L) = \bigcup_{k\geq 0} \mathcal{S}_k.$$
For each $k\geq 0$, we can find positive numbers $\varepsilon_k$, $c_{0,k}$ and a finite number of varifolds $\Sigma_{k,1},...,\Sigma_{k,q_k} \in \mathcal{S}(L)$ such that
\begin{itemize}
\item for any $k\geq 0$,  $\mathcal{S}_k \subset \bigcup_{r=1}^{q_k} \mathbf{B}^{\mathbf{F}}_{\varepsilon_{k}} (\Sigma_{k,r})$,
\item each $\Sigma_{k,r}$ is $2$-unstable in an $\varepsilon_k$-neighborhood for some family $\{F^{k,r}_v\}_{v\in \bar{B}^2} \subset \Diff(N)$ and $c_{0,k}>0$ (see \cite[Definitions 4.1]{MaNeindexbound}).  
\end{itemize}

Without loss of generality (by changing $\varepsilon_k$, $\{F^{k,r}_v\}$, $c_{0,k}$) we can also assume that 
\begin{equation}\label{oloi}
\min\{\mathbf{F}((F^{k,r}_v)_\sharp V,Z ) ; Z\in K_1, v\in \bar{B}^2\}>d/2 \text{ for all } \Sigma_{k,r} \text{ and } V\in \bar{\mathbf{B}}^{\mathbf{F}}_{2\varepsilon_k} (\Sigma_{k,r})
\end{equation}
and if the sequence $\{\varepsilon_k\}$ is infinite,
\begin{equation}\label{alai}
 \sup_{v\in \bar{B}^2}\|F^{k,r}_v - Id\|_{C^1(N)} \leq  \theta_k  \\
\end{equation}
where the $\theta_k$ satisfy $\lim_{k\to \infty} \theta_k = 0$ and will be chosen later.

Given a positive number $u>0$ and some $k\geq 0$, let $\mathcal{N}^{(k)}_u$ be the $u$-neighborhood of $\mathcal{S}_k $:
$$\mathcal{N}^{(k)}_u:= \bigcup_{\Sigma \in \mathcal{S}_k}{\mathbf{B}}^{\mathbf{F}}_{u} (\Sigma).$$ 
Consider for each $k\geq 0$ a positive number $\eta_k$ so that the following holds: 
%\begin{enumerate}[label=(\roman*)]
%\item 
$$\mathcal{N}^{(k)}_{3\eta_k}\subset \bigcup_{r=1}^{q_k}\mathbf{B}^{\mathbf{F}}_{\varepsilon_k} (\Sigma_{k,r}).$$
This $\eta_k$ will be fixed later in the proof and will depend on $\varepsilon_k$.

%\item for any $V\in \mathcal{N}^{(k)}_{\eta_k} \cap  \mathbf{B}^{\mathbf{F}}_{\varepsilon_k} (\Sigma_k)$, there is a vector $v=v(k,V) \in \bar{B}^2$ such that 
 %$$(F^k_{v})_\sharp V \notin \mathcal{N}_{\eta_k}.$$
%\end{enumerate}

\vspace{1em} 
We are given $\{\Phi_i\}$. In what follows, for each $k\geq 0$ we define a procedure which ``moves the sweepout away from $\mathcal{S}_k$'' and call this procedure \emph{$k^{th}$ move}. 
 
To define the $k^{th}$ move, fix $k$ and a sweepout $\Phi : [0,1] \to \mathcal{Z}_n(N;\mathbf{F},\mathbb{Z}_2)$. The definition of the $k^{th}$ move only depends on $\Phi$, and the choice of $\varepsilon_k, \{F^{k,r}_v\}_{r=1}^{q_k}, c_{0,k}$. 
Let $U_{\eta_k}= [a_1,a_2]\cup...\cup[a_{2p-1},a_{2p}]$ be a union of closed disjoint intervals in $[0,1]$ (the choice is not unique) so that 
$$\text{ for all } x\in U_{\eta_k}=\bigcup_{l=1}^{p}{[a_{2l-1},a_{2l}]}, \quad |\Phi(x)| \in \mathcal{N}^{(k)}_{2\eta_k}$$ 
$$\text{ and for all } x\in [0,1]\backslash \bigcup_{l=1}^{p}{(a_{2l-1},a_{2l})}, \quad |\Phi(x)| \in \mathcal{V}_n(N) \backslash \mathcal{N}^{(k)}_{\eta_k}.$$
This union of intervals $U_{\eta_k}$ exists by $\mathbf{F}$-continuity of $\Phi$.

We want to modify $\Phi$ on each of the intervals $[a_{2l-1},a_{2l}]$ into a map $\Psi$ ($\Phi$ is left unmodified outside of $U_{\eta_k}$). Let us  describe the changes on $[a_1,a_2]$, the modifications on the other intervals being similar. First we can decompose this interval into 
$$[a_1,a_2]=[b_1,b_2]\cup[b_2,b_3]\cup...[b_{m-1},b_m] \quad (\text{where } a_1=b_1, a_2=b_m)$$
so that for every $x=1,...,m-1$, there exists $r_x$
such that 
$$|\Phi|([b_x,b_{x+1}])\subset \mathbf{B}^{\mathbf{F}}_{\epsilon_k}(\Sigma_{k,{r_x}}).$$
 By construction $|\Phi|(a_1)$ and $|\Phi|(a_2)$ are not in $\mathcal{N}^{(k)}_{\eta_k}$. Next we move each $|\Phi|(b_l)$ also outside of $\mathcal{N}^{(k)}_{\eta_k}$, but keeping it inside $\mathbf{B}^{\mathbf{F}}_{\epsilon_{k}}(\Sigma_{k,r_x})$ as follows. 
For each $x=2,...,m-1$ and $r_x$ as above, one can find a vector $v_x\in \bar{B}^2$ so that the mass $||(F^{k,r_x}_{sv_x})_\sharp |\Phi|(b_x)||(N)$ is decreasing in $s$, and for all $s\in [0,1]$ 
$$(F^{k,r_x}_{sv_x})_\sharp |\Phi|(b_x) \in \mathbf{B}^{\mathbf{F}}_{\epsilon_{k}}(\Sigma_{k,r_x}) \cap \mathbf{B}^{\mathbf{F}}_{\epsilon_{k}}(\Sigma_{k,r_{x-1}}),$$
$$(F_{v_x})_\sharp |\Phi|(b_x) \notin \mathcal{N}^{(k)}_{\eta_k}.$$
For the above to be true, it may be necessary to reduce $\eta_k$.

For $x=2,...,m-1$, let $A_x : [0,1] \to \mathcal{Z}_n(N;\mathbf{F};\mathbb{Z}_2)$ be the path 
$\forall s\in [0,1], A_x(s):= (F_{sv_x})_\sharp \Phi(b_x)$
and let $A_x^-$ be the same path but with reverse parametrization. Up to reparametrization, $\Phi\big|_{[a_1,a_2]}$ is clearly homotopic in the $\mathbf{F}$-topology to the following concatenation (where $+$ stands for concatenation):

\begin{equation} \label{concat}
(\Phi\big|_{[b_1,b_2]} + A_2 )+ (A_2^- +\Phi\big|_{[b_2,b_3]} +A_3 )+( A_3^-+...) +... +(A_{m-1}^-+\Phi\big|_{[b_{m-1},b_m]}).
\end{equation}
Each subsum in parentheses is a path $P:[0,1] \to \mathcal{Z}_n(N;\mathbf{F},\mathbb{Z}_2)$ so that the image (in the space of varifolds) of $|P|$ is included in a ball $\mathbf{B}^{\mathbf{F}}_{\epsilon_{k}}(\Sigma_{k,r})$, and whose endpoints satisfy $|P(0)|, |P(1)| \notin \mathcal{N}^{(k)}_{\eta_k}$.

Let us explain how to deform any such path $P$ into $Q$, fixing the endpoints, so that the varifold image $|Q|([0,1])$ stays uniformly far from $\mathcal{S}_k$. 
A similar more general deformation is the object of the proof of \cite[Deformation Theorem A]{MaNeindexbound}, we will use the same notations and explain the few modifications needed here. Given $\Sigma_{k,r}$ as in the previous paragraph, recall the definition of the map 
$$V\mapsto m(V)$$ defined on $\bar{\mathbf{B}}^{\mathbf{F}}_{\varepsilon_{k}}(\Sigma_{k,r})$, given in \cite[Definition 4.1]{MaNeindexbound}:  by choice of $\varepsilon_{k}$, $\Sigma_{k,r}$ is $2$-unstable in an $\varepsilon_{k}$-neighborhood for some family $\{F^{k,r}_v\}_{v\in \bar{B}^2}\subset \Diff(N)$ and $c_{0,k}>0$, and for any $V\in\bar{\mathbf{B}}^{\mathbf{F}}_{\varepsilon_{k}}(\Sigma_{k})$ the smooth function 
$$v\mapsto \|(F^{k,r}_v)_\sharp V\|(N)$$
has a unique maximum at $m(V)\in B^2_{c_{0,k}/\sqrt{10}}(0).$
By abuse of notations, let $m:[0,1]\to \bar{B}^2$ be defined by $m(s):= m(|P|(s))$. Following the arguments of \cite[Deformation Theorem A]{MaNeindexbound}, there is a continuous homotopy 
$$\hat{H}: [0,1] \times [0,1]\to \bar{B}^2_{1/2}(0) \text{ so that } \hat{H}(s,0)=0 \quad \forall s\in[0,1],$$
$$ \inf_{s\in[0,1]} |m(s) - \hat{H}(s,1)| \geq \mu>0.$$
$$ \text{ and for $j=0,1$ } \quad \|(F^{k,r}_{\hat{H}(j,t)})_\sharp |P|(j)\|(N) \text{ decreases in $t\in[0,1]$}.$$
Let $T=T(\mu, \Sigma_{k,r},\varepsilon_{k},\{F^{k,r}_v\}, c_{0,k}) >0$ be given by \cite[Lemma 4.5]{MaNeindexbound}. The new path $Q$ is then obtained as the (reparametrized) concatenation of the following paths (see \cite{MaNeindexbound} for notations):
\begin{itemize}
\item $Q_1(t) := (F^{k,r}_{\hat{H}(0,t)})_\sharp P(0)$ for $t\in[0,1]$,
\item $Q_2(t) :=  (F^{k,r}_{\phi^{P(0)}(\hat{H}(0,1),tT_i)})_\sharp P(0)$ for $t\in[0,1]$,
\item $Q_3(t) :=  (F^{k,r}_{\phi^{P(t)}(\hat{H}(t,1),T_i)})_\sharp P(t)$ for $t\in[0,1]$,
\item $Q_4(t) :=  (F^{k,r}_{\phi^{P(1)}(\hat{H}(1,1),(1-t)T_i)})_\sharp P(1)$ for $t\in[0,1]$,
\item $Q_5(t) :=  (F^{k,r}_{\hat{H}(1,(1-t))})_\sharp P(1)$ for $t\in[0,1]$.
\end{itemize}
This replaces each path $P:[0,1] \to \mathcal{Z}_n(N;\mathbf{F},\mathbb{Z}_2)$ in parentheses in (\ref{concat}) with a new path $Q:[0,1] \to \mathcal{Z}_n(N;\mathbf{F},\mathbb{Z}_2)$, and so it changes $\Phi$ on the interval $[a_1,a_2]$. We do  similar replacements for each interval $[a_{2l-1},a_{2l}]$, and we get a new sweepout $\Psi:[0,1]\to \mathcal{Z}_n(N;\mathbf{F};\mathbb{Z}_2)$ homotopic in the $\mathbf{F}$-topology to $\Phi$. This finishes the definition of the $k^{th}$ move for $\Phi$. The main point of this construction is that now, there is $\hat{\mu}_k>0$ depending only on $k$ such that
\begin{equation} \label{covi}
|\Psi|([0,1]) \cap \mathcal{N}^{(k)}_{\hat{\mu}_k} = \varnothing.
\end{equation}
Another useful property is that a $k^{th}$ move does not move images of the original sweepout $\Phi$ too far away: there is a constant $\gamma_k$ converging to $0$ as $k\to \infty$ such that for any $[b_{x},b_{x+1}]$, $r_x$ as above, 
\begin{equation} \label{unajout}
|\Psi|([b_{x},b_{x+1}]) \subset \mathbf{B}^{\mathbf{F}}_{\gamma_k} (\Sigma_{k,r_x}).
\end{equation}
This property follows from (\ref{alai}); in fact $\gamma_k$ can be forced to converge arbitrarily fast to $0$ if $\varepsilon_k$ and  $\theta_k$ as in (\ref{alai}) are chosen to converge sufficiently fast to $0$.

\vspace{1em}

Now remember that we are given $\{\Phi_i\}$ as in the statement of the theorem. For each $i\geq 0$, define $\tilde{\Psi}_i:[0,1]\to \mathcal{Z}_n(N;\mathbf{F};\mathbb{Z}_2)$ to be the sweepout obtained by applying successively $k^{th}$ moves starting with $\Phi_i$, for $k=0,1,...,i-1,i$ (in that order). Hence $\tilde{\Psi}_i$ is the result of modifying $\Phi_i$ a total of  $i+1$ times. It is clear by construction that 
$$\mathbf{L}(\{\tilde{\Psi}_i\}_{i\in \mathbb{N}}) \leq L.$$ 
By (\ref{covi}) we begin with $\tilde{\Psi}_0$ which satisfies
$$|\tilde{\Psi}_0|([0,1]) \cap \mathcal{N}^0_{\hat{\mu}_0} = \varnothing.$$
Clearly if $\varepsilon_1,\varepsilon_2...$ and $\gamma_1,\gamma_2,...$ are chosen to be decreasing fast enough (depending on $\hat{\mu}_0$), then by (\ref{unajout}) for any $i\geq 0$:
$$|\tilde{\Psi}_i|([0,1]) \cap \mathcal{N}^0_{\hat{\mu}_0/2} = \varnothing.$$
We can continue similarly by induction on $i$ and choose $\{\varepsilon_j\}_{j\geq i}$ and $\{\gamma_j\}_{j\geq i}$ small enough (depending on $\hat{\mu}_{i-1}$) so that eventually for all $i_0$ and $i\geq i_0$,
$$|\tilde{\Psi}_i|([0,1]) \cap \mathcal{N}^{i_0}_{\hat{\mu}_{i_0}/2} = \varnothing.$$
This is enough to conclude the proof.

\end{proof}

We will not need Deformation Theorem B of \cite{MaNeindexbound}. However, Deformation Theorem C of \cite{MaNeindexbound} will be useful. Before stating it, consider a minimal hypersurface $S$ which is degenerate stable of Type II. Then we associate to $S$ some squeezing maps like the ones in \cite[Proposition 5.7]{MaNeindexbound}. Suppose that $S$ is 2-sided embedded in the interior of $N$, the other cases (1-sided or boundary component) are similar. A neighborhood of $S$ is foliated by hypersurfaces with mean curvature vector pointing towards $S$ when non-zero (Appendix A, Lemma \ref{degenerate}). Let $f$ be a real function defined on such a neighborhood with $\nabla f\neq 0$ such that $S(s) = f^{-1}(s)$ ($s\in[-1,1]$) are the smooth embedded hypersurfaces of the foliation and $S(0)=S$ (in particular, $s\langle \nabla f , \overrightarrow{H}(S(s)) \rangle \leq 0$). Write $X=\nabla f /|\nabla f|^2$, and let $\phi:S\times [-1,1]\to N$ such that $\frac{\partial \phi}{\partial s} (x,s) = X(\phi(x,s))$ and $\phi(S,s) = S(s)$ for all $s\in[-1,1]$. Set $\Omega_r = \phi(S\times (-r,r))$ and define the maps 
$$P_t:\Omega_1\to\Omega_1 \text{ such that } P_t(\phi(x,s)) = \phi(x,s(1-t)) \text{ for } t\in[0,1].$$ 

\begin{lemme}\label{squeeze}
 Let $S$, $P_t$ be as above. There exists $r_0>0$ such that 
\begin{enumerate}
\item $P_t$ satisfies items (i), (ii), (iii) of \cite[Proposition 5.7]{MaNeindexbound},
\item for all $V\in \mathcal{V}_n(\Omega_{r_0})$ and every connected component $\Omega$ of $\Omega_{r_0}$, the function $t \mapsto \|(P_t)_\sharp V\|(\Omega)$ is a strictly decreasing function of $t$, unless $\spt(V) \cap \Omega \subset S\cap\Omega$, in which case it is constant.
\end{enumerate}
\end{lemme}

\begin{proof}
The only minor change compared to \cite[Proposition 5.7]{MaNeindexbound} is that $t \mapsto \|(P_t)_\sharp V\|(\Omega)$ does not have strictly negative derivative. However, by inspecting the computations in its proof, we see that $t \mapsto \|(P_t)_\sharp V\|(\Omega)$ always has nonpositive derivative and for any $a<b\in[0,1]$, $t \mapsto \|(P_t)_\sharp V\|(\Omega)$ must have negative derivative at some time $t'\in(a,b)$ because $S$ is of Type II, except when $\spt(V) \cap \Omega \subset S\cap\Omega$.
\end{proof}

We observe that if $S$ is 1-sided, or a boundary component of $N$, the previous discussion still applies to a neighborhood of $S$, on which one can define squeezing maps $P_t$. See Appendix B for the definition of $\mathbf{\Lambda}(\{\Phi_i\})$.

\begin{theo}[Deformation Theorem C, \cite{MaNeindexbound}]
Suppose the sequence $\{\Phi_i\}$ is pulled-tight (every varifold in $\mathbf{\Lambda}(\{\Phi_i\})$ with $\|V\|(N)=L$ is stationary). Let $\{\Sigma^{(1)},...,\Sigma^{(Q)}\} \subset \mathcal{V}_n(N)$ be a collection of stationary integral varifolds such that for every $1\leq q\leq Q$:
\begin{itemize}
\item the support of $(\Sigma^{(q)})$ for $1\leq q\leq Q$ is a closed embedded minimal hypersurface $S^{(q)}$ whose components are either in the interior of $N$ or a component of $\partial N$,
\item each component of $S^{(q)}$ (its double cover if not 2-sided) is strictly stable or degenerate stable of Type II,
\item $L=\|\Sigma^{(q)}\|(N)$.
\end{itemize} 
Then there exist $\xi>0$, $j_0\in \mathbb{N}$ so that for all $i\geq j_0$ one can find $\Psi_i:[0,1]\to \mathcal{Z}_n(N;\mathbf{F};\mathbb{Z}_2)$ such that
\begin{enumerate}[label=(\roman*)]
\item $\Psi_i$ is homotopic to $\Phi_i$ with fixed endpoints in the flat topology,
\item $\mathbf{L}(\{\Psi_i\}_{i\in \mathbb{N}}) \leq L$,
\item $$\mathbf{\Lambda}(\{\Psi_i\}) \subset \big(\mathbf{\Lambda}(\{\Phi_i\})\backslash \cup_{q=1}^Q \mathbf{B}^\mathbf{F}_\xi(\Sigma^{(q)})\big) \cup \big(\mathcal{V}_n(N) \backslash  \mathbf{B}^\mathbf{F}_\xi(\Gamma)\big),$$
where $\Gamma$ is the collection of all stationary integral varifolds $V\in \mathcal{V}_n(N)$ with $L=\|V\|(N)$.
\end{enumerate}
\end{theo}

\begin{proof}
The proof is almost identical to the case of bumpy metrics \cite{MaNeindexbound}. The first difference is that we use Lemma \ref{squeeze} instead of \cite[Proposition 5.7]{MaNeindexbound} for degenerate stable minimal hypersurfaces of Type II. The second difference is that \cite[Corollary 5.8]{MaNeindexbound} is replaced with the following statement: for $q\in \{1,...,Q\}$ there exists $\varepsilon_0=\varepsilon_0(\Sigma^{(q)})$ so that every stationary integral varifold $V$ in $\mathbf{B}^{\mathbf{F}}_{\varepsilon_0}(\Sigma^{(q)})$ of total mass $L$ coincides with $\Sigma^{(q)}$. This statement is shown similarly to \cite[Corollary 5.8]{MaNeindexbound}, using Lemma \ref{squeeze} (2) and the constancy theorem for varifolds.
\end{proof}

Finally we need a last deformation theorem in order to deal with degenerate stable minimal hypersurfaces of Type I, which has no analogue in the generic case \cite{MaNeindexbound}. Let $S$ be such a connected hypersurface (it is necessarily 2-sided and contained in the interior of $N$). Note that one can again define squeezing maps here. A neighborhood inside $\interior(N)$ of $S$ is foliated by hypersurfaces with mean curvature vector pointing towards a fixed direction when non-zero. Let $f$ be a real function defined on such a neighborhood with $\nabla f\neq \varnothing$ such that $S(s) = f^{-1}(s)$ ($s\in[-1,1]$) are the smooth embedded hypersurfaces of the foliation, verifying $S(0) = S$ and $\langle \nabla f , \overrightarrow{H}(S(s)) \rangle \leq 0$). Write $X=\nabla f /|\nabla f|^2$, and let $\phi:S\times [-1,1]\to N$ such that $\frac{\partial \phi}{\partial s} (x,s) = X(\phi(x,s))$ and $\phi(S,s) = S(s)$ for all $s\in[-1,1]$. Set $\Omega'_r = \phi(S\times [-r,r))$ and define the maps 
$$P'_{r,t}:\Omega_r\to\Omega_r \text{ such that } P'_{r,t}(\phi(x,s)) = \phi(x,(s+r)(1-t)-r) \text{ for } t\in[0,1].$$ 
The following lemma is proved as Lemma \ref{squeeze} and \cite[Proposition 5.7]{MaNeindexbound}.

\begin{lemme}\label{squeeze2}
Let $S$ be as above. There exists $r_1>0$ such that $P'_{r_1,t}: \Omega'_{r_1} \to \Omega'_{r_1}$ satisfies:
\begin{enumerate}
\item $P'_{r_1,0}(x)=x$ for all $x\in \Omega'_{r_1}$, $P_{r_1,t}(y) = y $ for all $y\in S(-r_1)$, $t\in[0,1]$,
\item $P'_{r_1,t}(\phi(S\times [-r_1,r))\subset \phi(S\times [-r_1,r))$ for all $t\in[0,1]$, $r\leq r_1$, and $P'_{r_1,1}(\Omega'_{r_1})=S(-r_1)$,
\item the map $P'_{r_1,t} : \Omega'_{r_1}\to \Omega'_{r_1}$ is a diffeomorphism onto its image for $0\leq t<1$, 
\item for all $V\in \mathcal{V}_n(\Omega'_{r_1})$,
% and every connected component $\Omega'$ of $\Omega'_{r_1}$, 
the function $t \mapsto \|(P'_{r_1,t})_\sharp V\|(\Omega'_{r_1})$ is a strictly decreasing function of $t$, unless $\spt(V) \subset S(-r_1)$, in which case it is constant,
\item if $V_0$ is a stationary integral varifold with support $S$, then for all $\varepsilon>0$, there are $\kappa=\kappa(V_0,\varepsilon)>0$, $\varepsilon'=\varepsilon'(V_0,\varepsilon)>0$ and $t'=t'(V_0, \varepsilon)\in(0,1)$ so that for all $V\in \mathcal{V}_n(\Omega'_{r_1})\cap \mathbf{B}^\mathbf{F}_{\varepsilon'}(V_0)$ and for all $s\in[0,t']$, 
$$(P'_{r_1,s})_\sharp V \in \mathbf{B}^\mathbf{F}_{\varepsilon}(V_0) \text{ and }$$
$$ \|(P'_{r_1,t'})_\sharp V\|(\Omega'_{r_1})\leq \|V_0\|(\Omega'_{r_1}) - \kappa.$$
\end{enumerate}
\end{lemme}
\begin{proof}
Compared to the proof of Lemma \ref{squeeze}, the only new point is bullet (5), which follows from the continuity of $(P'_{r_1,s})_\sharp$ in the $\mathbf{F}$-topology.
\end{proof}

For the last deformation theorem, we will assume that
\begin{align*}
[\star]_L:  &\text{ any degenerate stable minimal hypersurface in $N$ }\\
&\text{ of $n$-volume at most $L$ is of Type I or II.} 
 \end{align*}
Let $\mathcal{T}(L)$ be the family of stationary integral varifolds in $\mathcal{V}_n(N)$ of total mass $L$ whose support is a stable smooth closed embedded minimal hypersurface and for which at least one of the components is degenerate stable of Type I.

\begin{remarque} \label{oulll}

(1) By \cite{SchoenSimon} for any sequence $V_i\in \mathcal{T}(L)$, the supports $\spt(V_i)$ converge subsequently smoothly to a stable minimal hypersurface $S_0$. Assuming $[\star]_L$, the components of $S_0$ are either strictly stable, or degenerate stable of Type I, II. At least one of the components of $S_0$ is degenerate stable of Type I. Indeed, we can write $V_i = a_{1,i}|S_i^1|+...+a_{{k_i},i}|S_i^{k_i}|$ where $S_i^1,...,S_i^{k_i}$ are disjoint minimal hypersurfaces, $k_i$, $a_{1,i}$,...,$a_{{k_i},i}$ are bounded independently of $i$ and we can suppose that these sequence of integers all stabilize to respectively $k$, $a_1$,...,$a_k$. Then $V_i = a_{1,i}|S_i^1|+...+a_{{k_i},i}|S_i^{k_i}|$ converges to $V=a_{1}|S^1|+...+a_{{k}}|S^{k}|$ where $S^l$ are the components of $S_0$ and are the smooth limit of $S^{l}_i$. But if $S^l$ are all strictly stable or degenerate stable of Type II, \cite[Proposition 5.7]{MaNeindexbound} and Lemma \ref{squeeze} would imply that $S^l_i=S^l$ for $l$ large since the mass $\|V_i\|(N)$ is constant equal to $L$. That contradicts the fact that $V_i\in \mathcal{T}(L)$. As a consequence of this discussion, assuming $[\star]_L$, $\mathcal{T}(L)$ is a compact subset of $\mathcal{V}_n(N)$ and for a given $\varepsilon>0$, the quantities $\kappa$, $\varepsilon'$ and $t'$ of Lemma \ref{squeeze2} can be chosen independently of $V \in \mathcal{T}(L)$. 
%the number $r_1$ in Lemma \ref{squeeze2} can be chosen independently of $V \in \mathcal{T}(L)$.

\vspace{1em}
(2) Similarly, under condition $[\star]_L$, the number of stationary integral varifolds $\Sigma$ of mass $L$ satisfying the assumptions of Deformation Theorem C is finite: indeed otherwise, we get a converging sequence of such varifolds $\Sigma^{(i)}$ whose support is a stable minimal hypersurface with components (the double cover if not $2$-sided) either strictly stable or degenerate stable of Type II. By compactness and a Jacobi field argument, the limit varifold has support a degenerate stable minimal hypersurface. By $[\star]_L$, its components are of Type I or II. By arguments like (1) of this remark, this is impossible.

\end{remarque}

Here is the last deformation theorem, which is new compared to \cite{MaNeindexbound}.

\begin{theo}[Deformation Theorem D]
Assuming $[\star]_L$, given $\{\Phi_i\}_{i\in \mathbb{N}}$, there is another sequence $\{\Psi_i\}_{i\in \mathbb{N}}$ such that
\begin{enumerate}[label=(\roman*)]
\item $\{\Psi_i\}_{i\in \mathbb{N}}$ is homotopic to $\{\Phi_i\}_{i\in \mathbb{N}}$ with fixed endpoints in the flat topology for all $i\in \mathbb{N}$,
\item $\mathbf{L}(\{\Psi_i\}_{i\in \mathbb{N}})\leq L,$
\item 
%for any $\Sigma\in \mathcal{T}(L)$, 
there exists $\hat{\varepsilon}>0$, $j_1\in\mathbb{N}$ so that for all $i\geq j_1$, 
$$\inf\{\mathbf{F}(\Sigma, |\Psi_i|([0,1])) ; \Sigma\in \mathcal{T}(L)\} \geq \hat{\varepsilon}.$$

%\cap \bar{\mathbf{B}}^\mathbf{F}_{\hat{\varepsilon}}(\Sigma) =\varnothing$.
\end{enumerate}
\end{theo}

\begin{proof} 
For any $\Sigma \in \mathcal{T}(L)$, one component $S_1$ of $\spt(\Sigma)$ is degenerate stable of Type I. Recall that $S_1$ has a neighborhood $\Omega'_{r_1}$ and associated squeezing maps $P'_{r_1,t}:\Omega'_{r_1}\to \Omega'_{r_1}$. Let $\tilde{\Omega}$ be a neighborhood of $\Sigma$ such that one of the component of $\tilde{\Omega}$ is $\Omega'_{r_1}$. By abuse of notations, we denote by $P'_{r_1,t}$ the map from $\tilde{\Omega}$ to itself, equal to $P'_{r_1,t}$ on $\Omega'_{r_1}$ and equal to the identity map on the other components. By Remark \ref{oulll} (1), 
%$r_1$ of Lemma \ref{squeeze2} can and will be chosen uniformly in $\Sigma \in \mathcal{T}(L)$. Similarly 
for a given $\varepsilon>0$, the quantities $\kappa$, $\varepsilon'$ of Lemma \ref{squeeze2} can be chosen uniformly over $\Sigma \in \mathcal{T}(L)$. 

Let $\varepsilon>0$ be small enough so that for all $\Sigma\in \mathcal{T}(L)$ and all path $p:[0,1]\to\mathbf{B}^\mathbf{F}_{\varepsilon}(\Sigma)$, we can apply the three constructions 5.9, 5.11, 5.13 in \cite{MaNeindexbound}, where $p$ replaces the connected components of $V_{i,\varepsilon}$, $P'_{r_1,t}$ replaces their maps $P_t$. It exists by compactness of $\mathcal{T}(L)$. We also suppose $\varepsilon$ small enough so that for all $V$ in an $\varepsilon$-neighborhood of $\mathcal{T}(L)$, 
$$\|V\|(N)\geq (L+\liminf_{i\to \infty} \max_{j=0,1} \mathbf{M}(\Phi_i(j)))/2$$
(this is to make sure not to modify the endpoints $\Phi_i(0)$, $\Phi_i(1)$).

For this $\varepsilon$, let $\varepsilon' < \varepsilon$, $t'$, $\kappa$ be given by Lemma \ref{squeeze2} ($\varepsilon'$ and $\kappa$ can be chosen uniformly in $V_0\in \mathcal{T}(L)$ by compactness, see Remark \ref{oulll} (1)).  Choose also an $\varepsilon''<\varepsilon'$ so that for all $\Sigma,\Sigma'\in \mathcal{T}(L)$, $V\in \mathbf{B}^\mathbf{F}_{\varepsilon''}(\Sigma)\cap\mathbf{B}^\mathbf{F}_{\varepsilon''}(\Sigma') $, using the techniques in the first construction 5.9 in \cite{MaNeindexbound} (cf \cite[Lemma 7.1]{Montezuma1}), there is a path $\{C_t\}_{t\in [0,1]}$ continuous in the mass topology from $C_0=V$ to a cycle $C_1$ with 
\begin{itemize}
\item $\spt(C_1)\subset \tilde{\Omega} \cap \tilde{\Omega}'$ where $ \tilde{\Omega}$, $\tilde{\Omega}'$ depend respectively on $\Sigma$, $\Sigma'$ and are defined as in the first paragraph of the proof,
\item for all $t\in [0,1]$, $|C_t| \in \mathbf{B}^\mathbf{F}_{\varepsilon'}(\Sigma)\cap\mathbf{B}^\mathbf{F}_{\varepsilon'}(\Sigma') $. 
\end{itemize}

We can cover $\mathcal{T}(L)$ with a finite number of balls $\mathbf{B}^\mathbf{F}_{\varepsilon''}(\Sigma_k)$, where $\Sigma_1,...,\Sigma_K\in \mathcal{T}(L)$. If $u>0$, let $\mathcal{N}_u$ be the $u$-neighborhood of $\mathcal{T}(L)$ in the $\mathbf{F}$ topology. Let $\mu>0$ be small enough so that 
\begin{equation} \label{truc}
\mathcal{N}_{2\mu}\subset \bigcup_{k=1}^K\mathbf{B}^\mathbf{F}_{\varepsilon''}(\Sigma_k) \text{ and}
\end{equation}
\begin{equation} \label{trac}
\forall V\in \mathcal{N}_{\mu}, \quad \|V\|(N)> L-\kappa.
\end{equation}
Consider $\Phi_i:[0,1]\to \mathcal{Z}_n(N;\mathbf{F},\mathbb{Z}_2)$ and let $U_{i,\mu} = [a_1,a_2]\cup...\cup [a_{2p-1},a_{2p}]$ be a union of closed disjoint intervals in $[0,1]$ so that 
$$\text{ for all } x\in U_{i,\mu}, |\Phi_i(x)| \in \mathcal{N}_{2\mu}$$
$$\text{ and for all } x\in [0,1]\backslash \bigcup_{l=1}^p(a_{2l-1},a_{2l}), |\Phi_i(x)|\in \mathcal{V}_n(N)\backslash \mathcal{N}_\mu.$$
For each $i$ large, we want to modify $\Phi_i$ on each interval $[a_{2l-1},a_{2l}]$ into a map $\Psi_i$ which coincide with $\Phi_i$ outside of $U_{i,\mu}$. Let us first focus on $[a_1,a_2]$, the other intervals will be treated in the same way. By (\ref{truc}), we can write $[a_1,a_2]$ as a union $[b_1,b_2]\cup[b_2,b_3]\cup...\cup[b_{m-1},b_m]$ ($a_1=b_1$ and $a_2=b_m$) such that for each $l=1,...,m_1$, there exists a $k=k(l)\in \{1,...,K\}$ with $|\Phi|([b_l,b_{l+1}])\subset \mathbf{B}^\mathbf{F}_{\varepsilon''}(\Sigma_k)$. Since $\Sigma_{k(l)}\in \mathcal{T}(L)$, 
%one component $S_1$ of $\spt(\Sigma_k)$ is degenerate stable, satisfies Condition [M] and is of Type I. R
recall from the beginning of the proof, that there is a neighborhood $\tilde{\Omega}$ of $\spt(\Sigma_{k(l)})$ and there are associated squeezing maps $P'_{r_1,t}:\tilde{\Omega}\to \tilde{\Omega}$. 
%Let $\tilde{\Omega}$ be a neighborhood of $\Sigma_k$ such that one of the component of $\tilde{\Omega}$ is $\Omega'_{r_1}$. 

By construction, $|\Phi_i|(a_1)$ and $|\Phi_i|(a_2)$ are outside of $\mathcal{N}_\mu$. We want to move each of the other $|\Phi_i|(b_l)$ outside of $\mathcal{N}_\mu$ as well, with a mass control on the deformation. Fix $l\in\{2,...,m-1\}$ and consider the neighborhood $\tilde{\Omega}$ of $\spt(\Sigma_{k(l)})$, and associated squeezing maps $P'_{r_1,t}:\tilde{\Omega}\to \tilde{\Omega}$. By the choice of $\varepsilon''<\varepsilon'$, and the first construction 5.9 in \cite{MaNeindexbound} (cf \cite[Lemma 7.1]{Montezuma1}), there is a path $\{C_t\}_{t\in [0,1]}$ continuous in the mass topology from $C_0=\Phi_i(b_l)$ to a cycle $C_1$ with 
\begin{itemize}
\item $\spt(C_1)\subset \tilde{\Omega}$,
\item for all $t\in [0,1]$, $|C_t| \in \mathbf{B}^\mathbf{F}_{\varepsilon'}(\Sigma_{k(l)})\cap \mathbf{B}^\mathbf{F}_{\varepsilon'}(\Sigma_{k(l-1)})$. 
\end{itemize}
Thanks to the second item above and Lemma \ref{squeeze2}, 
$$\forall s\in[0,t'], \quad(P'_{r_1,s})_\sharp |C_1| \in \mathbf{B}^\mathbf{F}_{\varepsilon}(\Sigma_{k(l)})\cap \mathbf{B}^\mathbf{F}_{\varepsilon}(\Sigma_{k(l-1)}),$$
$$\|(P'_{r_1,t'})_\sharp |C_1\|\\|(N) \leq L-\kappa$$
(here $P'_{r_1,s}$ is associated to $\Sigma_{k(l)}$). In particular by (\ref{trac}) 
$$(P'_{r_1,t'})_\sharp |C_1|\notin \mathcal{N}_\mu.$$
Denote by $A_l:[0,1]\to \mathcal{Z}_n(N;\mathbf{F};\mathbb{Z}_2)$ the (reparametrized) concatenation of $\{C_t\}_{t\in[0,1]}$ and $\{(P'_{r_1,t})_\sharp C_1\}_{t\in[0,1]}$. Let $A^-_l$ denote the same path with reverse parametrization. Up to reparametrization, $\Phi_i\big|_{[a_1,a_2]}$ is homotopic in the $\mathbf{F}$-topology to the following concatenation:
%\begin{equation}\label{oconcat}
$$
(\Phi_i\big|_{[b_1,b_2]} + A_2) + (A^-_2 +\Phi_i\big|_{[b_2,b_3]}+A_3) +...+(A^-_{m-1}\Phi_i\big|_{[b_{m-1},b_m]}).
$$
%\end{equation}
Each subsum in parentheses is a path $p_l:[0,1]\to \mathcal{Z}_n(N;\mathbf{F};\mathbb{Z}_2)$ ($l=1,...,m_1$) so that the (varifold) image of $|p_l|$ is included in a ball of the form $\mathbf{B}^\mathbf{F}_\varepsilon(\Sigma_k)$, whose endpoints satisfy $|p_l(0)|,|p_l(1)|\notin \mathcal{N}_\mu$. We can apply the first, second and third constructions 5.9, 5.11, 5.13 in \cite{MaNeindexbound} (with $p_l$ replacing their $V_{i,\varepsilon}$ and our squeezing maps $P'_{r_1,t}$ replacing their maps $P_t$) to $p_l$ and get a path $q_l$ which is $\mathbf{F}$-continuous and homotopic to $p_l$ in the flat topology, with the following properties: 
\begin{itemize}
\item the endpoints are the same $p_l(j)=q_l(j)$ ($j=0,1$) and are not in  $\mathcal{N}_\mu$,
\item 
\begin{align*}
\max_{t\in[0,1]}\mathbf{M}(q_l(t)) & \leq \max\{\mathbf{M}(p_l(0)),\mathbf{M}(p_l(1))\}+1/i\\
& \leq \max_{t\in [0,1]}\mathbf{M}(\Phi_i(t))+1/i,
\end{align*}
\item there exists $\hat{{\varepsilon}}=\hat{{\varepsilon}}(\mu)>0$ such that if $i$ is large enough,
$$\inf\{\mathbf{F}(\Sigma, |q_l|([0,1])) ; \Sigma\in \mathcal{T}(L)\} \geq \hat{\varepsilon}.$$
\end{itemize}
The last item follows from arguments very similar to Claims 1 and 2 in the proof of \cite[Deformation Theorem C]{MaNeindexbound}, and item (5) of Lemma \ref{squeeze2}. Finally we concatenate $q_1$,...,$q_{m-1}$ to get $\Psi_i\big|_{[a_1,a_2]}$ and we proceed similarly for the other intervals $[a_2,a_3]$,...,$[a_{2p-1},a_{2p}]$ to get $\Psi_i$.

\end{proof}

\subsection{Existence of saddle point minimal hypersurfaces}
Equipped with the previous deformation theorems, we prove the existence of saddle points in non-bumpy metrics. Remember that if a degenerate stable minimal hypersurface satisfies Condition [M] then it is of type I, II, or III (see Appendix A). The following theorem is an extension of Proposition \ref{nonbumpy}.

\begin{theo}\label{sadddle}
Let $(N^{n+1},g)$ be a compact manifold with minimal boundary, with $2\leq n\leq 6$, such that $\partial N$ is locally area minimizing inside $N$. Then there is a saddle point minimal hypersurface $\Gamma$ inside the interior $\interior(N)$, whose index is at most one and whose $n$-volume is bounded by $W+1$ where $W$ is the width of $(N,g)$. 
\end{theo}

\begin{proof}
We can assume that any minimal hypersurface embedded in $N$ of $n$-volume at most $W$ and index at most one satisfies Condition [M] (see Appendix A).

Let $W$ be the width of $N$ and let $\{\Phi_i\}_{i\in \mathbb{N}}$ be a pulled-tight sequence of sweepouts so that 
$$\mathbf{L}(\{\Phi_i\}_{i\in \mathbb{N}})=W.$$
Recall $W$ is larger than the $n$-volume of any component of $\partial N$. By Condition [M] each component of $\partial N$  is either strictly stable or is degenerate stable of Type II (by convention a minimal boundary component cannot be of Type I, see Appendix A).
We can also suppose that $[\star]_W$ is satisfied (see before Remark \ref{oulll}) otherwise there is already a saddle point $\Gamma$ of index at most one and $n$-volume bounded by $W+1$. We first apply Deformation Theorem D. Then 
%by using the monotonicity formula, Lemma \ref{squeeze}, Lemma \ref{squeeze2} and Lemma \ref{degenerate} (4), 
note that $\mathcal{T}(W)$ is at $\mathbf{F}$-positive distance from $\mathcal{S}(W)$: this can be checked with the compactness of $\mathcal{T}(W)\subset \mathcal{V}_n(N)$ (Remark \ref{oulll} (1)) and Lemma \ref{squeeze}, Lemma \ref{squeeze2}. Hence we can apply Deformation Theorem A with $K=\bar{\mathbf{B}}^\mathbf{F}_{\hat{\varepsilon}}(\mathcal{T}(W))$ (where $\hat{\varepsilon}$ might be chosen smaller than the one given by Deformation Theorem D).

We pull-tight the sequence of sweepouts obtained; by the properties of the pull-tight map (see \cite[Subsection 3.6]{MaNeindexbound}), this new pulled-tight sequence of sweepouts still satisfies the conclusions of Deformation Theorems A and D for some $\hat{\varepsilon}$ small enough. By Remark \ref{oulll} (2), the set of varifolds $\{\Sigma^{(i)}\}$ satisfying the assumptions of Deformation Theorem C is finite, so we can apply Deformation Theorem C to all these varifolds. Let $\{\Psi_i\}$ be the resulting sequence of sweepouts. By Almgren-Pitts' theory (see \cite[Theorem 3.8]{MaNeindexbound} and use that $\liminf_{i\to \infty} (W-\max_{j=0,1} \mathbf{M}(\Psi_i(j)))>0$, see also proof of Theorem \ref{discreteminmax} in Appendix B), an element $V$ of $\mathbf{\Lambda}(\{\Psi_i\})$ has smooth support and mass $W$. No component of $\spt(V)$ is $2$-unstable or degenerate stable of Type I, and the components of $\spt(V)$ (their double covers if not 2-sided) cannot be all strictly stable or degenerate stable of Type II. Hence, since we are assuming $[\star]_W$, at least one of its components satisfies the following:
\begin{itemize}
\item either it has Morse index one,
 \item or it is stable, 1-sided and its double cover is unstable.
 \end{itemize}
%So either we obtain an index one 2-sided closed embedded minimal hypersurface of $n$-volume at most $W$, or we get a 1-sided  closed embedded minimal hypersurface with unstable double cover.
In both cases, this minimal hypersurface is a saddle point and necessarily contained in the interior of $N$.
\end{proof}

\section{Zero-infinity dichotomy for manifolds thick at infinity}\label{sectiondichotomy}

\subsection{Local version of Gromov's result for manifolds thick at infinity} \label{womm}

Consider a complete manifold $(M^{n+1},g)$. Recall that saddle point minimal hypersurfaces are closed embedded minimal hypersurfaces satisfying a natural saddle point condition. By ``compact domain'', we mean a compact $(n+1)$-dimensional submanifold of $M$ with smooth boundary.

Let $B$ be a compact domain of $M$ and let $\Sigma \subset (M,g)$ be a closed embedded minimal hypersurface which may be empty. Suppose that $\Sigma$ is locally area-minimizing (it minimizes the $n$-volume among all smooth hypersurfaces isotopic to $\Sigma$ contained in a small tubular neighborhood of $\Sigma$). We say that $B\backslash \Sigma$ has a \emph{singular weakly mean convex foliation} if there is a compact Riemannian manifold $(B_0,g')$ with $C^{1,1}$ boundary containing isometrically a neighborhood of $B \cup \Sigma$ in $M$ endowed with $g$,
% and endowed with a metric $g'$ coinciding with $g$ on a neighborhood of $B \cup \Sigma$, 
so that there is a family of closed subsets of $B_0$, $\{K_t\}_{t\in[0,1)}$, satisfying:
\begin{itemize}
\item $K_0=B_0$, 
$$K_{t'}\subset K_{t} \text{ if  $t'>t$ and}$$
$$\interior(B) \cap \bigcap_{t\in[0,1)} K_t = \interior(B)\cap \Sigma,$$
%for all neighborhood $N$ of $\Sigma$, $K_t\cap (B\backslash N))=\varnothing$ for $t$ close enough to $1$ and $\Sigma\subset K_t$ for all $t\in [0,1)$, 
\item each $K_t$ is an integral $n+1$-current and $\{\partial K_t\}$ yields a family of cycles in $\mathcal{Z}_n(B_0;\mathbb{Z}_2)$ continuous in the flat topology (see Appendix B-C), 
\item $ \partial K_t$ is weakly mean convex in the sense of mean curvature flow, and the non strictly mean convex level sets $\partial K_t$ are smoothly embedded closed minimal hypersurfaces in $(B_0,g')$ (see Appendix C),
\item for any $t\in(0,1)$, if $\partial K_t$ is smooth then $\{\partial K_t\}_{t\in[t-\varepsilon,t+\varepsilon]}$ is a smooth foliation for some small $\varepsilon>0$, and if $\partial K_t$ is not smooth, then $\{K_t\}_{t\in[t-\varepsilon,t+\varepsilon]}$ is a level set flow for some small $\varepsilon$ (see Appendix C).
\end{itemize}

The following theorem is a more precise version of Theorem \ref{a} for manifolds thick at infinity.
\begin{theo}\label{b}
Let $(M,g)$ be an $(n+1)$-dimensional complete manifold with $2\leq n\leq 6$, thick at infinity, and let $B\subset M$ be a compact domain. Then 
\begin{enumerate}
\item either $M$ contains a saddle point minimal hypersurface intersecting $B$,
\item or there is an embedded closed locally area minimizing hypersurface $\Sigma_B\subset (M,g)$ (maybe empty) such that $B\backslash \Sigma_B$ has a singular weakly mean convex foliation.
\end{enumerate}
\end{theo}

\begin{exemple} Before starting the proof of Theorem \ref{b}, we give a short example suggested by a referee to illustrate the statement. Consider a metric on $S^2\times \mathbb{R}$ of the form $g:= f^2(t) g_{S^2}  \oplus dt^2$ where  $g_{S^2}$ is the unit round metric on $S^2$ and $f:\mathbb{R}\to (0,1)$ is  a smooth function with $f'(0)=0$. 

If $f$ is non-increasing on $(-\infty,0)$ and $B:=S^2\times[-1,0]$ then Item (2) of Theorem \ref{b} holds with $\Sigma_B=\varnothing$ and the weakly mean convex foliation is given by $K_t:=[t-1,0]$ with $0\leq t<1$.

If $f$ is moreover non-decreasing on $(0,\infty)$, and $B:=[-1,1]$ then Item (2) of Theorem \ref{b} holds with $\Sigma_B=S^2\times \{0\}$ and the weakly mean convex foliation is given by $K_t:=[t-1,1-t]$ with $0\leq t<1$.

If $f$ is a constant then both Items (1) and (2) hold (in particular, the possibilities of Theorem \ref{b} are not mutually exclusive).

\end{exemple}

\begin{proof}[Proof of Theorem \ref{b}] 
We suppose that $M$ is thick at infinity. Let us resume from the proof of Theorem \ref{a}. Recall that $B$ is contained in a ball $B_r(p)$, itself contained in $D$. We modify the metric $g$ near $\partial D$ to get $g_D$ so that $\partial D$ becomes mean convex for $g_D$. For any subdomain $D'\subset (D,g_D)$ containing $B$ and with $C^{1,1}$ weakly mean convex boundary, and any minimal hypersurface $S\subset \interior(D')\backslash B$, we consider the metric completion $\mathbf{D}$ of $D'\backslash S$, naturally endowed with a metric still denoted by $g_D$.

We found a solution to a constrained minimization problem, yielding a compact manifold $B''$ depending on $\mathbf{D}$, such that $B\subset B''$. If there is a saddle point minimal hypersurface $\Gamma_1$ included in $B''\backslash B$ then there is a thin mean concave neighborhood $N_{\Gamma_1}$ of $\Gamma_1$ embedded in $B''\backslash B$ and the metric completion of $B''\backslash N_{\Gamma_1}$ gives a new manifold $\mathbf{D}_2$ where we can solve the constrained minimization problem, get a manifold $B''_2$. If there is a saddle point minimal hypersurface $\Gamma_2$ embedded in $B''_2\backslash B$, we repeat the process and get $B''_3$. If for a $j_0$, $B''_{j_0}\backslash B$ does not contain a saddle point minimal hypersurface, then we define $B''_\infty :=B''_{j_0}$. Suppose that the sequence of $B''_j$ is infinite, the $n$-volume of $B''_j$ is strictly decreasing and we can suppose that we chose the saddle point minimal hypersurfaces in such a way that $\lim_{j\to \infty} \Vol_n(B''_j)$ is as small as possible, among all choices of sequence of saddle points $\Gamma,\Gamma_2, \Gamma_3,...$. Then $B''_j$ subsequently converges in the Gromov-Hausdorff topology to a closed mean convex set $B''_\infty$ enodwed with the metric $g_D$, in which $(B,g)$ is isometrically embedded, and by minimality of its volume, there is no saddle point minimal hypersurface embedded in $B''_\infty \backslash B$. Since the $C^{1,1}$ boundary of each $B''_j$ has a minimizing property and its $n$-volume is at most $\Vol_n(\partial B)$, by compactness \cite{Sharp}, $B''_\infty$ has a $C^{1,1}$ mean convex boundary (in the sense of level set flow), smooth and minimal outside of $B$. Let $\Phi:[0,1] \to \mathcal{Z}_n(B; \mathbb{Z}_2)$ a sweepout of $B$ with $\Phi(0)=0$, and $\Phi(1)$ is $\partial B$ with multiplicity one.  \newline

\begin{claim}

 \begin{itemize}
\item Either there is a locally area minimizing minimal hypersurface $\Sigma_{B,D} \subset (B''_\infty,g_D)$ of $n$-volume at most $\Vol_n(\partial B)$ such that $B\backslash\Sigma_{B,D}$ has a singular weakly mean convex foliation,
\item or there is a saddle point minimal hypersurface $\Gamma_D$ of index at most one and $n$-volume at most $\Vol_n(\partial B) + \max_{t\in[0,1]} \mathbf{M}(\Phi(t))+2$, embedded in $B''_\infty$ intersecting $B$. \newline
\end{itemize}
\end{claim}

Let us prove this statement. Note that in this claim, the metric on $B''_\infty$ is the one induced by $g_D$ and that $\Vol_n(\partial B) + \max_{t\in[0,1]} \mathbf{M}(\Phi(t))+2$ is independent of $D$. Recall that $\Vol_n(\partial B''_\infty)\leq \Vol_n(\partial B)$. If a component $T'$ of $\partial B''_\infty$ is smooth and does not touch $B$, it is a stable minimal hypersurface. It has to satisfy Condition [M] inside $B''_\infty$ because otherwise there would be saddle point minimal hypersurfaces arbitrarily close to $T'$ and outside of $B$, so we could repeat the minimization process and contradict the volume minimality of $B''_\infty$. Since $T'$ cannot be degenerate stable of Type III for the same reason, $T'$ is actually either strictly stable or degenerate stable of Type II. If a component $T$ of $\partial B''_\infty$ is smooth and touches $B$, if $T$ is a saddle point minimal hypersurface, the claim is true. If $T$ is not a saddle point minimal hypersurface and does not satisfy Condition [M] inside $B''_\infty$, then there are saddle point minimal hypersurfaces close to $T$ in the $C^\infty$ topology with index at most one (see Lemma \ref{degenerate} (4)), that intersect $B$ and the claim is still true in that case. The remaining possibilities are that $T$ is either strictly stable or degenerate stable of Type II or III as a minimal hypersurface inside $B''_\infty$ (by convention a minimal boundary component is not of Type I). In what follows, we will suppose that all smooth components of $\partial B''_\infty$ are either strictly stable or degenerate stable of Type II or III.
%As in the proof of Theorem \ref{a}, the width of $B''$ only depends of $B$ and is bounded by a constant $\hat{C}>0$ independent of the choice of $\mathbf{D}$, $B''$. We can suppose that

%Let $V$ be the union of smooth minimal components of $\partial B''_\infty$ which are unstable or degenerate stable of Type III, and let  $N_V$ be a neighborhood of $V$ foliated by weakly mean convex hypersurfaces (Appendix A). 

Some boundary components of $B''_\infty$ may be degenerate stable of Type III inside $B''_\infty$. We can first push these components by hand slightly inside of $B''_\infty$ and preserving their weak mean convexity, then we can run the level set flow (Appendix C) to the corresponding subset of $B''_\infty$. We obtain a family of closed sets $\{A_t\}_{t\geq 0}$ with $A_0=B''_\infty$. The level set flow will eventually get ``stuck''  at a stable minimal hypersurface $S$ (which can be empty),
%, but cannot be degenerate stable of Type III by the maximum principle), 
i.e. the level sets converge. There are a few possibilities. 

Either $\bigcap_{t\geq 0}{A}_t\cap\interior(B) = \varnothing$ and the claim is verified by taking $\Sigma_{B,D}=\varnothing$. The set $B$ has a singular weakly mean convex foliation. 

Or $\bigcap_{t\geq 0}{A}_t\cap\interior(B) \neq \varnothing$ and the stable minimal hypersurface $S$ intersects the interior of $B$. If we suppose that $S$ does not satisfy Condition [M], some saddle point minimal hypersurfaces close to $S$ intersect $B$ and they have the right index and $n$-volume bounds, so the claim is true in this case.

%if this occurs for any $D$, $\mathbf{D}$, $B''_\infty$ for take a sequence of larger and larger $D$ ($r\to \infty$), we get a sequence of saddle points $\Gamma^{(k)}$ of controlled index and $n$-volume. Taking a subsequence limit, we get a minimal hypersurface intersecting $B$, which is closed by thickness at infinity. Actually the metric $g_D$ on each $D$ coincides with $g$ on larger and larger regions as $D$ becomes larger so for $k$ large, $\Gamma^{(k)}$ is a saddle point minimal hypersurface for the original metric $g$ and we are done.

In the case $\bigcap_{t\geq 0}{A}_t\cap\interior(B) \neq \varnothing$ and $S$ satisfies Condition [M], each component of $S$ intersecting the interior of $B$ is either strictly stable or degenerate stable of Type I or II (for the metric $g_D$). If a component $S_0$ of $S$ is of Type I, then the level set flow approaches $S_0$ from above, and it is possible to prolongate the foliation $\{A_t\}_{t\geq 0}$ by hand beyond $S_0$, so that around $S_0$ the foliation is smooth weakly mean convex. We perform this around each such component of $S$. After this foliation goes beyond $S$, we can run the level set flow again, and we can repeat that process, extending the foliation whenever the level set flow converges to a degenerate stable minimal hypersurface of Type I. This way, we construct a foliation $\{\tilde{A}_t\}_{t\geq 0}$, that we can suppose sweeps out a region of $B$ of maximal volume. We can make sure to not get trapped at a degenerate stable minimal hypersurface of Type I by a compactness argument. This construction can only stop if 
$$\bigcap_{t\geq 0}\tilde{A}_t \cap \interior(B) = \varnothing$$
 (then the claim is true), 
%or if a saddle point minimal hypersurface intersects $B$,
or if the foliation arrives at a non-empty minimal hypersurface intersecting $\interior(B)$: in that case, the components intersecting $\interior(B)$ are strictly stable or degenerate stable of Type II. 

We will assume the last case in the remaining of the proof, namely $\bigcap_{t\geq 0}\tilde{A}_t \cap \interior(B) \neq \varnothing$. We write $B_{\mathrm{core}}:=\bigcap_{t\geq 0}\tilde{A}_t $. It is a union of an $(n+1)$-dimensional compact manifold $B_{\mathrm{core}}^{(1)}$ and a minimal hypersurface $B_{\mathrm{core}}^{(2)}$.
%, and $\Vol_n(\partial B_{\mathrm{core}}^{(1)}) + 2\Vol_n(B_{\mathrm{core}}^{(2)})$ is bounded by $\Vol_n(\partial B)+1$.
By the above remark, the connected components of $B_{\mathrm{core}}^{(2)}$ intersecting $\interior(B)$ are strictly stable or degenerate stable of Type II inside $(B''_\infty,g_D)$ (hence locally area minimizing). Similarly, all the boundary components of $B_{\mathrm{core}}^{(1)}$ are either strictly stable or degenerate stable of Type II inside $B_{\mathrm{core}}^{(1)}$ (hence locally area minimizing inside $B_{\mathrm{core}}^{(1)}$). By monotonicity along the flow, 
$$\Vol_n(\partial B_{\mathrm{core}}^{(1)}) + 2\Vol_n(B_{\mathrm{core}}^{(2)}) \leq \Vol_n(\partial B).$$ 
%We can take a sequence of subdomains $D'_i$ such that the corresponding volumes of $B_{\infty,i}\cap B$ converge to the infimum of volumes of the form $\Vol_n(B_{\infty}\cap B)$ for some $D'$. Up to subsequence, a Gromov-Hausdorff limit of $B_{\infty,i}$ exists and we call it $R$. 
Now two cases can happen:
\begin{itemize}
\item either the $(n+1)$-dimensional volume of $B_{\mathrm{core}}\cap B$ is $0$, then define $\Sigma_{B,D}$ to be the union of the connected components of $B_{\mathrm{core}}^{(2)}$ intersecting $\interior(B)$. This hypersurface $\Sigma_{B,D}$ is a closed embedded locally area-minimizing hypersurface in $ (B''_\infty,g_D)$ of $n$-volume at most $\Vol_n(\partial B)$, so that $B\backslash \Sigma_{B,D}$ has a weakly mean convex foliation, 
\item or the $(n+1)$-dimensional volume of $B_{\mathrm{core}}\cap B$ is positive so that $B_{\mathrm{core}}^{(1)}$ is a non-trivial $(n+1)$-dimensional compact manifold; in that case Theorem \ref{sadddle} produces a saddle point $\Gamma_D\subset (B_{\mathrm{core}}^{(1)},g_D)$ of Morse index at most one, which intersects $B$ by volume minimality of $B''_\infty$ (see beginning of proof). Its $n$-volume is at most 
$$\Vol_n(\partial B) + \max_{t\in[0,1]} \mathbf{M}(\Phi(t))+2.$$
Indeed we already saw in the proof of Theorem \ref{a} that there is a sweepout $\Psi$ of $\partial B''_\infty$ such that $\max_{t\in[0,1]}\mathbf{M}(\Psi(t)) \leq  \max_{t\in[0,1]} \mathbf{M}(\Phi(t)) +1$, and if $\Omega_t$ is the open set bounded by $\Psi(t)$ inside $B''_\infty$, $\partial (\Omega_t\cap B_{\mathrm{core}}^{(1)})$ gives a sweepout of $B_{\mathrm{core}}^{(1)}$ and its width is clearly bounded above by
$$\max_{t\in[0,1]} \mathbf{M}(\Psi(t)) + \Vol_n(\partial B_{\mathrm{core}}^{(1)}) \leq  \max_{t\in[0,1]} \mathbf{M}(\Phi(t))+1 + \Vol_n(\partial B).$$
\end{itemize}

In all cases we just checked that the claim is true.

If we take $r\to \infty$, larger and larger domains $D\supset B_r(p)$, and if the first bullet of the claim always occurs for $r$ large, then by thickness at infinity and \cite{Sharp}, the diameter of $\Sigma_{B,D}$ is uniformly bounded so for $r$ large enough $\Sigma_B:=\Sigma_{B,D}$ is a closed embedded locally area minimizing hypersurface for the original metric $g$, and $B\backslash \Sigma_B$ has a weakly mean convex foliation. This is item (2) of the theorem.

If for a sequence of radii $r_j\to \infty$ and choice of $(D_j,g_{D_j})$, the second bullet occurs, then we take a converging subsequence of saddle points $\Gamma_{D_j}$ (see \cite{Sharp}), and since $(M,g)$ is thick at infinity, $\Gamma_j$ is actually a saddle point for the original metric $g$ if $j$ is large and it intersects $B$. This is item (1) of the theorem.

\end{proof}

We end this subsection with a boundary version of Theorem \ref{b}, which will be used later. Its proof follows by inspecting the proof of Theorem \ref{b}.

In the following statement, a compact domain $B$ inside a complete $(n+1)$-manifold $M$ with boundary is by definition a compact $(n+1)$-submanifold of $M$ with smooth boundary. The definition of ``singular weakly mean convex foliation'' is the same as the one given at the start of Subsection \ref{womm}.

\begin{theo}  \label{itsok}
Let $({M},g)$ be an $(n+1)$-dimensional  complete manifold with smooth boundary with $2\leq n\leq 6$, thick at infinity. Suppose that the boundary $\partial M$ is a locally finite union of closed minimal hypersurfaces. Then for any compact domain $B\subset M$, 
\begin{enumerate}
\item either $M$ contains in its interior a saddle point minimal hypersurface intersecting $B$,
\item or there is an embedded closed locally area minimizing hypersurface $\Sigma_B\subset (M,g)$ (maybe empty) such that $B\backslash \Sigma_B$ has a singular weakly mean convex foliation. 
\end{enumerate}
In case (2), if $\{K_t\}_{[0,1)}$ is the singular weakly mean convex foliation, then for any connected component of $S$ of $\partial M$ included in $B$ and locally area minimizing in $B$, for any $t\in[0,1)$, we have $S\subset K_t$.
\end{theo}

%\begin{remarque} \label{itsok}
%By inspecting the proof of Theorem \ref{b}, we see that it also holds if $(M,g)$, thick at infinity, has a non-empty boundary and if the components of $\partial M$ are closed minimal hypersurfaces: for any compact domain with smooth boundary $B\subset M$, the dichotomy of Theorem \ref{b} holds true. Here if $B$ is in the case (2), i.e. it has a weakly mean convex foliation outside of a minimal hypersurface $\Sigma_B$, and if a component $S$ of $\partial M$ is included in $B$, then $S$ is embedded either in one leaf of the foliation or in $\Sigma_B$. In this case if moreover $S$ is locally area minimizing in $B$, then $S\subset \Sigma_B$.
%\end{remarque}

\subsection{Min-max in a manifold generated by a saddle point minimal hypersurface with the level set flow} \label{Appendix D}

Let $(M^{n+1},g)$ be a complete manifold thick at infinity. We assume in this subsection that the metric $g$ satisfies Condition [M]. 

We will often use $\{\Sigma_t\}_{t\in[0,1]}$ or similar notations for $1$-sweepouts of a region $R$ of $M$. Sometimes we will define $\Sigma_t$ as a hypersurface, even though rigorously speaking, each $\Sigma_t$ should be a current in $\mathcal{Z}_{n,rel}(R,\mathbb{Z}_2)$ (see \cite[Definition 1.20]{Alm1}, \cite[2.2]{LioMaNe}). For simplicity we will also denote by $\Vol_n(\Sigma_t)$ its mass instead of using $\mathbf{M}(\Sigma_t)$.  See Appendix B for definitions of sweepouts (we will use in this section the notion of $p$-widths $\omega_p$, note that $\omega_1$ is different from the width $W$ used in previous sections because they correspond to different kinds of sweepouts).

Let $\Gamma$ be a saddle point minimal hypersurface in $(M,g)$. We will treat the case where $\Gamma$ is 2-sided for simplicity, but the case where it is 1-sided is completely analogous. $\Gamma$ has to be either unstable, or degenerate stable of Type III, and so we can find a neighborhood $N_\Gamma$ of $\Gamma$ and a diffeomorphism $\phi:\Gamma\times(-\delta_1,\delta'_1) \to N_\Gamma$ such that $\phi(\Gamma\times\{0\}) = \Gamma$, the mean curvature of $\phi(\Gamma\times\{s\})$ is either vanishing or non-zero pointing away from $\Gamma$, and $\phi(\Gamma\times\{-\delta_1\})$, $\phi(\Gamma\times\{\delta'_1\})$ have non-zero mean curvature.

Let $K_0:=M\backslash N_\Gamma$. Let $B_k$ be an exhausting sequence of compact domains with smooth boundary, containing $N_\Gamma$. We choose a metric $g_k$ on $B_k$ verifying:
\begin{enumerate}
\item $\partial B_k$ is minimal strictly stable with respect to $g_k$,
\item  $\|g_k-g\|_{C^0}\leq \mu_k$, where the right-hand side  is the $C^0$ distance between $g$ and $g_k$ on $B_k$, computed with $g$, and $\mu_k$ converges to zero as $k\to\infty$,
\item $g=g_k$ except in a $1/k$-neighborhood of $\partial B_k$.
\end{enumerate}  

For each integer $k$, let $K^{(k)}_0:=(B_k\cap K_0,g_k)$. It is a weakly mean convex subset of $(B_k,g_k)$. We run the level set flow to $K^{(k)}_0$ and get a family $\{K^{(k)}_t\}_{t\geq0}$. Denote by $X_k$ the metric completion of $B_k\backslash \bigcap_{t\geq0} K^{(k)}_t$. $X_k$ is a compact manifold with closed minimal stable boundary components. Each of them has an $n$-volume bound coming from monotonicity properties of the level set flow and is minimal locally area minimizing inside $X_k$. Besides $\partial K^{(k)}_t$ locally ``converges'' to $\partial X_k$ smoothly (see Appendix C). 
%By Condition [M], it means that each component of $\partial X_k$ is either strictly stable or degenerate stable of Type II in $X_k$. 
Since interior points of $X_k$ are identified with points of $M$, there is a natural map $X_k \to M$.

As $k\to\infty$, $X_k$ subsequently converges in the Gromov-Hausdorff topology to a manifold $X$ that is naturally endowed with the metric $g$. $X$ is a (maybe non-compact) manifold with minimal stable boundary components. Their total $n$-volume is finite, each of them is compact because $(M,g)$ is thick at infinity, and minimal locally area minimizing inside $X$. By Condition [M], it means that each component of $\partial X$ is either strictly stable or degenerate stable of Type II. 
We will refer to this construction by saying that $X$ is \textit{generated} by $\Gamma$, $K_0$ in $(M,g)$.

Let $\mathcal{C}(X_k)$ denote the result of gluing the compact manifold $X_k$ to straight half-cylinders $(\partial X_k \times [0,\infty), g\big|_{\partial X_k} \oplus dt^2)$ along $\partial X_k$. The final metric (still denoted by $g_k$) is Lipschitz continuous around $\partial X_k$ in general. 
Let us define 
$$\tilde{\omega}_p(X_k,g_k):= \omega_p(\mathcal{C}(X_k),g_k)$$
where the $p$-widths $\omega_p$ of a possibly non-compact manifold are defined in \cite[Definition 8]{AntoineYau}. Note that if $X_k$ has no boundary then $\tilde{\omega}_p(X_k,g)=\omega_p(X_k,g)$. It is clear that the first width of $(\mathcal{C}(X_k),g_k)$ is finite. Indeed let $\{\Sigma_t\}_{t\in[0,1]}$ be an arbitrary $1$-sweepout of $X_k\backslash K^{(k)}_0=X\backslash K_0$ with $\Sigma_0=0$ and $\Sigma_1=\partial K_0$, let $\{S^{(k)}_t\}_{t\in[0,\infty)}$ with
$S^{(k)}_t:=\partial X_k \times \{t\} \subset \mathcal{C}(X_k)$. Then the three families $\{\Sigma_t\}_{t\in[0,1]}$, $\{\partial K^{(k)}_t\}_{t\geq 0}$, $\{S^{(k)}_t\}_{t\geq 0}$ concatenated gives an explicit $1$-sweepout of any bounded subset of $(\mathcal{C}(X_k),g_k)$ and 
\begin{align*}
\omega_1(\mathcal{C}(X_k),g_k) & \leq \max\{\sup_{t\in[0,1]}{\Vol_n(\Sigma_t)},\sup_t{\Vol_n(\partial K^{(k)}_t)}, \sup_{t}{\Vol_n(S^{(k)}_t)}\} \\
& \leq \max\{\sup_{t\in[0,1]}{\Vol_n(\Sigma_t)},\Vol_n(\partial K_0)\} <\infty.
\end{align*}
%see proof of Theorem 10 in \cite{AntoineYau} for a similar construction.
In \cite[proof of Theorem 5.1, Claim 5.6]{MaNeinfinity}, it is explained that given a compact manifold $Y$, given a $1$-sweepout $\Phi_1:[0,1]\to \mathcal{Z}_{n,rel}(Y,\mathbb{Z}_2)$ with ``no concentration of mass'' (see \cite{MaNeinfinity}), with 
$$\sup_{t\in[0,1]}\mathbf{M}(\Phi_1(t))\leq C_1,$$
one can construct $p$-sweepouts $\Phi_p:\mathbb{R}P^p\to \mathcal{Z}_{n,rel}(Y,\mathbb{Z}_2)$ for any $p$, satisfying $\sup_{t\in \mathbb{R}P^p}\mathbf{M}(\Phi_p(t))\leq p.C_1$. In particular, we have the following general \newline

\textbf{Fact:} if the first width is finite, then all the widths are finite and the $p$-width is bounded by $p$ times the first width. 
\newline

Using that fact, it is simple to check that for each $p$, $\tilde{\omega}_p(X_k,g_k)$ is bounded between to positive constants independent of $k$. We choose a subsequence of $B_k$ (that we do not rename) in a way that $\tilde{\omega}_p(X_k,g_k)$ converges for each $p$ and we define
$$\tilde{\omega}_p(X,g):=\lim_{k\to \infty}\tilde{\omega}_p(X_k,g_k).$$

For each $k$ we also set
 $$\mathcal{A}(\Gamma,g_k):=\max \{\Vol_n(\hat{C}) ; \hat{C} \text{ component of } \partial X_k\},$$
 where the $n$-volume of $\hat{C}$ is computed with $g_k$. By the monotonicity property of the level set flow, the sequence $\mathcal{A}(\Gamma,g_k)$ is bounded. We can assume that this sequence converges (by taking a subsequence if necessary) and we define
$$\mathcal{A}(\Gamma,g) : =\lim_{k\to \infty}  \mathcal{A}(\Gamma,g_k).$$
In what follows, $\partial X$ (resp. $X$, $\mathcal{A}(\Gamma,g)$) will roughly play the role of $\partial U$ (resp. $U$, the $n$-volume of the largest boundary component of $U$) in \cite[Section 2]{AntoineYau}. For simplicity, we will write $\mathcal{A}$, $\mathcal{A}_k$, $\tilde{\omega}_p(g)$, $\tilde{\omega}_p(g_k)$, for $\mathcal{A}(\Gamma,g)$, $\mathcal{A}(\Gamma,g_k)$, $\tilde{\omega}_p(X,g)$, $\tilde{\omega}_p(X_k,g_k)$ respectively.

By \cite{MaNeindexbound} and \cite[Theorem 10]{AntoineYau}, for any fixed $p$, there is a stationary integral varifold $V_k$ with support a closed minimal hypersurfaces of Morse index at most $p$ embedded inside the interior of $(X_k,g_k)$, each component intersecting $M\backslash K_0$ by the maximum principle (\cite[Theorem 3.5]{WhiteregMCF}), such that the total mass of $V_k$ is $\tilde{\omega}_p(g_k)$. Multiplicities of the 1-sided components of $\spt(V_k)$ are even. Making $k\to \infty$ and taking a subsequence, $V_k$ converges in the varifold sense to a stationary integral varifold $V_\infty$ with support a minimal hypersurface 
%of Morse index at most $p$ 
embedded inside the interior of $X$ (\cite{Sharp}), and closed because $(M,g)$ is thick at infinity.  Moreover, since $g=g_k$ except very close to $\partial B_k$, $V_\infty$ is stationary for the original metric $g$ if $k$ is large.
In other words, 
%by the Frankel property satisfied by $X$, 
we just proved the following
\begin{prop}\label{etouiminmax}
For all $p$ there are disjoint connected closed embedded minimal hypersurface $\Gamma^{(p)}_i\subset (\interior(X),g)$ ($1\leq i\leq J$) and integers $m_{p,i}$ ($m_{p,i}$ is even when $\Gamma^{(p)}_i$ is 1-sided) satisfying
$$\tilde{\omega}_p(X,g) = \sum_{i=1}^J m_{p,i} \Vol_n(\Gamma^{(p)}_i).$$
\end{prop}

%Note that $\mathcal{A}_k = \lim_{t\to \infty}\sup\{ \Vol_n(\hat{C}); \hat{C} \text{ is a connected component of } \partial K^{(k)}_t\}$ where the $n$-volume is computed with $g_k$. 

We denote by $\Vol_n(.,g_k)$ the $n$-volume computed with $g_k$. If $\Omega$ is a compact domain of $M$, for all $k$ it induces a compact subset $\Omega'$ in $X_k$ by pulling back with the natural map $X_k\to M$, and since $X_k\subset \mathcal{C}(X_k)$, $\Omega'$ in turn induces a compact subset of $\mathcal{C}(X_k)$ that we call $\mathfrak{i}_k(\Omega)$. For a region $R\subset\mathcal{C}(X_k)$ with rectifiable boundary, we will denote by $\mathcal{Z}_{n,rel}(R;\mathbb{Z}_2)$ the space of relative cycles mod 2 in the closure of $R$ (see \cite[Definition 1.20]{Alm1}, \cite[2.2]{LioMaNe}).

Before continuing, let us remark the following.
\begin{lemme} \label{tong}
For all $\bar{\varepsilon}>0$, there is a compact domain $\Omega\subset M$ so that for any $k$ large enough, 
%and any compact domain $D\subset \mathcal{C}(X_k)$, 
the region $X_k\backslash\mathfrak{i}_k(\Omega) \subset \mathcal{C}(X_k)$ has a $1$-sweepout $\{S_t\}_{t\in[0,1]}$ where $S_t\in \mathcal{Z}_{n,rel}(X_k\backslash\mathfrak{i}_k(\Omega); \mathbb{Z}_2)$ for each $t\in[0,1]$, satisfying
%\label{shouldbe}
$$\sup_{t\in[0,1]}\{\Vol_n(S_t,g_k)\} \leq \mathcal{A}+\bar{\varepsilon}.$$
\end{lemme}

\begin{proof}

Recall that $N_\Gamma$ is the neighborhood of $\Gamma$ introduced at the beginning of this subsection.

For any $k$, let $Y_{{k}}$ be the compact domain image of $X_{{k}}$ under the natural map $X_{{k}}\to M$. By definition of $\mathcal{A}$ and the properties of $g_k$, for $\bar{k}$ large enough $\mathcal{A}_{\bar{k}}\leq \mathcal{A}+\bar{\varepsilon}/2$. Then the boundary components of $Y_{\bar{k}}$ are closed 2-sided and each of them has $n$-volume at most $\mathcal{A}_{\bar{k}}$. For $k$ large enough compared to $\bar{k}$, $B_k$ contains $Y_{\bar{k}}$ in its interior. Let $\hat{C}$ be any component of $\partial Y_{\bar{k}}$, that we consider as a cycle in $\mathcal{Z}_n(B_k;\mathbb{Z}_2)$. Consider the following minimization problem with constraint (we used a similar minimization problem in the proof of Theorem \ref{a}): minimize $\Vol_n(\hat{C}_1,g_k)$ among cycles $\hat{C}_1$ such that 
\begin{itemize}
\item there is a path $\{\hat{C}_t\}_{t\in[0,1]}$ continuous in the $\mathbf{F}$-topology with $\hat{C}_0=\hat{C}$, 
\item the support of each $\hat{C}_t$ is in $B_k \backslash N_\Gamma$,
\item $\Vol_n(\hat{C}_t,g_k)\leq \mathcal{A}+\bar{\varepsilon}$ for all $t\in[0,1]$.
\end{itemize}
A solution (i.e a minimizer $\hat{C}_1$ and a path $\{\hat{C}_t\}_{t\in[0,1]}$) exists by compactness in the flat topology, and interpolation results \cite[Proposition A.2]{MaNeindexbound}. The support of $\hat{C}_1$ is a stable closed minimal hypersurface. Consider the image of $\{\hat{C}_t\}_{t\in[0,1]}$ by the Almgren map, $\mathcal{A}(\{\hat{C}_t\})\in \mathbf{I}_{n+1}(B_k;\mathbb{Z}_2)$ (see Appendix B). We call it $\mathcal{R}(\hat{C})$ for simplicity and to avoid confusion with the notation $\mathcal{A}=\mathcal{A}(\Gamma,g)$. 

Now we can repeat that construction for any component of $\partial Y_{\bar{k}}$. Consider $Y_{\bar{k}}$ as an element of $\mathbf{I}_{n+1}(B_k;\mathbb{Z}_2)$ and let 
$$Z_k:= Y_{\bar{k}} +\sum_{\hat{C} \text{ component of } \partial Y_k} \mathcal{R}(\hat{C}) \in \mathbf{I}_{n+1}(B_k;\mathbb{Z}_2).$$
The support of $Z_k$ is non-empty because it contains $N_\Gamma$ and is a compact $(n+1)$-dimensional region with weakly mean convex boundary for $g_k$. In fact, since the minimization problems for components of $\partial Y_{\bar{k}}$ are independent of one another, the different $\mathcal{R}(\hat{C})$ can a priori intersect each other so $\partial Z_k$ is not necessarily smooth but at least the interior of $\spt(Z_k)$ is locally the disjoint union of intersections of domains with smooth minimal boundary. Hence $Z_k$ serves as a barrier domain for the level set flow, which means that if $Y_k$ is the image of $X_k$ by the natural map $X_k\to M$, then by a slight abuse of notations
$$Y_k\subset Z_k.$$
Set 
$$\Omega:=Y_{\bar{k}}.$$
By concatenating the pathes $\{\hat{C}_t\}$ one after the other (for $\hat{C}$ component of $\partial Y_{\bar{k}}=\Omega$) and restricting these cycles to $Y_k \backslash \Omega$, we construct a $1$-sweepout $\{S^{(1)}_t\}$ of $Y_k \backslash \Omega$ such that $\sup_t \Vol_n(S^{(1)}_t, g_k) \leq \mathcal{A}+\bar{\varepsilon}$, which pulls back to a sweepout called $\{S_t\}$ of $X_k \backslash \mathfrak{i}_k(\Omega)\subset \mathcal{C}(X_k)$, where $\mathfrak{i}_k(\Omega)\subset X_k \subset \mathcal{C}(X_k)$ has been defined just before this lemma and where $S_t\in \mathcal{Z}_{n,rel}(X_k\backslash\mathfrak{i}_k(\Omega); \mathbb{Z}_2)$ for each $t\in[0,1]$. This sweepout still satisfies
$$\sup_{t\in[0,1]}\{\Vol_n(S_t,g_k)\} \leq \mathcal{A}+\bar{\varepsilon}$$
so the lemma is proved.

\end{proof}

We want to show the analogue of \cite[Theorem 9]{AntoineYau} for $\tilde{\omega}_p(g)$:
\begin{prop}\label{flore}
$$
\forall p\geq 1\quad \tilde{\omega}_{p+1}(g) - \tilde{\omega}_p(g) \geq \mathcal{A}, \quad \lim_{p\to\infty} \frac{\tilde{\omega}_p(g)}{p} = \mathcal{A}.
$$
\end{prop}
\begin{proof}
By Theorem 9 of \cite{AntoineYau}, for all $p\geq 1$ fixed, $\tilde{\omega}_{p+1}(g_k) - \tilde{\omega}_p(g_k) \geq \mathcal{A}_k$ and $\tilde{\omega}_{p}(g_k) \geq p.\mathcal{A}_k$, hence passing to the limit when $k\to\infty$, 
\begin{equation} \label{fleur}
\tilde{\omega}_{p+1}(g) - \tilde{\omega}_p(g) \geq \mathcal{A},
\quad \tilde{\omega}_{p}(g) \geq p.\mathcal{A}.
\end{equation}
Let $\bar{\varepsilon}>0$ be fixed, and let $k$ be large enough so that Lemma \ref{tong} is satisfied and $\mathcal{A}_k\leq \mathcal{A}+\bar{\varepsilon}$.
Let $D$ be a compact region of $\mathcal{C}(X_k)$ containing $X_k$. The region $D\backslash X_k$ has a $1$-sweepout $\{T_t\}_{t\in[0,1]}$ with 
\begin{equation} \label{shouldbe}
\sup_{t}\Vol_n(T_t) \leq \mathcal{A}+\bar{\varepsilon}.
\end{equation}
To explain that, let $\Sigma_1,...,\Sigma_m$ be the components of $\partial X_k$, and let $L$ be a large number. Consider the function 
$$f_L: \partial X_k \times [0,L]\to \mathbb{R}$$
$$f_L(x,t) := (j-1)L +t \text{  if } (x,t)\in\Sigma_j\times[0,L].$$
Then the level sets of $f_L$ gives a $1$-sweepout of $D\backslash X_k$ as desired if $L$ is large enough so that $X_k \cup \partial X_k \times [0,L]$ contains $D$, i.e. we consider 
$$T_t:=f_L^{-1}(t) \cap (D\backslash X_k) \in \mathcal{Z}_{n,rel}(D\backslash X_k;\mathbb{Z}_2).$$

Now by concatenating $\{S_t\}_t$ from Lemma \ref{tong} with $\{T_t\}_{t\in[0,L]}$, we get a $1$-sweepout $\{U_s\}_{s\in[0,L+1]}$ of the disjoint union of $X_k\backslash \mathfrak{i}_k(\Omega) \sqcup D\backslash X_k$ satisfying
$$\sup_{t}\Vol_n(U_t) \leq \mathcal{A}+\bar{\varepsilon}.$$ 
Note that $U_s$ is not in general a cycle in $\mathcal{Z}_n(\mathcal{C}(X_k);\mathbb{Z}_2)$ because the relative cycles $S_t$, $T_t$ may have boundaries, so $\{U_s\}$ is not a $1$-sweepout of the union $X_k\backslash \mathfrak{i}_k(\Omega) \cup D\backslash X_k$ considered as a subset of $\mathcal{C}(X_k)$. Next we will see how to form a genuine sweepout made of cycles out of $\{U_s\}$.

Again by the proof of \cite[Theorem 5.1, Claim 5.6]{MaNeinfinity}, we can construct from $\{U_s\}$ for each $p\geq 1$ a $p$-sweepouts $\Phi_p$ (with domain $\mathbb{R}P^p$) of the disjoint union $X_k\backslash \mathfrak{i}_k(\Omega) \sqcup D\backslash X_k$ such that 
$$\sup_{x\in\mathbb{R}P^p} \mathbf{M}(\Phi_p(x),g_k) \leq p.(\mathcal{A}+\bar{\varepsilon}).$$
On the other hand, since a $p$-sweepout of $\Omega \subset M$ lifts to a $p$-sweepout of $\mathfrak{i}_k(\Omega) \subset \mathcal{C}(X_k)$, \cite[Theorem 5.1]{MaNeinfinity}, there is a sequence of sweepouts $\Psi_p$ of $\mathfrak{i}_k(\Omega)$ (with domain $\mathbb{R}P^p$) so that for a constant $\underline{C}_0=\underline{C}_0(\Omega, g)$ (independent of $k$), 
$$\forall p\geq 1 \sup_{x\in\mathbb{R}P^p} \mathbf{M}(\Psi_p(x),g_k) \leq  \underline{C}_0p^{\frac{1}{n+1}}.$$
As explained in the proof of Theorem 9 in \cite{AntoineYau}, we can glue the $p$-sweepouts $\Phi_p$ and $\Psi_p$ parametrized by $\mathbb{R}P^p$ using \cite{LioMaNe} and get a new $p$-sweepout $\hat{\Phi}_p$ of $D \subset \mathcal{C}(X_k) $ with domain $\mathbb{R}P^p$ such that:
\begin{align*}
\sup_{x\in\mathbb{R}P^p} \mathbf{M}(\hat{\Phi}_p(x),g_k) \leq &\quad p.(\mathcal{A}+\bar{\varepsilon})+ \underline{C}_0 p^{\frac{1}{n+1}} + \Vol_n(\partial (\mathfrak{i}_k(\Omega)), g_k) \\
&+\Vol_n(\partial(X_k\backslash \mathfrak{i}_k(\Omega)), g_k)+\Vol_n(\partial X_k,g_k).
\end{align*}
As $\Omega$ is fixed and $\partial X_k$ has $n$-volume less than twice that of  $\Gamma$ by monotonicity of the level set flow, the last three $n$-volumes above are bounded by a constant $\underline{C}_1$ independent of $k$.

Taking $D$ arbitrarily large, we get for all $k$ large $\tilde{\omega}_p(g_k)\leq p.(\mathcal{A}+\bar{\varepsilon})+\underline{C}_0p^{\frac{1}{n+1}} +\underline{C}_1$ and thus
$$\tilde{\omega}_p(g)\leq p.(\mathcal{A}+\bar{\varepsilon})+\underline{C}_0p^{\frac{1}{n+1}} +\underline{C}_1,$$
which implies that $\limsup_{p\to \infty} \frac{\tilde{\omega}_p(g)}{p}\leq \mathcal{A} +\bar{\varepsilon}$. Since $\bar{\varepsilon}$ was arbitrarily small, we conclude with (\ref{fleur}) that the proposition is true.
\end{proof}

Finally we have the following lemma. Remember that in this section we assume that $g$ satisfies Condition [M].
\begin{lemme} \label{pareil}
Suppose that any closed  embedded minimal hypersurface in $X$ intersecting $\Gamma$ is a saddle point minimal hypersurface.
Then for all closed embedded minimal hypersurface $\Sigma\subset \interior(X)$ intersecting $\Gamma$, 
$$\Vol_n(\Sigma) > \mathcal{A}(\Gamma,g) \text{  if $\Sigma$ is 2-sided},$$
$$2\Vol_n(\Sigma) > \mathcal{A}(\Gamma,g) \text{  if $\Sigma$ is 1-sided}.$$
\end{lemme}
\begin{proof}
Let us check the lemma when $\Sigma$ is 2-sided, the other case being similar. It is a saddle point by assumption, so there is a mean concave neighborhood $N_\Sigma$ of $\Sigma$ foliated by hypersurfaces with mean curvature vector pointing away from $\Sigma$ when non-zero. The boundary $\partial N_\Sigma$ has two connected components $\Sigma_1$, $\Sigma_2$ whose $n$-volume are both strictly less than $\Vol_n(\Sigma)$. For $k$ large, $N_\Sigma\subset \interior(X_k)$ since $g_k=g$ on larger and larger balls. For $k$ large, minimal hypersurfaces in $(X_k,g_k)$ with $n$-volume at most $\max\{\Vol_n(\Sigma_1),\Vol_n(\Sigma_2)\}$ and intersecting $\Gamma$ are saddle point minimal hypersurfaces: to see this suppose there is a sequence of minimal hypersurfaces $S_{k_i}\subset (X_{k_i},g_{k_i})$ either strictly stable or degenerate stable of Type I or II, intersecting $\Gamma$ and of $n$-volume bounded uniformly. Then a subsequence converges, which is closed by thickness at infinity and it means that for $k_i$ large, $S_{k_i}$ is minimal for $g$, intersects $\Gamma$ but is not a saddle point: contradiction. Now since moreover minimal hypersurfaces in $(\interior(X_k),g_k)$ not intersecting $\Gamma$ are not locally area minimizing but degenerate stable of Type I, by arguments in the proof of \cite[Lemma 13]{AntoineYau} applied to $(X_k,g_k)$, 
$$\max\{\Vol_n(\Sigma_1),\Vol_n(\Sigma_2)\}> \mathcal{A}_k .$$
Since $\mathcal{A}_k$ converges to $\mathcal{A}$, we get the lemma.
\end{proof}

\subsection{Zero-infinity dichotomy for the space of cycles in manifolds thick at infinity} 

We now state our main theorem, which can be thought of as an extension of Yau's conjecture to non-compact manifolds thick at infinity: 
\begin{theo}[Zero-infinity dichotomy]  \label{dicho}
Let $(M,g)$ be an $(n+1)$-dimensional complete manifold with $2\leq n\leq 6$, thick at infinity. Then the following dichotomy holds true:
\begin{enumerate}
\item either $(M,g)$ contains infinitely many saddle point minimal hypersurfaces,
\item or there is none; in that case for any compact domain $B$, there is an embedded closed globally area minimizing hypersurface $\Sigma_B \subset (M,g)$ (maybe empty) such that $B\backslash \Sigma_B$ has a singular weakly mean convex foliation.

\end{enumerate}
\end{theo}

%\begin{remarque}
By inspecting the proof of the above theorem, we can check that it also holds more generally for manifolds $M^{n+1}$ ($2\leq n\leq 6$) thick at infinity with minimal boundary, such that each component of $\partial M$ is closed. In particular, if $M$ is compact with minimal boundary, then the following dichotomy holds: either there are infinitely many saddle point minimal hypersurfaces in the interior of $M$, or there is a closed embedded area minimizing minimal hypersurface $\Sigma\subset M$ such that $M\backslash \Sigma$ has a singular weakly mean convex foliation. 
%\end{remarque}

As briefly mentionned in the introduction, there is an interpretation of the zero-infinity dichotomy in Morse theoric terms. Let $\mathcal{Z}_n(M;\mathbb{Z}_2)$ be the space of integral cycles with bounded support in $M$, endowed with the flat topology. It is a space of generalized closed hypersurfaces and the mass functional $\mathbf{M}$ extends the notion of $n$-volume for smooth hypersurfaces. The dichotomy then says that either $\mathcal{Z}_n(M;\mathbb{Z}_2)$ contains infinitely many ``non-trivial critical points'' of $\mathbf{M}$ whose supports are smooth minimal hypersurfaces, or $\mathcal{Z}_n(M;\mathbb{Z}_2)$ is locally simple: the only ``critical points'' are supported on smooth stable minimal hypersurfaces and for each bounded region $B$ of $M$, the set of elements of $\mathcal{Z}_n(M;\mathbb{Z}_2)$ with support inside $B$ can be contracted to $\mathbf{M}$-minimizing elements of $\mathcal{Z}_n(M;\mathbb{Z}_2)$ by a retraction flow which is $\mathbf{M}$-nonincreasing and continuous in the flat topology.

The proof of Theorem \ref{dicho} will follow from Theorem \ref{b} and the following:

\begin{theo} \label{oneinfinity}
Let $(M,g)$ be an $(n+1)$-dimensional complete manifold with $2\leq n\leq 6$, thick at infinity. If there exists a saddle point minimal hypersurface, then there exists infinitely many.
\end{theo}

\begin{proof}
We suppose that $M$ is thick at infinity. Unless specified, minimal hypersurfaces are closed embedded. We assume that the metric $g$ satisfies Condition [M], otherwise there are already infinitely many saddle points (Appendix A). 

Let $\Gamma$ be a saddle point minimal hypersurface in $(M,g)$. We can suppose $\Gamma$ to be connected. We are now in the situation treated by Subsection \ref{Appendix D}. We will explain the case where $\Gamma$ is 2-sided for simplicity, but the case where it is 1-sided is completely analogous. $\Gamma$ has to be either unstable, or degenerate stable of Type III, and so we can find a neighborhood $N_\Gamma$ of $\Gamma$ and a diffeomorphism $\phi:\Gamma\times(-\delta_1,\delta'_1) \to N_\Gamma$ such that $\phi(\Gamma\times\{0\}) = \Gamma$, the mean curvature of $\phi(\Gamma\times\{s\})$ is either vanishing or non-zero pointing away from $\Gamma$, and $\phi(\Gamma\times\{-\delta_1\})$, $\phi(\Gamma\times\{\delta'_1\})$ have non-zero mean curvature.

Let $K_0:=M\backslash N_\Gamma$. Let $X$ be generated by $\Gamma$, $K_0$ in $(M,g)$ as explained in Subsection \ref{Appendix D}. Recall that $X$ is the limit of compact manifolds $X_k$.
%It is a strictly mean convex topologically closed subset of $M$. We run the level set flow to $K_0$ and get a family $\{K_t\}_{t\geq0}$. Denote by $X$ the metric completion of $M\backslash \bigcap_{t\geq0} K_t$. $X$ is a (maybe non-compact) manifold with minimal stable boundary components. Each of them is compact because $(M,g)$ is thick at infinity, and locally area minimizing inside $X$. It means that each component of $\partial X$ is either strictly stable or degenerate stable of Type II. 

Suppose first that there is a closed connected minimal hypersurface $S$ embedded inside the interior of $X$, intersecting $\Gamma$, which is either strictly stable or degenerate stable of Type I or II. 
\begin{itemize}
\item 
If either $S$ separates $X$ into $2$-components, or $S$ is not degenerate stable of Type I, then let $Y$ denotes the metric completion of $X\backslash S$. 
\item Otherwise $S$ does not separate and is degenerate stable of Type I (it is in particular 2-sided, see Appendix A). If $X'$ is the metric completion of $X\backslash S$, two boundary components $T_1,T_2$ of same $n$-volume come from $S$. One (say $T_1$) is locally area minimizing inside $X'$, the other component $T_2$ is not. We minimize the $n$-volume of $T_2$ inside $X'$ in its homology class and get a 2-sided locally area minimizing minimal hypersurface $W$ that separates $T_1$ from $T_2$; $W$ is not necessarily compact but at least it is locally compact. Let $Y$ be defined as the metric completion of the component of $X'\backslash W$ containing $T_1$. 
\end{itemize}

In any case, $Y$ has locally area minimizing boundary, whose components are closed. Let $\Gamma^*\subset Y$ denote the preimage of $\Gamma$ under the natural map $Y\to M$.
%In any case, if $X'$ denotes the metric completion of $X\backslash S$, by minimizing the $n$-volume of the boundary of $X'$ inside $X'$ we get a manifold $Y\subset X'$ with locally area minimizing boundary whose components are compact by thickness at infinity. Moreover $\Gamma\nsubseteq Y$ since $S$ intersects $\Gamma$ by assumption. Let $N_1$ be a thin strictly mean convex neighborhood of the compact boundary $\partial Y$ (the mean curvature vector points towards $\partial Y$), then for $T>0$ large enough, the closure of $M\backslash( N_1\cup \bigcap_{t\in[0,T]} K_t)$ is a compact domain $D_1$ inside $Y$ with strict mean concave boundary. 
We apply Theorem \ref{itsok} to a domain $B\subset Y$ containing $\Gamma^*$ and the connected components of $\partial Y$ touching $\Gamma^*$. We claim that only case (1) can happen: if case (2) was true, there would be a locally area minimizing minimal hypersurface $\Sigma_B$ and a singular weakly mean convex foliation $\{A_t\}_{t\geq 0}$ of $B\backslash \Sigma_B$. Since $\partial B\cap \partial Y$ is locally area minimizing, there is a first time $T$ when $\partial A_T$ touches $\Gamma^*\backslash \partial \Gamma^*$, which contradicts the maximum principle \cite{WhiteregMCF}.
%and by an argument analogous to the proof of Corollary \ref{explanation} (2), item (2) of Theorem \ref{b} cannot happen by strict mean concavity of $\partial D_1$. 
Thus we deduce that there is a saddle point $\Gamma_1$ embedded in the interior of $Y$, different from $\Gamma^*$; in particular it projects to an embedded saddle point in $M$ still denoted by $\Gamma_1$ and different from $\Gamma$. 
We can reapply this discussion to the new saddle point $\Gamma_1 \subset Y$, get a manifold $X^{(1)}$ with compact locally area minimizing boundary (which is constructed from $\Gamma_1$ inside $Y$), and if $X^{(1)}$ contains in its interior a closed minimal hypersurface $S_1$ intersecting $\Gamma_1$, which is either strictly stable or degenerate stable of Type I or II, we get a new saddle point $\Gamma_2$ whose projection in $M$ is still denoted by $\Gamma_2$, is embedded and different from $\Gamma$ and $\Gamma_1$. If at each step, $X^{(j)}$ contains in its interior a minimal hypersurface either strictly stable or degenerate stable of Type I or II, then we produce an infinite sequence of distinct saddle points $\Gamma_j$ inside the original manifold $(M,g)$, and the conclusion of the theorem is verified.

For these reasons, we now only need to assume that if $X$, $\Gamma$ are as above, minimal hypersurfaces embedded in the interior of $X$ intersecting $\Gamma$ are either unstable or degenerate stable of Type III, that is, they are all saddle points. Note that this condition implies the 
\begin{equation}\label{frankel}
\begin{aligned}
& \text{ Frankel property inside $X$ for minimal hypersurfaces \textit{intersecting} $\Gamma$:}\\
&\text{ any two closed embedded minimal hypersurfaces in $\interior(X)$} \\
& \text{ intersecting $\Gamma$ have to intersect each other.}
\end{aligned}
\end{equation}
To check this, first notice by the maximum principle applied to $(X_k,g_k)$ for $k$ large, that any minimal hypersurface in $\interior(X)$ disjoint from $\Gamma$ is of the form $\phi(\Gamma\times\{s\})$ for some $s\in(-\delta_1,\delta_1')\backslash \{0\}$, so it is degenerate stable of Type I. Secondly, if there were two disjoint saddle point minimal hypersurfaces $S_1$, $S_2$ in $\interior(X)$ (intersecting $\Gamma$), then we could consider the following minimization procedure: since $\Gamma$ is connected, there is a curve embedded in $\interior(X)$ and joining $S_1$ to $S_2$, so any current homologous to $S_1$ contained in $X\backslash (S_1\cup S_2)$ intersects that curve; in other words, we can minimize the $n$-volume of $S_1$ inside $X\backslash (S_1\cup S_2)$ and get a minimal hypersurface $S\subset X$ with compact components, which is not fully contained in the boundary $\partial X$. A fortiori each component of $S$ is either strictly stable or degenerate stable of Type II. But we saw that it means $S$ intersects $\Gamma$, contradicting our assumption on $X$.

Before continuing, we explain why all connected components of the minimal hypersurfaces we will construct from now on will always intersect $\Gamma$. If $\Gamma$ is unstable then we could have chosen $\phi$ so that all the hypersurfaces $\phi(\Gamma\times \{s\})$ with $s\neq 0$ are non-minimal with mean curvature vector pointing away from $\Gamma$. In that case by the maximum principle any minimal hypersurface in $(\interior(X),g)$ intersects $\Gamma$ so we are done. However if $\Gamma$ is degenerate stable of Type III, it might happen that some $\phi(\Gamma\times \{s\})$ distinct from $\Gamma$ are also minimal. By Lemma \ref{adeformation} (2) there is a sequence of metrics $h^{(q)}$ converging to $g$ so that $\Gamma$ is still a saddle point and any $\phi(\Gamma\times \{s\})$ with $s\neq 0$ is non-minimal with mean curvature vector pointing away from $\Gamma$. Hence any minimal hypersurface in $(\interior(X),h^{(q)})$ intersects $\Gamma$ by the maximum principle. Applying min-max theory to $(X,h^{(q)})$ will produce for each integer $p$ a minimal hypersurface with integer multiplicities in $\interior(X)$, intersecting $\Gamma$, of Morse index and $n$-volume bounded independently of $q$  (see Proposition \ref{etouiminmax} and paragraph above). By \cite{Sharp}, we take a subsequence limit as $q\to \infty$ and get a minimal hypersurface with similar properties in $(\interior(X),g)$: in particular each component intersects $\Gamma$. Thus by (\ref{frankel}), the usual Frankel property, even though it may not be satisfied for all minimal hypersurfaces in $\interior(X)$, will be satisfied for all the minimal hypersurfaces we will consider until the end of this proof (since they are constructed by min-max). In what follows, we will implicitly assume the use of such a limiting procedure involving $h^{(q)}$.

For the end of the proof, the strategy to produce infinitely many saddle points is to use arguments of the solution of Yau's conjecture \cite{MaNeinfinity} \cite{AntoineYau}. In the simplest case where $X$ is actually closed compact then by the Frankel property in $X$ satisfied by minimal hypersurfaces intersecting $\Gamma$, \cite{MaNeinfinity} implies the existence of infinitely many minimal hypersurfaces in $X$ which are saddle point minimal hypersurfaces by assumption on $X$.

In Subsection \ref{Appendix D}, we defined some numbers $\tilde{\omega}_p(X,g)$ ($p\geq1$) and $\mathcal{A}(\Gamma,g)$. Suppose that $X$ is non-compact without boundary and $\mathcal{A}(\Gamma,g)=0$. By Proposition \ref{etouiminmax} and by (\ref{frankel}), the widths $\tilde{\omega}_p(X,g)$ are all finite, for each $p$ there is a connected closed embedded minimal hypersurface $\Gamma^{(p)}\subset (\interior(X),g)$ and a positive integer $m_{p}$ (which is even when $\Gamma^{(p)}$ is 1-sided) satisfying
\begin{equation} \label{reali}
\tilde{\omega}_p(X,g) = m_{p} \Vol_n(\Gamma^{(p)}).
\end{equation}
Moreover we have the following asymptotics (Proposition \ref{flore}):
\begin{equation} \label{subblinear}
\lim_{p\to\infty}\tilde{\omega}_p(X,g)/p = 0.
\end{equation}
We claim that this implies the existence of infinitely many saddle point minimal hypersurfaces in $\interior(X)$. Suppose by contradiction that there are only finitely many minimal hypersurfaces $\Sigma_1,...,\Sigma_L\subset \interior(X)$ intersecting $\Gamma$. Then $\tilde{\omega}_p(X,g)$ is a strictly increasing sequence. To explain this, consider the compact manifolds $(X_k,g_k)$ defined in Subsection \ref{Appendix D}; by \cite{Sharp} and thickness at infinity, for any fixed $p$ and for $k$ large enough depending on $p$, the min-max minimal hypersurface in $(X_k,g_k)$ whose $n$-volume with multiplicity is $\tilde{\omega}_p(X_{k}, g_{k})$ is actually a minimal hypersurface for the original metric $g$ (see paragraph right before Proposition \ref{etouiminmax}). 
%Then because of (\ref{reali}),
Consequently by finiteness assumption on the number of minimal hypersurfaces intersecting $\Gamma$, we can find a subsequence $\{X_{k_i}\}$ of $\{X_k\}$, so that both sequences $\{\tilde{\omega}_p(X_{k_i}, g_{k_i})\}$ and $\{\tilde{\omega}_{p+1}(X_{k_i}, g_{k_i})\}$ stabilize to a constant, i.e.
$$\forall i\geq 1, \quad \tilde{\omega}_p(X_{k_i}, g_{k_i})= \tilde{\omega}_p(X,g), \quad \tilde{\omega}_{p+1}(X_{k_i}, g_{k_i})= \tilde{\omega}_{p+1}(X,g).$$
So if $\tilde{\omega}_p(X,g)=\tilde{\omega}_{p+1}(X,g)$, then in particular $\tilde{\omega}_p(X_{k_1}, g_{k_1})=\tilde{\omega}_{p+1}(X_{k_1}, g_{k_1})$ but this is not possible by \cite[Theorem 9 (1)]{AntoineYau} and since $\partial X_k\neq \varnothing$ ($X$ is non-compact).
Sublinearity (\ref{subblinear}) and the counting argument of \cite[Section 7]{MaNeinfinity} is then enough to get the existence of infinitely many minimal hypersurfaces (contradicting our assumption that there were only finitely many). The theorem is then proved in that case.

Suppose finally that $X$ is non-compact (with empty or non-empty boundary) and $\mathcal{A}(\Gamma,g)>0$. We appeal to the method of \cite{AntoineYau} as follows. In Subsection \ref{Appendix D}, we saw that $\tilde{\omega}_p(X,g)$ are finite numbers, and by Proposition \ref{etouiminmax} combined with (\ref{frankel}) for all $p$ there is a connected closed embedded minimal hypersurface $\Gamma^{(p)}\subset \interior(X)$ and a positive integer $m_{p}$ (which is even when $\Gamma^{(p)}$ is 1-sided) satisfying
$$\tilde{\omega}_p(X,g) = m_{p} \Vol_n(\Gamma^{(p)}).$$
%Note that the maximal $n$-volume of the boundary components of $\partial X$ is bounded above by $\mathcal{A}(\Gamma,g)$, since part of $\partial K_t$ is smoothly converging to $\partial X$. 
Furthermore, by Proposition \ref{flore}, for all $p$
$$\tilde{\omega}_{p+1}(X,g) -\tilde{\omega}_p(X,g) \geq \mathcal{A}(\Gamma,g),$$
$$\lim_{p\to\infty}\tilde{\omega}_p(X,g)/p = \mathcal{A}(\Gamma,g).$$
Additionally, by Lemma \ref{pareil} for any closed embedded minimal hypersurface $\Sigma\subset \interior(X)$, 
$$\Vol_n(\Sigma) > \mathcal{A}(\Gamma,g) \text{  if $\Sigma$ is 2-sided},$$
$$2\Vol_n(\Sigma) > \mathcal{A}(\Gamma,g) \text{  if $\Sigma$ is 1-sided}.$$
Consequently the previous identities and estimates combined with the arithmetic lemma \cite[Lemma 14]{AntoineYau} imply that there are infinitely many closed embedded minimal hypersurfaces inside $\interior(X)$ intersecting $\Gamma$, which we recall have to be saddle points by assumption on $X$. This finishes the proof of the theorem.

\end{proof}

\begin{proof}[Proof of Theorem \ref{dicho}]

Theorem \ref{dicho} readily ensues from Theorem \ref{b} and Theorem \ref{oneinfinity}. The only non trivial point is that in item (2), for any compact domain $B$, the minimal hypersurface $\Sigma_B$ is globally area minimizing in its $\mathbb{Z}_2$-homology class in $(M,g)$ (instead of just locally area minimizing as in item (2) of Theorem \ref{b}). Assume that there are no saddle points, let $B$ be a bounded domain and $\Sigma_B$ the associated locally area minimizing minimal hypersurface given by Theorem \ref{b} (2). Suppose that there is another closed embedded hypersurface $\Gamma$ with $\Vol_n(\Gamma) \leq \Vol_n(\Sigma_B)$ and such that $\partial D = \Gamma \cup \Sigma_B$ for some compact domain $D$. Let $X$ be a compact domain containing $D$, by Theorem \ref{b} (2), there are a locally area minimizing hypersurface $\Sigma_X\subset (M,g)$, a manifold with boundary $Y\supset X$ with a metric $g_Y$ coinciding with $g$ on $X$, and a collection of weakly mean convex closed sets $\{K_t\}_{t\geq 0}$ such that $K_t$ foliates $X\backslash \Sigma_X$. We can minimize the $n$-volume of $\Gamma$ inside $K_0$ (which is mean convex) and obtain a locally area minimizing minimal hypersurface $\Gamma' \subset (K_0,g_Y)$, of $n$-volume at most $\Vol_n(\Sigma_B)$.  The non strictly mean convex components of $\partial K_t$ which intersect $\interior(X)$ have to be smooth minimal, degenerate stable of Type I. Hence $\Sigma_B \subset \Sigma_X$ and $\Gamma' \subset \Sigma_X$. But since these two hypersurfaces are $\mathbb{Z}_2$-homologous inside $K_0$, $\Sigma_B=\Gamma'$. Moreover a posteriori, we now know that $\Gamma$ was actually already locally area minimizing, hence $\Gamma\subset \Sigma_X$ and so again $\Gamma=\Sigma_B$. This proves that $\Sigma_B\subset (M,g)$ is the unique area minimizer in its $\mathbb{Z}_2$-homology class inside $(M,g)$.

\end{proof}

\section{Yau's conjecture for finite volume hyperbolic $3$-manifolds} \label{hyperbolicyau}

Combining methods from previous sections and \cite{ColHauMazRos,ColHauMazRoserratum}, we prove that Yau's conjecture holds true for finite volume hyperbolic $3$-manifolds.
\begin{theo} \label{vrille}
In any finite volume hyperbolic $3$-manifold $M$, there are infinitely many saddle point minimal surfaces. In particular, there are infinitely many closed embedded minimal surfaces.
\end{theo}

\begin{proof}
We assume that $g_{\mathrm{hyp}}$ satisfies Condition [M], otherwise the conclusion of the theorem is already true. Let $p\in M$ and denote by $B_r(p)$ the geodesic ball of radius $r$ centered at $p$ in $(M,g_{\mathrm{hyp}})$. For simplicity let us assume that $M$ is oriented so that sections of the cusps are tori. In the non-orientable case, those sections can be Klein bottles, but otherwise the proof is the same.

If $(M,g_{\mathrm{hyp}})$ is not compact then outside of a compact subset, $M$ is made of hyperbolic cusps $C_1,...,C_K$, each of them is naturally foliated by mean-concave tori $\{T^{(k)}_t\}_{t\geq0}$ ($k=1,...,K$). Let us cut $M$ along a collection of tori $\{T^{(1)}_{t_1},...,T^{(K)}_{t_K}\}$ to get $M'$, and deform $g_{\mathrm{hyp}}$ a bit around these tori to obtain a metric $g'$ with respect to which these tori become minimal stable, and such that $\{T^{(k)}_t\}_{t\in[0,t_k)}$ remain strictly mean concave (mean curvature vector pointing towards $T_{t_k}$). If $t_k$ are chosen large enough and the deformations small enough then by \cite{ColHauMazRoserratum}, for all $A>0$, there is $R=R(A)$ independent of $\{t_k\}_{k\in\{1,...,K\}}$ so that any closed embedded minimal surface in $(M',g')$ with Morse index bounded by $A$ is contained in $B_R(p)$ or is a component of $\partial M'$.

Hence by choosing $t_k$ larger and larger, we construct a sequence of compact manifolds $(M_m,g_m)$ approximating the hyperbolic manifold $(M,g_{\mathrm{hyp}})$ with the following properties:
\begin{enumerate}[label=(\roman*)] 
\item $M_m\subset M$, the metrics $g_m$ and $g_{\mathrm{hyp}}$ coincide on $(B_m(p),g_{\mathrm{hyp}})$,
\item for all $\varepsilon>0$, there is a radius $r=r(\varepsilon)$ so that for all $m$ large, $(M_m\backslash B_r(p), g_m)$ has a $1$-sweepout $\{\Sigma_t\}_{t\in[0,1]}$ with $\sup_{t\in[0,1]}\Vol_n(\Sigma_t) \leq \varepsilon$,
\item $\partial M_m$ is non-empty, is a strictly stable minimal surface and has area converging to zero,
\item for all $A>0$, there is $R=R(A)$ so that any closed embedded minimal surface in $(M_m,g_m)$ not equal to a component of $\partial M_m$, and with
%area and
Morse index bounded by $A$ is contained in $B_R(p)$,
\item the widths of $(M_m,g_m)$ are uniformly bounded independently of $m$.
\end{enumerate}
By these properties, applying Theorem \ref{sadddle} to each $(M_m,g_m)$ and using Item (iv), we get a saddle point minimal surface $\Gamma\subset (M,g_{\mathrm{hyp}})$. 

Let us explain how to adapt the proof of Theorem \ref{oneinfinity}. As in the proof of Theorem \ref{oneinfinity}, there is a thin strictly mean concave neighborhood $N_\Gamma$ of $\Gamma\subset (M,g_{\mathrm{hyp}})$ which is foliated by surfaces with mean curvature pointing away from $\Gamma$ when non-zero. For each $m$ large, define $K^m_0 :=M_m\backslash N_\Gamma $. Let $X^m$ be generated by $\Gamma$, $K^m_0$ in $(M_m,g_m)$. Then $X^m$ is compact and has locally area minimizing boundary. By the properties of $g_m$ and of hyperbolic cusps, $(X^m,g_m)$ converges (say in the pointed Gromov-Hausdorff distance) to a manifold $X$ with compact boundary endowed with a hyperbolic metric still denoted by $g_{\mathrm{hyp}}$. Here the boundary $\partial X$ is compact because of Item (iv). Let $(\mathcal{C}(X),h)$ be the result of gluing $(X,g_{\mathrm{hyp}})$ to a straight half-cylinder $(\partial X\times [0,\infty),g_{\mathrm{hyp}}\big|_{\partial X} \oplus dt^2)$. For each positive integer $p$, let 
$$\tilde{\omega}_p(X,g_{\mathrm{hyp}}) = \omega_p(\mathcal{C}(X),h)$$
(see Definition 8 in \cite{AntoineYau}) and define
$$\mathcal{A}(\Gamma, g_{\mathrm{hyp}}) = \max\{\Vol_2(C) ; C\text{ is a component of $\partial X$}\}.$$
Similarly we introduce for all $m$ large and all $p$, $(\mathcal{C}(X^m),h_m)$, $\tilde{\omega}_p(X^m,g_m)$ and $\mathcal{A}(\Gamma, g_m)$. It is not hard to see that 
\begin{equation}\label{trotsky}
\lim_{m\to \infty }\mathcal{A}(\Gamma, g_m)=\mathcal{A}(\Gamma, g_{\mathrm{hyp}}).
\end{equation} 
Besides we have 
\begin{equation}\label{eeequality}
\forall p, \quad \lim_{m\to\infty} \tilde{\omega}_p(X^m,g_m)=\tilde{\omega}_p(X,g_{\mathrm{hyp}}).
\end{equation} 
Note that for any compact domain $D$ of $(\mathcal{C}(X),h)$, if $m$ is large then $(\mathcal{C}(X^m),h_m)$ contains an isometric copy of $D$. 
So the inequality $\geq$ in (\ref{eeequality}) follows from definitions whereas the inequality $\leq$ in (\ref{eeequality}) comes from Items (ii) (iii) in the list of properties of $g_m$, the Fact of Subsection \ref{Appendix D} (which says that a $1$-seepout naturally yields a $p$-sweepout with natural mass bounds) and the possibility to glue two $p$-sweepouts together with good bound on the mass (see end of proof of Theorem \ref{flore}): in other words for any $\varepsilon>0$, a $p$-sweepout of a compact domain $D$ of $(\mathcal{C}(X),h)$ with slices of mass at most $A'$ yields, for $m$ large, a $p$-sweepout of $(\mathcal{C}(X^m),h_m)$ with slices of mass at most $A'+\varepsilon$.

If there is a closed minimal surface $S$ in $\interior(X)$ intersecting $\Gamma$, either strictly stable or degenerate stable of Type I or II, then by the same arguments as in the proof of Theorem \ref{oneinfinity} applied to $(X^m,g_m)$, for $m$ large we get another saddle point minimal surface $\Gamma_1\subset (X,g_{\mathrm{hyp}})$ and we can continue. Either we get infinitely many saddle points or the process stops.

Hence we can assume that actually any minimal surface intersecting $\Gamma$ is a saddle point. As in the proof of Theorem \ref{oneinfinity}, by an approximation argument, minimal surfaces produced by min-max can be supposed to intersect $\Gamma$. From here, we can apply almost verbatim the end of proof of Theorem \ref{oneinfinity} to $(X,g_{\mathrm{hyp}})$ by using $(X^m,g_m)$ for $m$ large. More precisely: if $X$ is compact without boundary then we conclude using \cite{MaNeinfinity}. If $X$ is non-compact,  for all $p$ we produce as in the proof of Proposition \ref{etouiminmax} a connected minimal surface $\Gamma^{(p)}_m\subset (\interior(X^m),g_m)$ intersecting $\Gamma$, of index at most $p$, satisfying for an integer $k_{p,m}$ (which is even when $\Gamma^{(p)}_m$ is 1-sided):
$$\tilde{\omega}_p(X^m,g_m) = k_{p,m} \Vol_n(\Gamma^{(p)}_m).$$
By Item (iv) in the list of properties of $(M_m,g_m)$, for each $p$, $\Gamma^{(p)}_m$ is a minimal hypersurface for the original hyperbolic metric if $m$ is large enough. By taking a converging subsequence, there is a connected closed embedded minimal surface $\Gamma^{(p)}\subset (\interior(X),g_{\mathrm{hyp}})$ intersecting $\Gamma$ and a positive integer $k_p$ (which is even when $\Gamma^{(p)}$ is 1-sided) so that:
\begin{equation}\label{realization}
\tilde{\omega}_p(X,g_{\mathrm{hyp}}) = k_{p} \Vol_n(\Gamma^{(p)}).
\end{equation}
Here the use of Item (iv) in the properties of $g_m$ is essential to get a closed limit surface, since $(M,g_{\mathrm{hyp}})$ is not thick at infinity in general.
Moreover since for all $m$, $\tilde{\omega}_{p+1}(X^m,g_m) - \tilde{\omega}_p(X^m,g_m) \geq \mathcal{A}(\Gamma,g_m)$, we get from (\ref{trotsky}) for all $p$:
$$\tilde{\omega}_{p+1}(X,g_{\mathrm{hyp}}) - \tilde{\omega}_p(X,g_{\mathrm{hyp}}) \geq \mathcal{A}(\Gamma,g_{\mathrm{hyp}}).$$
Since outside a compact set, $X$ is made of finitely many cusps, for any $\varepsilon>0$ there is an $R'$ so that any bounded domain of $X\backslash B_{R'}(p)$ has a $1$-sweepout $\{\Sigma_t\}_{t\in[0,1]}$ with $\sup_t \Area(\Sigma_t) \leq \varepsilon$. Thus by arguments of Subsection \ref{Appendix D}, we get 
$$\lim_{p\to \infty} \frac{\tilde{\omega}_p(X,g_{\mathrm{hyp}})}{p} = \mathcal{A}(\Gamma,g_{\mathrm{hyp}}).$$
As in the proof of Theorem \ref{oneinfinity}, by (\ref{trotsky}) we have for any minimal surface in $(\interior(X),g_{\mathrm{hyp}})$ intersecting $\Gamma$:
$$\Vol_n(\Sigma)> \mathcal{A}(\Gamma,g_{\mathrm{hyp}}) \text{ if $\Sigma$ is 2-sided,}$$
$$2 \Vol_n(\Sigma)> \mathcal{A}(\Gamma,g_{\mathrm{hyp}})\text{ if $\Sigma$ is 1-sided.}$$
We have the ingredients to conclude. If $X$ is non-compact and has non-empty boundary then \cite[Lemma 14]{AntoineYau} imply the existence of infinitely many minimal surfaces inside $\interior(X)$ intersecting $\Gamma$. They are saddle points by assumption on $X$ and $\Gamma$. It remains to treat the case where $X$ is non-compact and has empty boundary.
%, then $\mathcal{A}(\Gamma,g_{\mathrm{hyp}})=0$ and $\tilde{\omega}_p(X,g_{\mathrm{hyp}})=\omega_p(X,g_{\mathrm{hyp}})$ (similarly for $(X^m,g_m)$). 
Suppose towards a contradiction that $\{\Gamma^{(p)}\}_{p\geq 1}$ is a finite set. Then $\tilde{\omega}_p(X,g_{\mathrm{hyp}})$ is strictly increasing in $p$: indeed if $\tilde{\omega}_p(X,g_{\mathrm{hyp}})=\tilde{\omega}_{p+1}(X,g_{\mathrm{hyp}})$, by (\ref{realization}) and the finiteness of $\{\Gamma^{(p)}\}_{p\geq 1}$, for some $m$ arbitrarily large 
$$\tilde{\omega}_p(X^m,g_m) = \tilde{\omega}_p(X,g_{\mathrm{hyp}})=\tilde{\omega}_{p+1}(X,g_{\mathrm{hyp}})=\tilde{\omega}_{p+1}(X_m,g_m),$$
which is not possible because of $\partial X_m\neq \varnothing$ and \cite[Theorem 9 (1)]{AntoineYau}. Finally the counting argument of \cite[Section 7]{MaNeinfinity} applies and shows that $\{\Gamma^{(p)}\}_{p\geq 1}$ is actually an infinite set.

\end{proof}

\section{Density of the union of finite volume minimal hypersurfaces} \label{densii}

In our setting the natural topology on the space of complete metrics is the usual strong $C^\infty$-topology, or ``Whitney $C^\infty$-topology'' (see \cite[Chapter 2, \S 1]{Hirsch94} or \cite[Chapter II, \S 3 ]{GG} where it is defined for spaces of functions). It can be described as follows. Let $b_1,b_2,...$ be a sequence of open balls forming a locally finite covering $M$, and for each complete metric $g$, and $i$, $\varepsilon$, $k$, set
$$O(g,i,\varepsilon,k):=\{g' ; \|(g'-g)\big|_{b_i}\|_{C^k}< \varepsilon\}.$$
Now consider $g$ be a complete metric, $\mathbf{e}=(\varepsilon_1,\varepsilon_2,...)$ a sequence of positive numbers, 
%$\mathbf{k}=(k_1,k_2,...)$ a bounded sequence of integers
$k$ an integer and set
$$O(g,\mathbf{e},{k}) := \bigcap_{i=1}^\infty O(g,i,\varepsilon_i,k).$$
Then by definition the strong $C^\infty$  topology on the space of complete metrics of $M$ has a basis of open sets given by $\{O(g,\mathbf{e},{k})\}_{g,\mathbf{e},{k}}$. The weak $C^\infty$-topology where a basis of open sets is given by $\{O(g,i,k,\varepsilon)\}_{g,i,k,\varepsilon}$ is less relevant for us, because it is ``easier'' to change the behavior at infinity and to be generic in this topology. 
%The reader can check that a generic set of metrics for the strong topology is also generic in the weak topology, but the converse is false. 
The space of complete metrics endowed with the strong topology is a Baire space \cite[Chapter 2, \S 4, Theorem 4.4]{Hirsch94}.

As usual, a closed minimal hypersurface is said to be non-degenerate when it has no non-trivial Jacobi fields.

Recall from the introduction that $\mathcal{F}_{\mathrm{thin}}$ (resp. $\mathcal{T}_\infty$) is the family of complete metrics on $M$ with a thin foliation at infinity (resp. thick at infinity). $\mathcal{F}_{\mathrm{thin}}$ and $\mathcal{F}_{\mathrm{thin}} \cap \interior(\mathcal{T}_\infty)$ are non-empty open subsets for the strong topology. For instance, consider the following example already mentioned in the introduction: if $N$ is a closed manifold, there is a metric $h$ on $N\times \mathbb{R}$ with a thin foliation at infinity while also satisfying condition $\star_k$ of \cite{Montezuma1}, hence being thick at infinity. Since $\star_k$ is an open condition, a neighborhood of $h$ is in $\mathcal{F}_{\mathrm{thin}} \cap \interior(\mathcal{T}_\infty)$. Note that if $M$ is compact, any metric on $M$ is in $\mathcal{F}_{\mathrm{thin}} \cap \interior(\mathcal{T}_\infty)$. The following theorem generalizes the density theorem of Irie, Marques and Neves \cite{IrieMaNe} to these manifolds:

\begin{theo} \label{densiity}
Let $M$ be a complete $(n+1)$-dimensional manifold with $2\leq n\leq 6$. 
\begin{enumerate}
\item For any metric $g$ in a $C^\infty$-dense subset of $\mathcal{F}_{\mathrm{thin}}$, the union of complete finite volume embedded minimal hypersurfaces in $(M,g)$ is dense.
\item For any metric $g'$ in a $C^\infty$-generic subset of $\mathcal{F}_{\mathrm{thin}} \cap \interior(\mathcal{T}_\infty)$, the union of closed embedded minimal hypersurfaces in $(M,g')$ is dense.
\end{enumerate}
\end{theo}

\begin{proof}

We first prove $(1)$ of the statement.
Suppose that $M$ is endowed with a metric $g$ with a thin foliation at infinity. We would like to find a metric $h\in \mathcal{F}_{\mathrm{thin}}$ arbitrarily close to $g\in \mathcal{F}_{\mathrm{thin}}$ in the $C^\infty$-topology, such that the union of complete finite volume embedded minimal hypersurfaces in $(M,h)$ is dense. 

Since $(M,g)\in \mathcal{F}_{\mathrm{thin}}$, for any $\mu>0$, there is a compact subset $C\subset M$ so that any bounded domain of $M\backslash C$ has a foliation $\{\Sigma_t\}_{t\in[0,1]}$ (given for example by the level sets of the function defining the thin foliation at infinity) such that 
$$\sup_{t\in[0,1]} \Vol_n(\Sigma_t) \leq \mu.$$
Hence by techniques explained in Subsection \ref{Appendix D} and Proposition \ref{flore}, the widths $\omega_p(M,g)$ are finite and satisfy
\begin{equation}\label{vraiment}
\lim_{p\to \infty} \frac{\omega_p(M,g)}{p}=0.
\end{equation}

Recall that $b_1,b_2,...$ are open balls forming a locally finite covering of $M$. Fix a integer $\hat{k}$ and a sequence of positive numbers between $0$ and $1$ called $\hat{\varepsilon}=(\varepsilon_1,\varepsilon_2...)$. In what follows, the norms $\|.\|_{C^{{k}}}$ are all computed with respect to the background metric $g$. By abuse of notations, for any symmetric $2$-tensor $g'$ and constant $c$, we will write 
$$\|g'\|_{C^{\hat{k}}} \leq \hat{\varepsilon}c$$ instead of 
$$ \forall j, \quad\|g'\big|_{b_j}\|_{C^{\hat{k}}} \leq \varepsilon_jc.$$

Let $B_{2r_q}(x_q) \subset (M,g)$ be a sequence of open balls centered at points $x_q$, of radii $2r_q$, forming a base of open sets for the usual topology of $M$. We will construct successive deformations of $g$, called $h_0=g$, $h_1$, $h_2$..., converging to a metric $h_\infty$ satisfying $\|g-h_\infty\|_{C^{\hat{k}}} \leq \hat{\varepsilon}$, such that the union of complete finite volume embedded minimal hypersurfaces in $(M,h_\infty)$ is dense. Suppose that the metrics $h_0$, $h_1$, ..., $h_{L-1}$ have already been constructed and that 
\begin{equation}\label{pasloin}
\|g-h_{L-1}\|_{C^{\hat{k}}}\leq \hat{\varepsilon}\sum_{i=1}^{L-1} \frac{1}{2^i}.
\end{equation} 
We construct $h_{L}$ as follows. 

First we can find compact approximations $(M_m,\gamma_m)$ of $(M,h_{L-1})$ in the following manner. Let $f:M\to [0,\infty)$ be the function defining the thin foliation at infinity. One can check that $$\lim_{t\to \infty}\Vol_n(f^{-1}(t), h_{L-1}) = \lim_{t\to \infty}\Vol_n(f^{-1}(t), g)=0.$$ We can assume that positive integers are not critical values of $f$. Consider $M_1\subset ...\subset M_m...$ the exhaustion of $M$ by compact domains $M_m:=f^{-1}([0,m])$. For each $m$, we perturb slightly the metric $h_{L-1}$ on $M_m$ into $\gamma_m$ so that
\begin{itemize} 
\item the metric $\gamma_m$ is bumpy (use \cite{Whitebumpy2}),
\item the boundary $\partial M_m$ is a strictly stable minimal hypersurface with respect to $\gamma_m$,
\item if we restrict the tensors $\gamma_m$, $h_{L-1}$ to $M_m$ then $\|\gamma_m-h_{L-1}\|_{C^{0}}\leq \frac{1}{m}$,
\item if we restrict the tensors $\gamma_m$, $h_{L-1}$ to $M_m\backslash N_m$ where $N_m$ is a $1/m$-neighborhood of $\partial M_m$, then $\|\gamma_m-h_{L-1}\|_{C^{\hat{k}}}\leq \frac{1}{m}$.
\end{itemize}
Let $\tilde{\omega}_p(M_m,\gamma_m) := {\omega}_p(\mathcal{C}(M_m),\gamma_m)$ where $\mathcal{C}(M_m)$ is the result of gluing to $M_m$ a straight half-cylinder $(\partial M_m\times [0,\infty), \gamma_m\big|_{\partial M_m} \oplus dt^2)$ along $\partial M_m$ (as in Subsection \ref{Appendix D}).

Let $q_L$ be the first integer for which $B_{2r_{q_L}}(x_{q_L})$ does not intersect any complete finite volume embedded minimal hypersurface in $(M,h_{L-1})$. If $q_L$ does not exists, we are done by taking $h_L=h_\infty$. Otherwise let $s_L$ be a nonnegative symmetric $2$-tensor with support $B_{2r_{q_L}}(x_{q_L})$ and equal to $h_{L-1}$ inside $B_{r_{q_L}
}(x_{q_L})$. There is a $\mu_L>0$ small enough so that 
\begin{equation}\label{flingue}
\forall t\in[0,\mu_L], \quad\|ts_L\|_{C^{\hat{k}}}\leq \hat{\varepsilon}\frac{1}{2^{L}}.
\end{equation}

Consider the deformation of $ \gamma_m$ given by
$$\gamma_m(t)=\gamma_m +ts_L, \quad t\in[0,\mu_L].$$ 

For each fixed $p$, by the third bullet, $\Vol_n(\partial M_m,\gamma_m(t)) $ goes to zero so by arguments in proof of Theorem \ref{vrille},
\begin{align}\label{attention}
\begin{split}
& \lim_{m\to \infty} \tilde{\omega}_p(M_m,\gamma_m) = \omega_p(M,h_{L-1}),\\
& \lim_{m\to \infty} \tilde{\omega}_p(M_m,\gamma_m+\mu_Ls_L) = \omega_p(M,h_{L-1}+\mu_Ls_L) .
\end{split}
\end{align}

\begin{claim}There is an integer $\bar{p}$, such that for all $m$ large enough, 
$$\tilde{\omega}_{\bar{p}}(M_m,\gamma_m+\mu_L s_L)>\tilde{\omega}_{\bar{p}}(M_m,\gamma_m).$$
\end{claim}

We postpone its proof. The following argument essentially appears in \cite{IrieMaNe}. Fix $\bar{p}$ as in the claim. From \cite[Theorem 10]{AntoineYau} that for all $t\in[0,\mu_L]$, there are connected closed minimal hypersurfaces $\Gamma_1,...,\Gamma_P$ embedded in $\interior(M_m)$ and positive integers $q_1,...,q_P$ so that 
$$\tilde{\omega}_{\bar{p}}(M_m,\gamma_m+ts_L) = \sum_{i=1}^P q_i\Vol_n(\Gamma_i,\gamma_m+ts_L).$$
The hypersurfaces $\Gamma_i$ can be chosen to have index at most $\bar{p}$ \cite{MaNeindexbound}.
Since $\gamma_m$ is bumpy the set of numbers 
$$\{\sum_{i=1}^P q'_i\Vol_n(\Gamma'_i,\gamma_m); q'_i \text{ integers}, \Gamma'_i \subset (M_m,\gamma_m) \text{ closed minimal hypersurfaces}\}$$
is countable. Moreover, since $\tilde{\omega}_{\bar{p}}(M_m,\gamma_m+ts_L)$ is continuous in $t$, for all $m$ large, there exists $t'_m\in(0,\mu_L)$ for which there is a minimal hypersurface $\Gamma^{(m)} \subset (M_m,\gamma_m+t'_ms_L)$ intersecting the support of $s_L$, $B_{2r_{q_L}}(x_{q_L})$, of index at most $\bar{p}$. Taking a subsequence limit as $m\to \infty$ (\cite{Sharp}), we get a $t'_\infty\in[0,\mu_L]$ and a complete finite volume embedded minimal hypersurface in $(M,h_{L-1} + t'_\infty s_L )$ intersecting $B_{2r_{q_L}}(x_{q_L})$. Set
$$h_L=h_{L-1} + t'_\infty s_L.$$
By (\ref{flingue}), we have $\|h_L-h_{L-1}\|_{C^{\hat{k}}} \leq \hat{\varepsilon}\frac{1}{2^{L}}$ so indeed $\|g-h_{L}\|_{C^{\hat{k}}}\leq \hat{\varepsilon}\sum_{i=1}^{L} \frac{1}{2^i}$. We continue this construction and get by completeness a limit metric $h_\infty$ with $\|g-h_{\infty}\|_{C^{\hat{k}}}\leq \hat{\varepsilon}$, for which the union of complete finite volume embedded minimal hypersurfaces in $(M,h_\infty)$ is dense. Note that it was very helpful to use \cite[Theorem 10]{AntoineYau} and construct closed minimal hypersurfaces in order to take advantage of the bumpy metric theorems \cite{Whitebumpy,Whitebumpy2}.

Let us prove the claim (which plays here the role of the Weyl Law \cite{LioMaNe} in \cite{IrieMaNe}). Because of (\ref{attention}), we only need to show that for a certain $\bar{p}$, 
$$\omega_{\bar{p}}(M,h_{L-1} +\mu_Ls_L) > \omega_{\bar{p}}(M,h_{L-1}). $$
The large inequality is always true since $s_L$ is nonnegative, the point is to show that these two terms are not equal. Note that as a consequence of (\ref{vraiment}) and (\ref{pasloin}), (\ref{flingue}), there is a subsequence $\{p_k\}$ and a sequence of nonnegative numbers $\{\delta_k\}$ such that
\begin{equation} \label{thinkey}
\forall k\quad \omega_{p_k+1}(M,h_{L-1})-\omega_{p_k}(M,h_{L-1}) =\delta_k,\quad \lim_{k\to\infty} \delta_k =0.
\end{equation}
Suppose towards a contradiction that for all $k$ we have
$$\omega_{{p}_k+1}(M,h_{L-1} +\mu_Ls_L) = \omega_{{p}_k+1}(M,h_{L-1}),$$
$$\omega_{{p}_k}(M,h_{L-1} +\mu_Ls_L) = \omega_{{p}_k}(M,h_{L-1}). $$
Let $\bar{\delta}>0$ be a small number that we will fix later, let $D\subset M$ be a compact domain containing $B_{r_{q_L}}(x_{r_{q_L}})$. Consider a $(p_{k}+1)$-sweepout $\Phi:X\to \mathcal{Z}_{n,rel}(D;\mathbb{Z}_2)$ so that 
$$\sup_{x\in X} \mathbf{M}(\Phi(x),h_{L-1} +\mu_Ls_L) \leq \omega_{{p}_k+1}(M,h_{L-1} +\mu_Ls_L) +\bar{\delta}$$
where $\mathbf{M}(. , g')$ denotes the mass computed with a metric $g'$. For clarity let us write $B_{r_q}(x_{r_q})$ instead of $B_{r_{q_L}}(x_{r_{q_L}})$. Now the key remark is that $\Phi$ restricted to
$$X_1:=\{x\in X ; \mathbf{M}(\Phi(x)\llcorner  B_{r_q}(x_{r_q}), h_{L-1} +\mu_Ls_L) \geq \omega_1(B_{r_q}(x_{r_q}), h_{L-1} +\mu_Ls_L)/2 \}$$
is a $p_{k}$-sweepout of $D$. The proof is a Lusternik-Schnirelmann type argument used in \cite{Gromovnonlinearspectra,Gromovwaist} (see also \cite[Section 3]{Guth} \cite[Section 8]{MaNeinfinity}). Indeed suppose that it is not, remark that $\Phi$ restricted to $X_2:=X\backslash X_1$ is clearly not a $1$-sweepout (a $1$-sweepout of $M$ has to be a $1$-sweepout of $B_{r_q}(x_{r_q})$ after restricting the image currents to $B_{r_q}(x_{r_q})$). Consider $\lambda:=\Phi^*(\bar{\lambda}) \in H^1(X,\mathbb{Z}_2)$ as in \cite[Definition 4.1]{MaNeinfinity}. Consider the inclusion maps $i_a:X_a\to X$ ($a=1,2$), then the previous restrictions not being sweepouts means $i^*_1(\lambda^{p_k}) =0 \in H^{p_k}(X_1,\mathbb{Z}_2)$,  $i^*_2(\lambda) =0 \in H^{1}(X_2,\mathbb{Z}_2)$. Consider the exact sequences
$$H^{p_k}(X,X_1;\mathbb{Z}_2)\overset{j^*}{\rightarrow}H^{p_k}(X;\mathbb{Z}_2) \overset{i^*_1}{\rightarrow} H^{p_k}(X_1;\mathbb{Z}_2),$$ 
$$H^{1}(X,X_2;\mathbb{Z}_2)\overset{j^*}{\rightarrow}H^1(X;\mathbb{Z}_2) \overset{i^*_2}{\rightarrow} H^1(X_2;\mathbb{Z}_2).$$ 
Then we can find $\lambda_1\in H^{p_k}(X,X_1;\mathbb{Z}_2)$, $\lambda_2\in H^{1}(X,X_2;\mathbb{Z}_2)$ so that 
$j^*(\lambda_1) = \lambda^{p_k}$, $j^*(\lambda_2) = \lambda$, which implies
$$\lambda^{p_k+1} = j^*(\lambda_1) \smile j^*(\lambda_2)= j^*(\lambda_1\smile \lambda_2)$$
but the last term has to be zero since it belongs to $H^{p_k+1}(X,X_1\cup X_2;\mathbb{Z}_2)=H^{p_k+1}(X,X;\mathbb{Z}_2)=0$. This is impossible by definition of $(p_{k}+1)$-sweepouts, hence $\Phi\big|_{X_1}$ is a $p_k$-sweepout. Then since $s_L=h_{L-1}$ in $B_{r_q}(x_{r_q})$, if we estimate the mass for the metric $h_{L-1}$:
\begin{align*}
\sup_{x\in X_1}\mathbf{M}(\Phi(x),h_{L-1}) 
 \leq & \sup_{x\in X_1}\mathbf{M}(\Phi(x),h_{L-1} +\mu_L s_L) \\
 & - (1-(1+\mu_L)^{-n/2})\omega_1(B_{r_q}(x_{r_q}), h_{L-1} +\mu_Ls_L)/2\\
\leq & \quad \omega_{p_k+1}(M,h_{L-1}) + \bar{\delta}\\
& - (1-(1+\mu_L)^{-n/2})\omega_1(B_{r_q}(x_{r_q}), h_{L-1} +\mu_Ls_L)/2\\
 \leq&  \quad \omega_{p_k}(M,h_{L-1}) +\delta_k  + \bar{\delta}\\
 & - (1-(1+\mu_L)^{-n/2})\omega_1(B_{r_q}(x_{r_q}), h_{L-1} +\mu_Ls_L)/2.
 \end{align*}
However, as $k$ goes to infinity, $\delta_k$ converges to zero while the last term is positive and independent of $k$, $D$, $\bar{\delta}$: it follows that if we chose $k$ large and then $\bar{\delta}$ smaller than a fraction of the last term, then $\sup_{x\in X_1}\mathbf{M}(\Phi(x),h_{L-1})<\omega_{p_k}(M,h_{L-1})$ for any $D$ large enough. This obviously contradicts the definition of $p_k$-width (\cite[Definition 8]{AntoineYau}). Hence the claim is verified.

 \vspace{1em}

Next we prove $(2)$ of the statement. Let $\mathcal{U} := \mathcal{F}_{\mathrm{thin}}\cap \interior(\mathcal{T}_\infty)$. Reasoning as in \cite{IrieMaNe}, it is enough to show that for a bounded open set $U\subset M$,  the space $\mathcal{M}_U$ of metrics $g$ in $\mathcal{U}$ such that there is a non-degenerate closed embedded minimal hypersurface in $(M,g)$ intersecting $U$ is open and dense inside $\mathcal{U}$. Openness follows from \cite{Whitebumpy}. Denseness is proved as follows. Fix an integer $\hat{k}$ and a sequence of positive numbers $\hat{\varepsilon}=(\varepsilon_1,\varepsilon_2,...)$, we follow the same notations as previously. We pick a metric $\bar{g}\in\mathcal{U}$, let $\bar{s}$ be a nonnegative symmetric $2$-tensor with support $U$ and equal to $\bar{g}$ in an open ball $B\subset U$. For $\mu>0$ small enough, 
\begin{equation}\label{glace}
\forall t\in[0,\mu],\quad \|t.\bar{s}\|_{C^{\hat{k}}} \leq \hat{\varepsilon}/2
\end{equation} 
where the norm is computed with $\bar{g}$ and by choosing $\varepsilon_1,\varepsilon_2...$ even smaller if necessary, we can suppose that 
\begin{equation}\label{feu}
\forall t\in[0,\mu], \quad\bar{g}+t.\bar{s}\in\mathcal{U}.
\end{equation}
Consider the compact manifolds $(M_m,\gamma_m)$ introduced previously, but with $\bar{g}$ (resp. $U$, $\bar{s}$) replacing $h_{L-1}$ (resp. $B_{2r_{q_L}}(x_{q_L})$, $s_L$), so that the bumpy metrics $\gamma_m$ converge locally to $\bar{g}$. Consider an integer $\bar{p}$ so that the claim in the proof of (1) is valid. As in the proof of (1), this implies that there is $t'_\infty\in[0,\mu]$ such that $(M,\bar{g}+t'_\infty\bar{s})$ contains a finite volume connected complete embedded minimal hypersurface $\Gamma$ intersecting $U$. Since $\bar{g}+t'_\infty\bar{s}\in \mathcal{T}_\infty$ by (\ref{feu}), $\Gamma$ is closed. Then we can perturb $\bar{g}+t'_\infty\bar{s}$ conformally like in \cite[Proposition 2.3]{IrieMaNe} to get by (\ref{glace}) a metric $g\in \mathcal{U}$ satisfying $\|g-\bar{g}\|_{C^{\hat{k}}}\leq \hat{\varepsilon}$ and for which $\Gamma$ is a non-degenerate closed embedded minimal hypersurface intersecting $U$. The denseness of $\mathcal{M}_U $ in $\mathcal{U}$ is checked, which finishes the proof.

\end{proof}

We end this section with a slightly non rigorous construction of non-compact manifolds of finite volume which do not obey the Weyl law  \cite{LioMaNe}, which justify our use of a new argument in the proof of Theorem \ref{densiity} compared to \cite{IrieMaNe}.

\begin{remarque} \label{weird} By \cite{Guthwidthvolume} (see also \cite{PapaSwen,SabourauRic}), 
%for any $A,\varepsilon>0$ there is a metric on the $3$-sphere $S^3$ of volume $1$ such that any surface separating $S^3$ into two regions of volume more than $\varepsilon$ has area greater than $A$. From this, 
one can construct a sequence of spheres $S_k:=(S^3,g_k)$ of volume less than $\frac{1}{k^2}$ with first width in the sense of Almgren-Pitts equal to $1$. Then by forming the infinite connected sum of $S_i$ with thin necks, one gets a non-compact manifold of finite volume but with widths $\omega_p$ growing linearly, by a Lusternik-Schnirelmann type argument. In this construction, there is a lot a freedom in the rate of growth of $\omega_p$ by playing with the sizes of $S_k$. In particular, such a finite volume manifold generally does not satisfy a Weyl law. 
\end{remarque}

\section*{Appendix A: Degenerate stable minimal hypersurfaces} 

We collect some simple facts about the structure of neighborhoods of minimal hypersurfaces.

Let $(N,g)$ be a compact Riemannian manifold. In this Appendix, minimal hypersurfaces are smooth closed embedded. We say that a 2-sided minimal hypersurface is \textit{degenerate} if its Jacobi operator has a non-trivial kernel. If such a hypersurface is degenerate and stable, then the kernel of its Jacobi operator is spanned by a positive eigenfunction $\phi_0$. Note that for a 2-sided minimal hypersurface which is either unstable or non-degenerate stable, it is well-known that the hypersurface has a neighborhood foliated by closed leaves which, when not equal to the minimal hypersurface itself, have non-zero mean curvature vector. A similar result is true for degenerate stable minimal hypersurfaces, as we noted in \cite[Lemma 11]{AntoineYau}.

\begin{lemme} \label{degenerate}
Let $\Gamma$ be a 2-sided degenerate stable minimal hypersurface in the interior of $(N,g)$ and $\nu$ a choice of unit normal vector on $\Gamma$. Denote by $\phi_0$ a positive function in the kernel of the Jacobi operator of $\Gamma$. Then there exist a positive number $\delta_1$ and a smooth map $w:\Gamma \times (-\delta_1,\delta_1) \to \mathbb{R}$ with the following properties:
\begin{enumerate}
\item for each $x\in \Gamma$, we have $w(x,0)=0$ and $\phi_0=\frac{\partial}{\partial t}w(x,t)|_{t=0}$,
\item for each $t\in(-\delta_1,\delta_1)$, we have $\int_\Gamma (w(.,t)-t\phi_0)\phi_0 =0 $,
\item for each $t\in(-\delta_1,\delta_1)$, the mean curvature of the hypersurface 
$$\Gamma_t :=\{\exp(x,w(x,t) \nu(x)) ; x\in \Gamma\}$$
is either positive or negative or identically zero,
\item if $\Gamma_t$ is minimal for a $t\in(-\delta_1,\delta_1)$, its Morse index is at most one.
\end{enumerate}
\end{lemme}

\begin{proof}
The first three items were proved in \cite{AntoineYau}. The last item follows from the fact that, since the first eigenvalue of the Jacobi operator is simple, if a sequence of connected minimal hypersurfaces $S_k$ converges smoothly to a stable minimal hypersurface, then for $k$ large $S_k$ has index at most one.
\end{proof}

If the minimal hypersurface $\Gamma$ is 1-sided, one can still apply the previous lemma in a double-cover of $N$ where $\Gamma$ lifts to a 2-sided hypersurface. If $\Gamma$ is a boundary component of $N$ then the lemma is still valid on the interior side of $\Gamma$.

\textbf{Discussion:} Note that Lemma \ref{degenerate} implies the following for a 2-sided degenerate stable minimal hypersurface $\Gamma$ (the situation is completely similar for 1-sided minimal hypersurfaces): by the first variation formula if $\Gamma$ is degenerate stable, the $n$-volume of the hypersurfaces $\Gamma_t$ is a smooth function 
$$A: (-\delta_1,\delta_1) \to \mathbb{R}$$
so that the sign of $\partial_t A$ is the sign of the mean curvature of $\Gamma_t$ (with the correct continuous choice of unit normal). This function $A$ associated to $\Gamma$ is defined for 2-sided $\Gamma$ embedded inside $\interior(N)$; if $\Gamma$ is 1-sided, then we call $A$ the analogue function associated to the double 2-sided cover and if $\Gamma$ is a boundary component of $N$, then $A$ is only defined on $[0,\delta_1)$. If the function $A$ is not strictly monotonous on any intervals of the form $(-\delta_2,0)$ or $(0,\delta_2)$ where $\delta_2\leq \delta_1$, then there are clearly an infinite sequence of $t_k\in (0,\delta_1)$ converging to $0$ so that for each $k$, $A$ restricted to an open interval containing $t_k$ achieves a maximum at $t_k$. In other words, the minimal hypersurfaces $\Gamma_{t_k}$ are saddle point minimal hypersurfaces (as defined in the introduction). If on the contrary the function $A$ is strictly monotonous on both $(-\delta_2,0)$ and $(0,\delta_2)$, then we have three possibilities: 
\begin{itemize}
\item Type I: $A$ is monotonous on $(-\delta_2,\delta_2)$, 
\item Type II: $A$ achieves a strict minimum at $0$,
\item Type III: $A$ achieves a strict maximum at $0$ (in this last case $\Gamma$ is a saddle point). 
\end{itemize}
Because we are interested in constructing saddle point minimal hypersurfaces, we introduce the following condition: we say that 
\begin{align*}
&\text{a degenerate stable minimal hypersurface $\Gamma$ }\\
&\text{(resp. the metric $g$) satisfies Condition [M]} 
\end{align*}
if, with the previous notations, the function $A$ associated to $\Gamma$ (resp. to any degenerate stable minimal hypersurface) is strictly monotonous on both $(-\delta_2,0)$ and $(0,\delta_2)$ for $\delta_2>0$ small enough. By what we just said, if a degenerate stable minimal hypersurface does not satisfy Condition [M] then there are infinitely many saddle point minimal hypersurfaces. On the other hand if a degenerate stable minimal hypersurface satisfies Condition [M], then it is either of Type I, II or III. 

Again this classification clearly extend to 1-sided hypersurfaces or boundary components. By convention a degenerate stable minimal hypersurface of Type I is always 2-sided embedded in the interior of $(N,g)$ (so for instance a boundary component can be degenerate stable of Type II or III only).

Let $\Gamma$ be stable degenerate of Type II or III. The next lemma enables to construct approximations of $g$ for which the hypersurfaces $\Gamma_t$ are non-minimal if $t\neq 0$.
\begin{lemme} \label{adeformation}
Let $(N,g)$ be a compact manifold with boundary and let $\Gamma$ be a disjoint union of degenerate stable minimal hypersurfaces of Type II or III, embedded either in $\interior(N)$ or in $\partial N$. Then there is a sequence of metrics $h^{(q)}$ converging to $g$ in the $C^\infty$-topology so that $\Gamma$ is still minimal for $h^{(q)}$ and with the previous notations, for each component $\Gamma'$ of $\Gamma$:
\begin{enumerate}
\item if $\Gamma'$ is of Type II, with respect to $h^{(q)}$, $\Gamma'$ is strictly stable, moreover for all $t\in (-\delta_2,0)\cup (0,\delta_2)$ the mean curvature vector of $\Gamma'_t$ is never zero and points towards $\Gamma'$,
\item if $\Gamma'$ is of Type III, with respect to $h^{(q)}$, $\Gamma'$ is unstable, moreover for all $t\in (-\delta_2,0)\cup (0,\delta_2)$, the mean curvature vector of $\Gamma'_t$ is never zero and points away from $\Gamma'$.
\end{enumerate}
\end{lemme}

\begin{proof}
The two bullets have similar proofs. Let us check (2) for instance. Again for brevity we limit ourselves to the case where $\Gamma$ is 2-sided embeddded in the interior of $(N,g)$. Since the deformations are going to be local, we can suppose $\Gamma$ connected. Using previous notations, since $\Gamma$ is of Type III, there is a diffeomorphism $\phi$ from $\Gamma\times (-\delta_2,\delta_2)$ to a neighborhood $N_\Gamma$ of $\Gamma$ so that each $\phi(\Gamma\times \{s\})$ is either minimal or has mean curvature vector pointing away from $\Gamma$. Consider a sequence of functions $f_q:(-\delta_2,\delta_2)\to [1,2)$ increasing on $(-\delta_2,0)$, decreasing on $(0,\delta_2)$, equal to $1$ at $-\delta_2$ and $\delta_2$. We also impose that $f_q$ converges smoothly to the constant function equal to $1$. Each $f_q$ induces via $\phi$ a function on $N_\Gamma$ that we assume can be extended to a smooth function on $M$ equal to $1$ outside of $N_\Gamma$ (one might need to change a bit $f_q$ around the endpoints $-\delta_2$ and $\delta_2$). Let us still call $f_q$ this function defined on $M$, and consider the metrics $h^{(q)}:=f_q.g$. With respect to $h^{(q)}$, $\Gamma$ is still minimal, and any hypersurface $\phi(\Gamma\times\{s\})$ with $s\neq0$ has now non-zero mean curvature vector pointing away from $\Gamma$. That follows from the following general fact: let $S$ be a 2-sided embedded hypersurface in $(\interior(N),g)$ endowed with a choice of unit normal $\nu$ and let $\varphi :M\to \mathbb{R}$ be a smooth function. Consider the conformal change of metric $h:=\exp(2\varphi) g$. If $\mathbf{A}_g$ (resp. $\mathbf{A}_h$) denotes the second fundamental form of $S$ with respect to $\nu$ for $g$ (resp. $h$), one can check that
$$\mathbf{A}_h(a,b) = \exp(\varphi) (\mathbf{A}_g(a,b) + g(a,b) d\varphi(\nu)).$$
In particular if $S$ has vanishing mean curvature or if its mean curvature vector $\overrightarrow{H}$ satisfies $\langle \overrightarrow{H},\nu \rangle<0$, and if $d\varphi(\nu)>0$ then the mean curvature vector $\overrightarrow{H}_h$ of $S$ with respect to $h$ now satisfies in any case $\langle \overrightarrow{H}_h,\nu \rangle<0$.
To finish the proof of (2), we notice that if $f_q:(-\delta_2,\delta_2) \to [1,2)$ was chosen to have a strictly negative second derivative at $0$, then with respect to $h_q$, $\Gamma$ is an unstable minimal hypersurface.
\end{proof}

\section*{Appendix B: Local min-max constructions in the Almgren-Pitts' setting}

%(see also \cite[Theorem 10]{Antoinespheres})

We give a review of the basic definitions from geometric measure theory and some notions of the Almgren and Pitts' theory used in the paper, particularly the local $1$-parameter min-max theory in that setting (it is essentially a mixture of \cite{P} and \cite[Theorem 1.7]{MaNeindexbound}). For a complete presentation, we refer the reader to the book of Pitts \cite{P}, to Section 2 in \cite{MaNeinfinity} and \cite[Section 3]{MaNeindexbound}.

Let $N$ be a compact connected Riemannian $(n+1)$-manifold, assumed to be isometrically embedded in $\mathbb{R}^P$. We work with the space $\mathbf{I}_k(N;\mathbb{Z}_2)$ of $k$-dimensional flat chains with coefficients in $\mathbb{Z}_2$ and with support contained in $N$, the subspace $\mathcal{Z}_k(N;\mathbb{Z}_2) \subset \mathbf{I}_k(N;\mathbb{Z}_2)$ whose elements are boundaries, and with the space $\mathcal{V}_k(N)$ of the closure, in the weak topology, of the set of $k$-dimensional rectifiable varifolds in $\mathbb{R}^P$ with support in $N$.

An integral current $T\in \mathbf{I}_k(N;\mathbb{Z}_2)$ determines an integral varifold $|T|$ and a Radon measure $\|T\|$ (\cite[Chapter 2, 2.1, (18) (e)]{P}). If $V \in \mathcal{V}_k(N)$, denote by $\|V\|$ the associated Radon measure on $N$. Given an $(n+1)$-dimensional rectifiable set $U\subset N$, if the associated rectifiable current is an integral current in $\mathbf{I}_{n+1}(N;\mathbb{Z}_2)$, it will be written as $[|U|]$. To a rectifiable subset $R$ of $N$ corresponds an integral varifold called $|R|$. 
The support of a current or a measure is denoted by $\spt$. 
The notation $\mathbf{M}$ stands for the mass of an element in $\mathbf{I}_k(N;\mathbb{Z}_2)$. On $\mathbf{I}_k(N;\mathbb{Z}_2)$ there is also the flat norm $\mathcal{F}$ which induces the so-called flat topology. 
The space $\mathcal{V}_k(N)$ is endowed with the topology of the weak convergence of varifolds. The mass of a varifold is denoted by $\mathbf{M}$. The $\mathbf{F}$-metric was defined in \cite{P} and induces the varifold weak topology on any subset of $\mathcal{V}_k(N)$ with mass bounded by a constant.

Suppose that $\partial_0 N$ and $\partial_1 N$ are disjoint closed sets (which can be empty), $\partial_0 N \cup \partial_1 N = \partial N$, and $C_0$ (resp. $C_1$) is the cycle in $\mathcal{Z}_n(N;\mathbb{Z}_2)$ which is determined by $\partial_0 N$ (resp. $\partial_1 N$). Note that $C_0=C_1=0$ if $\partial N=\varnothing$. Let $\Phi: [0,1]\to \mathcal{Z}_n(N;\mathbf{F};\mathbb{Z}_2)$ be a continuous map so that $\Phi(0)=C_0$, $\Phi(1) =C_1$. Let $\Pi$ be the class of all continuous maps $\Phi':[0,1]\to \mathcal{Z}_n(N;\mathbf{F};\mathbb{Z}_2)$ homotopic to $\Phi$ in the flat toplogy, with endpoints fixed, i.e. there is a continous map $H:[0,1]\times[0,1] \to \mathcal{Z}_n(N;\mathcal{F};\mathbb{Z}_2)$ so that $H(.,0)=\Phi(.)$, $H(.,1)=\Phi'(.)$, $H(j,s)=C_j$ for $j=0,1$ and all $s\in[0,1]$. Let $\pi^\sharp(C_0,C_1)$  be the family of such homotopy classes $\Pi$ of maps starting at $C_0$ and ending at $C_1$.

In \cite{Alm1}, Almgren describes how to associate to a continuous map $\Phi: [0,1]\to \mathcal{Z}_n(N;\mathbf{F};\mathbb{Z}_2)$ an element of $\mathbf{I}_{n+1}(N,\mathbb{Z}_2)$. Let us explain this construction. There is a number $\mu>0$ such that if $T\in \mathbf{I}_n(N,\mathbb{Z}_2)$ has no boundary and $\mathcal{F}(T)\leq \mu$, then there is an $S_T\in \mathbf{I}_{n+1}(N,\mathbb{Z}_2)$ such that $\partial S_T =T$ and $$\mathbf{M}(S_T)=\mathcal{F}(T)=\inf\{\mathbf{M}(S') ; S'\in \mathbf{I}_{n+1}(N,\mathbb{Z}_2) \text{ and } \partial S' = T \}.$$
Such an $S_T$ is called an $\mathcal{F}$-isoperimetric choice for $T$. Now let $k_0$ be large enough so that for all $i=1,...,k_0$,
$\mathcal{F}(\Phi(\frac{i}{k_0})-\Phi(\frac{i-1}{k_0}))\leq \mu$.
With the previous notation for $\mathcal{F}$-isoperimetric choices, consider the $(n+1)$-dimensional integral current
\begin{equation} \label{Almgren}
\sum_{i=1}^{k_0} S_{\Phi(\frac{i}{k_0})-\Phi(\frac{i-1}{k_0})}.
\end{equation}
This current does not depend of $k_0$ provided it is sufficiently large. 
% attention bien vrifier
By the interpolation formula of \cite[Section 6]{Alm1}, this sum is also invariant by homotopies. Hence when $\Phi \in\Pi\in\pi^\sharp(C_0,C_1)$, the map which associates to $\Pi$ the $(n+1)$-dimensional current (\ref{Almgren}), defined with $\Phi$, is well defined. We call this map the Almgren map and we denote it by
$$\mathcal{A} :  \pi^\sharp(C_0,C_1) \to \mathbf{I}_{n+1}(N,\mathbb{Z}_2).$$

Given $\pi^{\sharp}(C_0,C_1)$, consider the function $\mathbf{L}: \pi^{\sharp}(C_0,C_1) \to [0,\infty]$ defined by 
$$\mathbf{L}(\Pi) = \inf\{\sup_{x\in[0,1]} \mathbf{M}(\Phi(x)) ; \Phi \in \Pi\}.$$
A sequence $\{\Phi_i\}_i\subset  \Pi$ is a min-max sequence if
$$\limsup_{i\to \infty} (\sup_{x\in[0,1]} \mathbf{M}(\Phi(x))) = \mathbf{L}(\Pi).$$
The image set of $\{\Phi_i\}_i$ is 
\begin{align*}
\mathbf{\Lambda}(\{\Phi_i\}_i) = & \{ V\in \mathcal{V}_n(N) ; \exists \{i_j\}\to \infty, x_{i_j}\in[0,1],\\
& \text{ such that  } \lim_{j\to\infty}\mathbf{F}(|\Phi_{i_j}(x_{i_j})| ,V) =0 \}.
\end{align*}
If $\{\Phi_i\}_i\subset  \Pi$ is a min-max sequence, define the critical set $\mathbf{C}(S) \subset \mathcal{V}_n(N)$ of $\{\Phi_i\}_i$ as
$$
\mathbf{C}(\{\Phi_i\}_i) = \{V \in \mathbf{\Lambda}(\{\Phi_i\}_i);  \|V\|(N)= \mathbf{L}(\Pi).
$$
A min-max sequence $\{\Phi_i\}_i\subset \Pi$ such that every element of $\mathbf{C}(\{\Phi_i\}_i) $ is stationary is called pulled-tight.

We finally define the width $W$ of $(N,g)$ to be 
\begin{equation} \label{wiidth}
W:=\inf_{\mathcal{A}(\Pi)=[|N|]} \mathbf{L}(\Pi),
\end{equation}
the infimum being taken over all the possible following choices: we start with a partition $\partial N = X_1\cup X_2$ ($X_1$, $X_2$ are closed), $C_i$ is the cycle in $\mathcal{Z}_n(N;\mathbb{Z}_2)$ determined by $X_i$ ($i=1,2$), and $\Pi\in \pi_1^{\sharp}(C_0,C_1)$ satisfies 
$$\mathcal{A}(\Pi)=[|N|].$$
This width should not be confused with the first width $\omega_1(N,g)$ which is defined with relative cycles (see \cite{Gromovnonlinearspectra, Guth, MaNeinfinity, LioMaNe}).

A metric $g$ is said to be bumpy if no smooth immersed closed minimal hypersurface has a non-trivial Jacobi vector field. White showed that bumpy metrics are generic in the Baire sense \cite{Whitebumpy,Whitebumpy2}. The following theorem is essentially a consequence of the index bound of Marques-Neves \cite{MaNeindexbound}.

\begin{theo} \label{discreteminmax}

Let $(N^{n+1},g)$ be a compact manifold with boundary endowed with a bumpy metric $g$, with $2\leq n \leq 6$. Suppose that the boundary $\partial N$ is a strictly stable minimal hypersurface. Then there exists a stationary integral varifold $V$ whose support is a smooth embedded minimal hypersurface $\Sigma \subset N$ of index bounded by one, such that
$$\|V\|(N)= W.$$ 
Moreover one of the components of $\Sigma$ is contained in the interior $\interior(N) $. 
\end{theo}

\begin{proof}
We consider $(N,g)$ as isometrically embedded inside $(\tilde{N}, \tilde{g})$ such that $\partial \tilde{N}$ is strictly mean convex (the mean curvature vector points inwards), and $\tilde{N}\backslash N$ is foliated by strictly mean convex hypersurfaces, so that any closed minimal hypersurface embedded in $\tilde{N}$ is embedded in $N$.

Suppose that $\partial_0 N,\partial_1 N \subset N$ are disjoint closed sets, $\partial_0 N \cup \partial_1 N = \partial N$, and $C_0$ (resp. $C_1$) is the cycle in $\mathcal{Z}_n(N;\mathbb{Z}_2)$ which is determined by $\partial_0 N$ (resp. $\partial_1 N$). 

By \cite{MorganRos}, any homotopy class $\Pi \in \pi^\sharp (C_0,C_1)$ satisfies
$$ \mathbf{L}(\Pi) >\max (\mathbf{M}(C_0) , \mathbf{M}(C_1)\}.$$

Hence combining \cite[Theorem 1.7]{MaNeindexbound} and the arguments of \cite[Theorem 2.1]{MaNe} (see also \cite[Theorem 2.7]{Zhou}),
we get $V$ as in the statement but with $\|V\|(N)= \mathbf{L}(\Pi)$, and a connected component of $\spt(V)$ which either is 2-sided and has index one, or is 1-sided. This component has to be embedded in the interior of $N$. 

If we apply the previous discussion to a sequence of homotopy classes $\Pi_k$ such that $\lim_{k\to \infty} \mathbf{L}(\Pi_k)=W$, we get a sequence of varifolds $V_k$ which subsequently converge to $V$ (\cite{Sharp}) as in the statement. A connected component of $\spt(V)$ is in $\interior(N)$ because by the maximum principle, a sequence of minimal hypersurfaces in $\interior(N)$ cannot converge in the Gromov-Hausdorff topology to some components of $\partial N$, which is strictly stable.

\end{proof}

\section*{Appendix C: Facts about mean curvature flow}

We find it convenient to use the level set flow formulation of the mean curvature flow. Under the level set flow, also called ``biggest flow'', introduced by Chen-Giga-Goto \cite{CGG} and Evans-Spruck \cite{ES}, any compact set $K$ in $\mathbb{R}^{n+1}$ canonically generates a one-parameter family of compact sets $\{K_t\}_{t\geq0}$ with $K_0=K$. The set $K$ is said to be \emph{weakly} (resp. \emph{strictly}) \emph{mean convex} if $K_t\subset K$ (resp. $K_t\subset \interior(K)$) for all $t>0$. If a weakly mean convex set $K$ has $C^{1,1}$ boundary and one component $S$ of $\partial K$ is such that after any small positive time, it completely enters inside the interior of $K$, we say that $S$ is \emph{strictly mean convex}. Recall that $C^{1,1}$ hypersurfaces become instantaneously smooth under the level set flow in small times \cite[Section 3]{Whitetopology}. Moreover such a boundary component $S$ is either smooth minimal, or strictly mean convex by \cite[Theorem 8.2]{ES}.

When $X:=\partial K$ happens to be smooth, then $K$ is weakly (resp. strictly) mean convex in the previous sense if and only if the mean curvature of $X$ is nonnegative (resp. the mean curvature is nonnegative and positive somewhere). See for instance \cite{WhiteregMCF,Whitetopology}.

When $K$ is strictly mean convex, the level set flow is given by the level sets of a Lipschitz function $u$ satisfying in the viscosity sense
$$-1=|\nabla u|\divergence(\frac{\nabla u}{|\nabla u|}),$$
and $\partial K_t = \{x\in\mathbb{R}^{n+1} ; u(x)=t\}$.
In that case, there is a unique Lipschitz function $u$ giving a viscosity solution of the level set flow \cite[Theorem 7.4]{ES}, \cite{CGG} (see also \cite{CMarrivaltime}). The flow is non-fattening (the level sets have no interior) \cite[Corollary 3.3]{WhiteregMCF}. The singular set at each time $t>0$ has Hausdorff dimension at most $n-1$, in particular at all positive time $t$, $\partial K_t$ is a closed hypersurface smooth outside a subset of dimension $n-1$ \cite{WhiteregMCF} (see also \cite{CMsingularsetMCF}). Moreover for all $T>0$, $\{\partial K_t\}_{t\in [0,T)}$ forms a possibly singular mean convex foliation of $K\backslash K_T$ \cite{WhiteregMCF}. By monotonicity, the $n$-dimensional Hausdorff measure (that we usually denote by $\Vol_n$) of the set $\partial K_t$ is nonincreasing in $t$ \cite[Corollary 3.6]{WhiteregMCF}.

The previous facts extend naturally to closed ambient manifolds $(N^{n+1},g)$, $2\leq n\leq 6$, \cite{IlmanenmanifoldMCF,WhiteregMCF}. The long term behavior can be different: in general a closed mean convex non-minimal $n$-dimensional hypersurface $X_0\subset N$ evolving under the level set flow will sweep out an $(n+1)$-dimensional manifold with boundary $Y\subset N$, and $\partial Y = X_0\cup X_\infty$, where $X_\infty$ is a (possibly empty) closed smooth embedded stable minimal hypersurface. If $\{K_t\}_{t\in [0,\infty)}$ denotes the corresponding mean convex level set flow such that $K_0= Y$, each $K_t$ can be considered as an integral current in $\mathbf{I}_{n+1}(N;\mathbb{Z}_2)$ and $\{\partial K_t\}$ can be considered as a family of cycles in $\mathcal{Z}_n(N;\mathbb{Z}_2)$ continuous in the flat topology: this fact can be checked using the regularity property \cite[Theorem 5.1]{WhiteregMCF} and the coarea formula. Define $X_t$ to be $\partial K_t\backslash X_\infty$; the support of $\partial K_t$ as a cycle in $\mathcal{Z}_n(N;\mathbb{Z}_2)$ is contained in the set $X_t \cup X_\infty$.
%NO and so $\{X_t\}_{t\in[0,\infty)}$ yields a sweepout of $Y$ in the sense of Almgren-Pitts theory (see Appendix B).

In the situation above, the convergence of $X_t $ to $X_\infty$ is smooth and one-sheeted (resp. two-sheeted) for 2-sided (resp. 1-sided) components of $X_\infty$ \cite[Section 11]{WhiteregMCF}. By the avoidance principle (or maximum principle, see \cite[Lemma 6.3]{IlmanenmanifoldMCF}), if $K$ is a mean convex closed set generating $\{K_t\}_{t\geq0}$ then any strictly mean concave closed set $W\subset K$ satisfies $$W\subset \interior(\cap_{t\geq0}K_t).$$

\bibliographystyle{plain}
\bibliography{biblio19_08_30}

\end{document}